\documentclass[11pt]{article}

\usepackage{amssymb}
\usepackage{amsmath}
\usepackage{graphicx}
\usepackage{caption}
\usepackage{subcaption}
\usepackage{caption}
\usepackage{float}
\usepackage{subcaption}
\usepackage{dcolumn}
\usepackage{amsthm}
\usepackage{multirow}
\usepackage{enumitem}
\usepackage{amsfonts}
\usepackage{color}
\usepackage{cite}

\usepackage[hscale=0.8,vscale=0.85]{geometry}

\usepackage{pgf}
\usepackage{tikz}
\usetikzlibrary{decorations.pathmorphing}
\usetikzlibrary{arrows.meta,decorations.pathmorphing,backgrounds,positioning,fit}
\usetikzlibrary{shapes,snakes}
\usetikzlibrary{calc}
\tikzset{snake it/.style={decorate, decoration=snake}}

\numberwithin{equation}{section}

\newcommand{\ddt}[1]{\frac{\partial}{\partial t}#1}
\newcommand{\inp}[2][]{\left(#1, #2\right)}
\newcommand{\gnp}[2][]{\langle#1, #2\rangle}

\def\As{A_{\sigma\sigma}}
\def\Asu{A_{\sigma u}}
\def\Asg{A_{\sigma\gamma}}
\def\Asp{A_{\sigma p}}
\def\Az{A_{zz}}
\def\Azp{A_{zp}}
\def\Azz{A_{zz}}
\def\Ap{A_{pp}}

\def\Ausu{A_{u\sigma u}}
\def\Ausg{A_{u\sigma \g}}
\def\Ausp{A_{u\sigma p}}
\def\Agsg{A_{\g\sigma \g}}
\def\Agsp{A_{\g\sigma p}}
\def\Apszp{A_{p\sigma z p}}

\def\Tc{\mathcal{T}}

\def\Qc{\mathcal{Q}}

\def\BDM{\mathcal{BDM}}
\def\RT{\mathcal{RT}}

\def\Pc{\mathcal{P}}

\def\X{\mathbb{X}}

\def\Q{\mathbb{Q}}
\def\R{\mathbb{R}}
\def\M{\mathbb{M}}
\def\S{\mathbb{S}}
\def\N{\mathbb{N}}

\def\H{\mathbb{H}}

\def\a{\alpha}
\def\s{\sigma}
\def\t{\tau}
\def\g{\gamma}

\def\l{\lambda}

\def\z{\zeta}
\def\tet{\theta}

\def\r{\mathbf{r}}

\def\As{A_{\s\s}}
\def\Au{A_{\s u}}
\def\Ag{A_{\s\g}}

\def\skew{\operatorname{Skew}}

\def\tr{\operatorname{tr}}

\def\dvr{\mathrm{div} \,}
\def\curl{\mathrm{curl} \,}
\def\Hdiv{H(\mathrm{div}; \O)}

\def\dt{\partial_t}

\def\Eh{\hat{E}}

\def\Qh{\hat{\Qc}}

\def\rh{\hat{\r}}
\def\eh{\hat{e}}

\def\nh{\hat{n}}

\def\xh{\hat{x}}
\def\yh{\hat{y}}

\def\th{\hat{\t}}

\def\Xh{\hat{\X}}
\def\Vh{\hat{V}}
\def\Wh{\hat{W}}

\def\Zh{\hat{Z}}

\def\Pih{\hat{\Pi}}

\def\O{\Omega}
\def\Gn{\Gamma_N}
\def\Gd{\Gamma_D}
\def\dO{{\partial\Omega}}

\def\l2o{{L^{2}(\Omega)}}

\def\api{{\alpha p_h I}}

\def\K{{K^{-1}}}

\newtheorem{remark}{Remark}[section]
\newtheorem{lemma}{Lemma}[section]
\newtheorem{corollary}{Corollary}[section]
\newtheorem{theorem}{Theorem}[section]
\newtheorem{proposition}{Proposition}[section]

\begin{document}
	\title{A coupled multipoint stress - multipoint flux mixed finite element method for 
the Biot system of poroelasticity}

\author{Ilona Ambartsumyan\thanks{Department of Mathematics, University of
		Pittsburgh, Pittsburgh, PA 15260, USA;~{\tt \{ila6@pitt.edu, elk58@pitt.edu, yotov@math.pitt.edu\}}. Partially supported by DOE grant DE-FG02-04ER25618 and NSF grant DMS 1818775.} \thanks{Oden Institute for Computational Engineering and Sciences, The University of Texas at Austin, Austin, TX 78712, USA;
{\tt \{ailona@austin.utexas.edu, ekhattatov@austin.utexas.edu\}}.}~\and
  Eldar Khattatov\footnotemark[1] \footnotemark[2]
  \and Ivan Yotov\footnotemark[1]~}

\date{\today}
\maketitle
\begin{abstract}
We present a mixed finite element method for a five-field formulation
of the Biot system of poroelasticity that reduces to a cell-centered
pressure-displacement system on simplicial and quadrilateral grids. A
mixed stress-displacement-rotation formulation for elasticity with
weak stress symmetry is coupled with a mixed velocity-pressure Darcy
formulation. The spatial discretization is based on combining the
multipoint stress mixed finite element (MSMFE) method for elasticity
and the multipoint flux mixed finite element (MFMFE) method for Darcy
flow.  It uses the lowest order Brezzi-Douglas-Marini mixed finite
element spaces for the poroelastic stress and Darcy velocity,
piecewise constant displacement and pressure, and continuous
piecewise linear or bilinear rotation. A vertex quadrature rule is
applied to the velocity, stress, and stress-rotation bilinear forms,
which block-diagonalizes the corresponding matrices and allows for
local velocity, stress, and rotation elimination. This leads to a
cell-centered positive-definite system for pressure and displacement
at each time step. We perform error analysis for the semidiscrete and
fully discrete formulations, establishing first order convergence for
all variables in their natural norms. The numerical tests confirm the
theoretical convergence rates and illustrate the locking-free
property of the method.
\end{abstract}

\section{Introduction}
The Biot system of poroelasticity \cite{biot1941general,
  showalter2000diffusion} models fluid flow within deformable porous
media. It has been extensively studied in the literature due to its
wide range of applications. Examples include geosciences, such as
groundwater cleanup, hydraulic fracturing, and carbon sequestration,
as well as biomedical applications, such as modeling of arterial flows
and organ tissue. The system consists of an equilibrium equation for
the solid and a mass balance equation for the fluid.  This is a fully
coupled system, as the fluid pressure contributes to the solid stress,
while the divergence of the solid displacement affects the fluid
content. There is a large literature on the the numerical solution of
the Biot system. Schemes for the two-field displacement--pressure
formulation include finite difference \cite{Gaspar-FD-Biot}, finite
volume \cite{Jan-SINUM-Biot}, and finite element methods
\cite{murad1992improved,Zik-MINI}. The finite element methods are
either based on inf-sup stable pairs \cite{murad1992improved,Zik-MINI}
or employ a suitable stabilization to avoid pressure oscillations
\cite{Zik-MINI}. The three-field displacement--pressure--Darcy
velocity formulation has also been studied extensively. It has the
advantage that stable mixed finite element spaces for the Darcy
velocity and the pressure can be utilized, resulting in accurate fluid
velocity and local mass conservation. Various choices of displacement
discretizations have been used in the three-field formulation,
including continuous, \cite{phillips2007coupling1,
  phillips2007coupling2,Zik-stab,Yi-Biot-locking},
nonconforming\cite{Zik-nonconf,Lee-Biot-three-field,Yi-Biot-nonconf},
and discontinuous elements \cite{phillips-DG,Liu-thesis}. The last two
choices provide locking-free approximations.  Alternatively,
stabilized continuous displacement elements can be used to suppress
pressure oscillations \cite{Zik-stab,Yi-Biot-locking}. Locking-free
discretizations for a different three-field
displacement--pressure--total pressure formulation are developed in
\cite{Lee-Mardal-Winther,ORB}.
A least squares method based on a stress--displacement--velocity--pressure
formulation is developed in \cite{korsawe2005least}.
More recently, fully-mixed formulations
of the Biot system have been studied
\cite{Yi-Biot-mixed,Lee-Biot-five-field}.  In \cite{Yi-Biot-mixed}, a
stress--displacement mixed elasticity formulation is coupled with a
velocity-pressure mixed Darcy model. This approach is extended in
\cite{Lee-Biot-five-field}, where a weakly symmetric
stress--displacement--rotation elasticity formulation is considered.

In this paper we develop a new fully-mixed finite element method for
the quasistatic Biot system of poroelasticity. The advantages of
fully-mixed approximations include locking-free behavior, robustness
with respect to the physical parameters, local mass and momentum
conservation, and accurate stress and velocity approximations with
continuous normal components across element edges or faces. They can
also handle discontinuous full tensor permeabilities and Lam\'{e}
coefficients that are often encountered in modeling subsurface flows.
A disadvantage of fully-mixed methods is that they result in large
algebraic systems of saddle point type at each time step. In
particular, the methods developed in \cite{Yi-Biot-mixed} and
\cite{Lee-Biot-five-field} involve four-field and five-field
formulations, respectively. Our goal is to develop a fully-mixed
method that can be reduced to a positive definite cell-centered
displacement--pressure system. As a result, the method inherits all
the advantages of fully-mixed finite element methods, while having a
significantly reduced computational cost. In fact, the number of
unknowns in the reduced algebraic system is smaller than in any of the
aforementioned finite element methods. It is comparable to the cost of
the finite volume method developed in \cite{Jan-SINUM-Biot}.

Our approach is based on the five-field formulation proposed in
\cite{Lee-Biot-five-field}. We couple the recently developed
multipoint stress mixed finite element (MSMFE) method for elasticity
\cite{msmfe1,msmfe2} with weak stress symmetry and the multipoint flux
mixed finite element (MFMFE) method for Darcy flow
\cite{ingram2010multipoint,wheeler2006multipoint,wheeler2012multipoint}.
The MFMFE method is related to the the finite volume multipoint flux
approximation (MPFA) method
\cite{aavatsmark1998discretization,edwards1998finite}.  The MFMFE
method provides a variational formulation for the MPFA method, which
allows for utilizing mixed finite element tools for its analysis. It
uses the lowest order Brezzi-Douglas-Marini $\BDM_1$
\cite{brezzi1985two,Nedelec86} spaces for the Darcy velocity and piecewise
constant pressure. The vertex quadrature rule for the velocity
bilinear form gives a block-diagonal mass matrix with blocks
associated with the mesh vertices and allows for local velocity
elimination, resulting in a cell-centered pressure system. The MFMFE
method is analyzed on simplices and smooth quadrilateral and
hexahedral grids, i.e., with elements that are $O(h^2)$-perturbations of
parallelograms, in
\cite{ingram2010multipoint,wheeler2006multipoint}. A similar approach
on simplices is proposed in \cite{Brezzi.F;Fortin.M;Marini.L2006}.  A
non-symmetric version of the MFMFE method for general quadrilateral
and hexahedral grids is developed in \cite{wheeler2012multipoint}; see
also an alternative formulation based on a broken Raviart-Thomas
velocity space in \cite{klausen2006robust}. The MSMFE method for
elasticity with weak stress symmetry was recently developed in
\cite{msmfe1} on simplices and in \cite{msmfe2} on smooth
quadrilateral grids.  It uses $\BDM_1$ elements for the stress,
piecewise constant displacement, and continuous piecewise linear 
rotation. The vertex quadrature
rule is applied for the stress bilinear form, as well as the two
stress--rotation bilinear forms. This allows for local stress and
rotation elimination around the mesh vertices, resulting in a
cell-centered displacement system. The development of the MSMFE method
was motivated by the finite volume multipoint stress approximation
(MPSA) method for elasticity introduced in \cite{Jan-IJNME} and
analyzed in \cite{nordbotten2015convergence} as a discontinuous
Galerkin (DG) method. A weak symmetry MPSA method, which is more
closely related to the MSMFE method has been developed in
\cite{keilegavlen2017finite}.

In this work we develop and analyze a coupled MSMFE--MFMFE method for
the Biot system of poroelasticity. Starting with the five-field
stress--displacement--rotation--velocity--pressure formulation from
\cite{Lee-Biot-five-field}, we employ the vertex quadrature rule for
the stress, stress--rotation, and velocity bilinear forms. Since the
stress, rotation, and velocity degrees of freedom can be associated
with the mesh vertices, the quadrature rule localizes their
interaction around the vertices, resulting in block-diagonal
matrices. The stress and velocity, and consequently the rotation, can
then be locally eliminated by solving small vertex-based linear
systems. This procedure reduces the five-field saddle point system to
a cell-centered displacement-pressure system. The elimination
procedure resembles the approach in the finite volume method for the
Biot system developed in \cite{Jan-SINUM-Biot}, which couples the MPSA
and MPFA methods, although the method there is not based on weak
symmetry and does not explicitly involve rotations. We also note that
in our method we utilize a symmetric quadrature rule, as in the
symmetric MFMFE method
\cite{ingram2010multipoint,wheeler2006multipoint} and the MSMFE method
\cite{msmfe1,msmfe2}. As the individual methods, our coupled method is
suitable for simplicial grids in two and three dimensions and
quadrilateral grids with elements that are $O(h^2)$-perturbations of
parallelograms.  While a non-symmetric MFMFE method on general
quadrilaterals and hexahedra is available
\cite{wheeler2012multipoint}, such non-symmetric MSMFE method for
elasticity has not yet been developed.

We perform solvability, stability, and error analysis for the
semidiscrete continuous-in-time and the fully discrete methods. The
well-posedness of the semidiscrete formulation utilizes techniques
from degenerate evolution operators
\cite{showalter2013monotone,Showalter-SIAMMA}. For this purpose, we
differentiate in time the constitutive elasticity equation and
introduce as new variables the time derivatives of the displacement
and the rotation. Stability is obtained for all variables in their
natural spatial norms in both $L^2(0,T)$ and $L^\infty(0,T)$. In order
to obtain control of the divergence of the Darcy velocity, a bound on
the time derivative of the pressure is first derived, using time
differentiation of the rest of the equations. First order spatial
convergence is proven for all variables by combining stability
arguments with bounds on the quadrature and approximation errors. It
is important to note that the stability and convergence bounds are
independent of the storativity coefficient $c_0$ and are valid even
for $c_0 = 0$. As the regime of small $c_0$ results in locking effects
\cite{phillips2009overcoming}, our theory confirms the locking-free
property of the method. We also present the fully-discrete scheme,
based on backward Euler time discretization. The analysis of the
fully-discrete scheme uses the framework developed for the
semidiscrete formulation, combined with standard tools for treating
the discrete time derivatives. 

The rest of the paper is organized as follows. The Biot system and its
fully mixed five-field weak formulation are presented in Section~2.
The semidiscrete MSMFE--MFMFE method is developed in Section~3.  Its
solvability and stability are established in Section~4 and Section~5,
respectively. The error analysis for the semidiscrete method is
carried out in Section~6. Section~7 is devoted to the fully-discrete
MSMFE--MFMFE method, where in addition to its analysis, the procedure
for reducing the algebraic system to a cell-centered
displacement--pressure system is presented. It is further shown that
the resulting system is positive definite.  Numerical results that
confirm the theoretical convergence rates and illustrate the
robustness with respect to $c_0$ and the locking-free behavior of the
method are presented in Section~8.

\section{Model problem and a fully mixed weak formulation}\label{sec:model}
In this section we describe the poroelasticity system and its fully
mixed formulation based on a weak stress symmetry, Let $\O$ be a
simply connected bounded domain of $\R^d,\, d=2,3$, occupied by a
poroelastic media saturated with fluid.  Let $\M$, $\S$, and $\N$ be
the spaces of real $d\times d$ matrices, symmetric matrices, and
skew-symmetric matrices, respectively. The divergence operator $\dvr:
\R^d \to \R$ is the usual divergence for vector fields. It also acts
on matrix fields, $\dvr: \M \to \R^d$ by applying the divergence
row-wise. We will also utilize the operator $\curl$
  acting on scalar fields in two dimensions, $\curl: \R \to \R^2$, defined as
 $ \curl{\phi} = (\partial_2 \phi, -\partial_1 \phi)$.

The stress-strain constitutive relationship for the poroelastic body is
\begin{equation}\label{elast-stress}
A\sigma_e = \epsilon(u),
\end{equation}
where at each point $x \in \O$, $A(x): \S \to \S$, extendible to
$A(x): \M \to \M$, is a symmetric,
bounded and uniformly positive definite linear operator representing the
compliance tensor, $\sigma_e$ is the elastic stress, $u$ is the solid
displacement, and $\epsilon(u) = \frac12(\nabla u + \nabla u^T)$.
In the case of a homogeneous and isotropic body,
$$ 
A\sigma = \frac{1}{2\mu} \left( \sigma - \frac{\lambda}{2\mu
  + d\lambda}\operatorname{tr}(\sigma)I \right), 
$$
where $I$ is the $d \times d$ identity matrix and $\mu > 0, \lambda \ge
0$ are the Lam\'{e} coefficients. In this case the elastic stress is
$\sigma_e = 2\mu\epsilon(u) + \lambda\dvr u\,I$. The poroelastic stress,
which includes the effect of the fluid pressure $p$, is given as
\begin{align}\label{poroelast-stress}
\sigma = \sigma_e - \alpha pI,
\end{align}
where $0 < \alpha \le 1$ is the Biot-Willis constant.

Given a vector field $f$ representing the body forces and a 
source term $q$, the quasi-static Biot system \cite{biot1941general}
that governs the fluid flow within the poroelastic media is as follows:
	\begin{align}
	-\dvr\s &= f \quad \text{in } \O \times (0,T], \label{biot-1} \\
	\K z + \nabla p &= 0 \quad \text{in } \O \times (0,T], \label{biot-2} \\ 
	\ddt(c_0 p + \a \, \dvr u) + \dvr z &= q \quad \text{in } \O \times (0,T], \label{biot-3}  
	\end{align}
where $z$ is the Darcy velocity,
$c_0 \ge 0$ is a mass storativity coefficient, and $K$ is a
symmetric and positive 
definite tensor representing the permeability of the porous media divided
by the fluid viscosity. The system is closed with the boundary conditions
	\begin{align}
	u &= g_u \quad \text{on } \Gd^{displ} \times (0,T], \qquad \s\, n = 0 \quad\text{on } \Gn^{stress} \times (0,T], \label{biot-bc-1}\\
	p &= g_p \quad \text{on } \Gd^{pres} \times (0,T], \qquad z \cdot n = 0 \quad\text{on } \Gn^{vel} \times (0,T] \label{biot-bc-2},
	\end{align}
and the initial condition $p(x,0) = p_0(x)$ in $\O$, where
$\Gd^{displ} \cup \Gn^{stress}= \Gd^{pres} \cup \Gn^{vel} =
\partial\O$ and $n$ is the outward unit normal vector field on
$\partial\Omega$.  To avoid technical issues due to non-uniqueness
in the case of pure Neumann boundary conditions,
we assume that $|\Gd^{*}| > 0$, for
$* = \{displ, \,pres\}$. We note that equations \eqref{biot-1} and
\eqref{biot-2}, which do not include time derivatives, are assumed to
hold at $t = 0$. This is used to construct compatible initial data
for the rest of the variables. The well posedness of the above system
has been studied in \cite{showalter2000diffusion}.

Throughout the paper, $C$ denotes a generic positive constant that is
independent of the discretization parameter $h$. We will also use the
following standard notation. For a domain $G \subset \R^d$, the
$L^2(G)$ inner product and norm for scalar, vector, or tensor valued functions
are denoted $\inp[\cdot]{\cdot}_G$ and $\|\cdot\|_G$,
respectively. The norms and seminorms of the Sobolev spaces
$W^{k,p}(G),\, k \in \R, p>0$ are denoted by $\| \cdot \|_{k,p,G}$ and
$| \cdot |_{k,p,G}$, respectively. The norms and seminorms of the
Hilbert spaces $H^k(G)$ are denoted by $\|\cdot\|_{k,G}$ and $| \cdot
|_{k,G}$, respectively. We omit $G$ in the subscript if $G = \O$. For
a section of the domain or element boundary $S \subset \R^{d-1}$ we
write $\gnp[\cdot]{\cdot}_S$ and $\|\cdot\|_S$ for the $L^2(S)$ inner
product (or duality pairing) and norm, respectively. We will also use
the spaces
\begin{align*}
	&H(\dvr; \O) = \{v \in L^2(\O, \R^d) : \dvr v \in L^2(\O)\}, \\
	&H(\dvr; \O, \M) = \{\t \in L^2(\O, \M) : \dvr \t \in L^2(\O,\R^d)\},
\end{align*}    
equipped with the norm
$$
\|\t\|_{\dvr} = \left( \|\t\|^2 + \|\dvr\t\|^2 \right)^{1/2}.
$$

We next present the mixed weak formulation, which has been proposed in 
\cite{Lee-Biot-five-field}. Using \eqref{elast-stress} and \eqref{poroelast-stress}, 
we have
$$
\dvr u = \tr(\epsilon(u)) = \tr(A\sigma_e) = \tr A(\s + \alpha p I),
$$
which can be substituted in \eqref{biot-3} to give
$$
\dt(c_0 p + \a \tr A(\s + \alpha p I)) + \dvr z = q,
$$
where $\dt$ is a short notation for $\ddt$.
In the weakly symmetric stress 
formulation, we allow for $\sigma$ to be non-symmetric and 
introduce the Lagrange multiplier $\g = \skew(\nabla u)$,
$\skew(\tau) = \frac12(\tau - \tau^T)$, from the space of
skew-symmetric matrices. The constitutive equation \eqref{elast-stress}
can be rewritten as
$$
A(\sigma + \alpha p I) = \nabla u - \gamma.
$$
The mixed weak formulation of the Biot problem 
reads: find $(\sigma,u,\gamma,z,p): [0,T] \mapsto \X \times V \times \Q \times Z \times W$
such that $p(0) = p_0$ and, for a.e. $t \in (0,T)$,
	\begin{align}
	  &\inp[A(\sigma + \a p I)]{\tau} + \inp[u]{\dvr{\tau}}
          + \inp[\gamma]{\tau} = \gnp[g_u]{\t\,n}_{\Gamma_D^{displ}}, &&
          \forall \tau \in \X, \label{eq:cts1}\\
	  &\inp[\dvr{\sigma}]{v} = - \inp[f]{v}, && \forall v \in V,
          \label{eq:cts2}\\
	&\inp[\sigma]{\xi}  = 0, && \forall \xi \in \Q, \label{eq:cts3}\\
	  &\inp[\K z]{\zeta} - \inp[p]{\dvr{\zeta}}  =
          -\gnp[g_p]{\zeta\cdot n}_{\Gd^{pres}}, && \forall \zeta \in Z,
          \label{eq:cts4}\\
	&\inp[c_0\dt{p}]{w} + \a\inp[\dt A(\sigma + \a p I)]{wI} 
+ \inp[\dvr{z}]{w}   = \inp[q]{w}, && \forall w \in W, \label{eq:cts5}
	\end{align}
        where we have used the identity $(\tr A \tau, w) = (A\tau, w I)$ and
        the functional spaces are defined as
\begin{align*}
  &\X = \big\{ \t\in H(\dvr;\Omega,\M) :
  \t\,n = 0 \text{ on } \Gn^{stress}  \big\}, \quad V = L^2(\Omega, \R^d),
  \quad \Q = L^2(\Omega, \N), \\
  &Z = \big\{ \zeta \in H(\dvr;\Omega,\R^d) :
  \zeta\cdot n = 0 \text{ on } \Gn^{vel}  \big\}, \quad W = L^2(\Omega).
\end{align*}
We refer the reader to \cite{showalter2000diffusion} for the analysis
of the well-posedness of a related displacement-pressure weak formulation.
In Section~\ref{sec:well-posed} we establish existence, uniqueness, and stability
for the semidiscrete continuous-in-time approximation of
\eqref{eq:cts1}--\eqref{eq:cts5}. The arguments there also apply to
the weak formulation \eqref{eq:cts1}--\eqref{eq:cts5} itself. We make a remark here on
the initial data $p_0(x)$. In particular, we assume that
\begin{equation}\label{init-cond}
p_0 \in H^1(\Omega), \quad p_0(x) = g_p(x,0) \mbox{ on } \Gd^{pres},  \quad \mbox{and}
\quad K \nabla p_0 \in Z.
\end{equation}
A similar  assumption is also made in \cite{showalter2000diffusion}. In our case, we can
set $z_0 = -K \nabla p_0 \in Z$ and show that it satisfies \eqref{eq:cts4}. We can also
determine $\s_0$, $u_0$, and $\g_0$ by solving the elasticity problem
\eqref{eq:cts1}--\eqref{eq:cts3} with $p_0$ given as data.
We refer to the initial data obtained by this procedure
as compatible initial data. It is needed for the well posedness of the
\eqref{eq:cts1}--\eqref{eq:cts5}, as we will discuss in Section~\ref{sec:well-posed}.

\section{Mixed finite element discretization}
We begin with the discretization of the fully mixed weak formulation
of the poroelasticity system \eqref{eq:cts1}--\eqref{eq:cts5}, based
on mixed finite element methods for elasticity and Darcy flow.  We
then present the multipoint stress - multipoint flux mixed finite
element method, which employs the vertex quadrature rule for the
stress, rotation, and velocity bilinear forms and can be reduced to a
positive definite cell centered system for displacement
and pressure only.

\subsection{Mixed finite element spaces}
        
We next present the MFE discretization of \eqref{eq:cts1}--\eqref{eq:cts5}.
For simplicity, assume that $\O$ is a polygonal domain.  Let $\Tc_h$
be a shape-regular and quasi-uniform \cite{ciarlet2002finite} finite
element partition of $\O$, consisting of triangles and/or 
quadrilaterals in two dimensions and tetrahedra in three
dimensions. Let $h = \max_{E\in\Tc_h}\operatorname{diam}(E)$.  For any
element $E \in \Tc_h$ there exists a bijection mapping $F_E: \Eh \to
E$, where $\Eh$ is a reference element. We denote the Jacobian matrix
by $DF_E$ and let $J_E = \left|\operatorname{det}
(DF_E)\right|$. We note that the mapping is affine with constant $DF_E$ in
the case of simplicial elements and bilinear with linear $DF_E$ in the
case of quadrilaterals. The shape-regularity and quasiuniformity of the
grids imply that 
\begin{equation}\label{mapping-bounds}
  \| DF_E \|_{0,\infty, \Eh} \sim h, \quad \| J_E \|_{0,\infty, \Eh} \sim h^d
  \quad \forall E \in \Tc_h.
\end{equation}
	
Let $\X_h \times V_h \times \Q_h$ be the triple $\left(\BDM_1\right)^d
\times \left(\Pc_0\right)^d \times \left(\Pc_1^{cts}\right)^{d\times
  d, skew}$ on simplicial elements or $\left(\BDM_1\right)^2 \times
\left(\Qc_0\right)^2 \times \left(\Qc_1^{cts}\right)^{2\times 2,
  skew}$ on quadrilaterals, where $\Pc_k$ denotes the space of polynomials of
total degree $k$ and $\Qc_k$ denotes the space of polynomials of
degree $k$ in each variable. This triple has been shown to
be inf-sup stable for mixed elasticity with weak stress symmetry in
\cite{brezzi2008mixed,BBF-reduced,FarFor} on simplices, in \cite{lee2016towards} on
rectangles, and in \cite{msmfe2} on quadrilaterals; see also related spaces
with constant rotations on simplices \cite{arnold2007mixed} and
quadrilaterals \cite{arnold2015mixed}.
For the Darcy flow discretization we consider $Z_h\times W_h$
to be the lowest order $\BDM_1 \times \Pc_0$ MFE spaces
\cite{brezzi1985two,Nedelec86,BBF-book}.
On the reference simplex, these spaces are defined as 
\begin{align}
  &\Xh(\Eh) = \left(\Pc_1(\Eh)^d\right)^d, &&\quad \Vh(\Eh) = \Pc_0(\Eh)^d,
  &&\quad \hat{\Q}(\Eh) = \Pc_1(\Eh)^{d\times d, skew}, \\
	&\Zh(\Eh) = \Pc_1(\Eh)^d, &&\quad \Wh(\Eh) = \Pc_0(\Eh).&&
\end{align}
On the reference square, the spaces are defined as
\begin{align}
\begin{aligned}
\hat{\X}(\Eh) &= \left(\Pc_1(\Eh)^2 + r\,\curl(\xh^2 \yh) + s\, \curl (\xh\yh^2)\right)^2 \\ 
	&= \begin{pmatrix} \alpha_1 \xh + \beta_1 \yh + \gamma_1 + r_1\xh^2 + 2s_1\xh\yh & \alpha_2 \xh + \beta_2 \yh + \gamma_2 - 2r_1\xh\yh - s_1\yh^2 \\ \alpha_3 \xh + \beta_3 \yh + \gamma_3 + r_2\xh^2 + 2s_2\xh\yh & \alpha_4 \xh + \beta_4 \yh + \gamma_4 - 2r_2\xh\yh - s_2\yh^2 \end{pmatrix},  \\
\Vh(\Eh) &= \Pc_0(\Eh)^d, \quad
\hat{\Q}(\Eh) = \Qc_1(\Eh)^{2\times 2, skew}, \\
\Zh(\Eh) &= \Pc_1(\Eh)^2 + r\,\curl(\xh^2 \yh) + s\, \curl (\xh\yh^2)
= \begin{pmatrix}
	\alpha_5 \xh + \beta_5 \yh + \gamma_5 + r_3\xh^2 + 2s_3\xh\yh \\
	\alpha_6 \xh + \beta_6 \yh + \gamma_6 - 2r_3\xh\yh - s_3\yh^2
	\end{pmatrix},\\
	\Wh(\Eh) &= \Pc_0(\Eh).
\end{aligned}
\end{align}
These spaces satisfy
\begin{align*}
&\dvr \Xh(\Eh) = \Vh(\Eh), \, \dvr \Zh(\Eh) = \Wh(\Eh); \,\,
  \forall \th \in \Xh (\Eh),\, \forall \hat\zeta\in\Zh(\Eh),\, \forall \eh \in \partial \Eh,
  \,\, \th\, \nh_{\eh} \in \Pc_1(\eh)^d, \, \hat\zeta\cdot \nh_{\eh} \in \Pc_1(\eh).
\end{align*}
It is known \cite{brezzi1985two,BBF-book} that the degrees of
freedom for $\BDM_1$ can be chosen to be the values of the normal
fluxes at any two points on each edge $\eh$ of $\Eh$ in 2d 
or any three points one each face $\eh$ of $\Eh$ in 3d;
similarly for the normal stresses in the
case of $(\BDM_1)^d$. Here we choose these points to be at the
vertices of $\eh$ for both the velocity and stress spaces. This choice
is motivated by the use of the vertex quadrature rule introduced in the next
section.
	
To define the above spaces on any physical element $E \in \Tc_h$,
the following transformations are used
\begin{align*}
  &\tau \overset{\Pc}\leftrightarrow \hat{\tau} :
  \tau^T = \frac{1}{J_E} DF_E \hat{\tau}^T \circ F_E^{-1},
  & &v \leftrightarrow \hat{v} : v = \hat{v} \circ F_E^{-1}, 
  & & \xi \leftrightarrow \hat{\xi} : \xi = \hat{\xi} \circ F_E^{-1}, \\
  &\zeta \overset{\Pc}\leftrightarrow \hat\zeta : \zeta = \frac{1}{J_E} DF_E
  \hat{\zeta}\circ F_E^{-1}, 
&& w\leftrightarrow \hat{w} : w = \hat{w} \circ F_E^{-1},&&
\end{align*}
for $\tau \in \X$, $v \in V$, $\xi \in \Q$, $\zeta \in Z$ and $w\in
W$.  The velocity vector and stress tensor are mapped by the Piola
transformation, where the stress is transformed row-wise. The Piola
transformation preserves the normal components and the divergence of
the stress and velocity on element edges or faces. In particular, it
can be shown that
\begin{equation}\label{piola-prop}
\tau\, n_e = \frac{1}{|J_E DF^{-T}\hat{n}_{\eh}|_{\mathbb{R}^d}} \hat{\tau}\, \hat{n}_{\eh},
\quad
\zeta \cdot n_e = \frac{1}{|J_E DF^{-T}\hat{n}_{\eh}|_{\mathbb{R}^d}} \hat{\zeta}\cdot \hat{n}_{\eh},
\quad \dvr \tau = \frac{1}{J_E} \dvr \hat{\tau},
\quad \dvr \zeta = \frac{1}{J_E} \dvr \hat\zeta,
\end{equation}
where $|\cdot|_{\mathbb{R}^d}$ denotes the Euclidean vector norm.
The finite element spaces on $\Tc_h$ are defined as
\begin{align}
\begin{aligned} \label{mixed-spaces}
  \X_h &= \{\tau \in \X:  \tau|_E \overset{\Pc}\leftrightarrow \hat{\tau},
  \: \hat{\tau} \in \hat{\X}(\Eh) \quad \forall E\in\mathcal{T}_h\},  \\
  V_h &= \{v \in V:  v|_E \leftrightarrow \hat{v},\:
  \hat{v} \in \hat{V}(\Eh) \quad \forall E\in\mathcal{T}_h\}, \\
\Q_h &= \{\xi \in H^1(\Omega,\N):  \xi|_E \leftrightarrow \hat{\xi},
\: \hat{\xi} \in \hat{\Q}(\Eh) \quad \forall E\in\mathcal{T}_h\}, \\
Z_h &= \{\zeta \in Z: \zeta|_E \overset{\Pc}\leftrightarrow \hat{\zeta},\:
\hat{\zeta} \in \hat{Z}(\Eh) \quad \forall E\in\mathcal{T}_h\},  \\
W_h &= \{w \in W: w|_E \leftrightarrow \hat{w}, \: \hat{w} \in \hat{W}(\Eh)
\quad \forall E\in\mathcal{T}_h\}.
\end{aligned}
\end{align}

\begin{remark}
  Due to \eqref{piola-prop}, on each $E \in \mathcal{T}_h$, it holds
  that $\dvr \X_h = \frac{1}{J_E} V_h$ and $\dvr Z_h = \frac{1}{J_E}
  W_h$. In several places we will make choices for test functions, on
  each E, $v = J_E\, \dvr\tau$ or $w = J_E\, \dvr\zeta$. On quadrilaterals, $J_E$ is linear
  and positive. On simplices, $J_E$ is a positive constant, so in this case
  $\dvr \X_h = V_h$ and $\dvr Z_h = W_h$.
  \end{remark}

\subsection{The coupled $\BDM_1$ mixed finite element method}\label{sec:BDM}
With the finite element spaces defined above, the semidiscrete
five-field mixed finite element approximation
of the Biot poroelasticity system 
\eqref{eq:cts1}--\eqref{eq:cts5} reads as follows:
find $(\sigma_h,u_h,\gamma_h,z_h,p_h):[0,T] \mapsto
\X_h \times V_h \times \Q_h \times Z_h \times W_h$ such that, for a.e. $t \in (0,T)$,
\begin{align}
	  &\inp[A(\sigma_h + \a p_h I)]{\tau} + \inp[u_h]{\dvr{\tau}}
          + \inp[\gamma_h]{\tau} = \gnp[g_u]{\t\,n}_{\Gamma_D^{displ}}, &&
          \forall \tau \in \X_h, \label{eq:mfe1}\\
	  &\inp[\dvr{\sigma_h}]{v} = - \inp[f]{v}, && \forall v \in V_h,
          \label{eq:mfe2}\\
	&\inp[\sigma_h]{\xi}  = 0, && \forall \xi \in \Q_h, \label{eq:mfe3}\\
	  &\inp[\K z_h]{\zeta} - \inp[p_h]{\dvr{\zeta}}  =
          -\gnp[g_p]{\zeta\cdot n}_{\Gd^{pres}}, && \forall \zeta \in Z_h,
          \label{eq:mfe4}\\
	&\inp[c_0\dt{p_h}]{w} + \a\inp[\dt A(\sigma_h + \a p_h I)]{wI} 
+ \inp[\dvr{z_h}]{w}   = \inp[q]{w}, && \forall w \in W_h, \label{eq:mfe5}
\end{align}
with initial condition $p_h(0) = p_{h,0}$, where $p_{h,0}$ is a suitable approximation
of $p_0$. The
convergence of the above method is studied in \cite{Lee-Biot-five-field},
where it is shown that the method is robust for small storage
coefficient and for nearly incompressible materials. With an implicit
time discretization, it requires the solution of a large five-field
saddle point system at each time step, which is computationally
expensive. Motivated by the MFMFE \cite{wheeler2006multipoint} and
MSMFE \cite{msmfe1,msmfe2} methods, in the next sections we develop a
coupled MSMFE--MFMFE method based on a vertex quadrature rule that
allows for local elimination of the stress, rotation, and velocity
without loss of accuracy, resulting in a significantly more efficient
positive-definite cell-centered displacement-pressure system.
	
\subsection{A quadrature rule}
For any element-wise continuous vector or
tensor functions $\phi$ and $\psi$ on $\Omega$, we denote by
$$
(\varphi,\psi)_Q = \sum_{E \in \Tc_h}(\varphi,\psi)_{Q,E }
$$
the application of the element-wise vertex quadrature rule for computing
$(\varphi,\psi)$.  The integration on any element $E$ is performed by
mapping to the reference element $\Eh$. Let $\tilde\phi$ and
$\tilde\psi$ be the mapped functions on $\Eh$, using the standard
change of variables. Since
$(\phi,\psi)_E = (\tilde\phi,\tilde\psi J_E)_{\Eh}$, we define
$$
(\phi,\psi)_{Q,E} = \frac{|\Eh|}{s}\sum_{i=1}^s
\tilde\phi(\rh_i)\cdot\tilde\psi(\rh_i) J_E(\rh_i) = \frac{|\Eh|}{s}\sum_{i=1}^s
\phi(\r_i)\cdot\psi(\r_i)J_E(\rh_i),
$$
where $s$ is the number of vertices of $E$, $\r_i$ and $\rh_i$,
$i = 1,\ldots,s$, are the vertices of $E$ and $\Eh$, respectively,
and $\cdot$ has a meaning of inner product for both vector and
tensor valued functions.

The quadrature rule will be applied to the velocity, stress, and
stress-rotation bilinear forms. All three variables have degrees of
freedom associated with the mesh vertices. The quadrature rule
decouples degrees of freedom associated with a vertex from the rest of
the degrees of freedom, resulting in block-diagonal matrices
corresponding to these bilinear forms. Therefore the velocity, stress,
and rotation can be locally eliminated, reducing the method to solving
a cell-centered pressure-displacement system. More details on this
reduction will be provided in the following sections. 

The analysis of the MSMFE--MFMFE method will utilize the following
continuity and coercivity properties of the quadrature bilinear forms. 

\begin{lemma}\label{quad-coercive}
There exist positive constants $C_1$ and $C_2$ independent of $h$,
such that for any linear uniformly bounded and positive-definite
operator $L$ and for all $\phi, \psi \in \X_h, \Q_h, Z_h, W_h$,
\begin{equation} \label{quad-coercive-eqn}
\inp[L\phi]{\phi}_Q \ge C_1 \| \phi \|^2, \quad
\inp[L\phi]{\psi}_Q \le C_2 \| \phi \| \| \psi \|.
\end{equation}
\end{lemma}
\begin{proof}
The proof for functions in $\X_h, \Q_h, Z_h$ has been shown in
\cite{wheeler2006multipoint,msmfe1,msmfe2}. The proof for functions in $W_h$
is similar.
\end{proof}
Lemma~\ref{quad-coercive} implies the following norm equivalence.
\begin{corollary}\label{quad-inner-prod}
$\inp[L\phi]{\phi}_Q^{1/2}$
is a norm equivalent to $\|\phi\|$, which will be denoted by
$\|L^{1/2}\phi\|_Q$.
\end{corollary}
  
\subsection{The coupled multipoint stress-multipoint flux mixed finite
  element method}

We first note that there is a slight difference in the incorporation
of the Dirichlet boundary conditions between the simplicial and
quadrilateral grids. In particular, in the case of quadrilaterals, the
$L^2$ projection of the boundary data onto the space of piecewise
constants must be used in order to obtain optimal approximation of the
boundary term. On the other hand, such projection should not be used
on simplices, since it would result in non-optimal approximation. The
difference is due to different properties of the quadrature rules on
simplicial and quadrilateral grids, see
\cite{WheXueYot-msmortar,msmfe1,msmfe2}. For the conformity and
simplicity of the presentation, for the rest of the paper we consider
$g_u = g_p = 0$.

Our method, referred to as the MSMFE--MFMFE method, in its semidiscrete
form is defined as follows: find $(\sigma_h,u_h,\gamma_h,z_h,p_h):[0,T] \mapsto
\X_h \times V_h \times \Q_h \times Z_h \times W_h$ such that $p_h(0) = p_{h,0}$ and,
for a.e. $t \in (0,T)$,
\begin{align}
	  &\inp[A(\sigma_h + \a p_h I)]{\tau}_Q + \inp[u_h]{\dvr{\tau}}
          + \inp[\gamma_h]{\tau}_Q = 0, &&
          \forall \tau \in \X_h, \label{eq:msfmfe1}\\
	  &\inp[\dvr{\sigma_h}]{v} = - \inp[f]{v}, && \forall v \in V_h,
          \label{eq:msfmfe2}\\
	&\inp[\sigma_h]{\xi}_Q  = 0, && \forall \xi \in \Q_h, \label{eq:msfmfe3}\\
	  &\inp[\K z_h]{\zeta}_Q - \inp[p_h]{\dvr{\zeta}}  = 0,
          && \forall \zeta \in Z_h,
          \label{eq:msfmfe4}\\
	&\inp[c_0\dt{p_h}]{w} + \a\inp[\dt A(\sigma_h + \a p_h I)]{wI}_Q
+ \inp[\dvr{z_h}]{w}   = \inp[q]{w}, && \forall w \in W_h. \label{eq:msfmfe5}
\end{align}
\begin{remark}
We note that the quadrature rule is employed for both 
$\inp[A(\sigma_h + \a p_h I)]{\tau}_Q$ in \eqref{eq:msfmfe1}
and $\a\inp[\dt A(\sigma_h + \a p_h I)]{wI}_Q$ in \eqref{eq:msfmfe5},
since these two terms will be combined to obtain a coercive term in the
well-posedness analysis, while only quadrature rule on the stress term
$\inp[A\sigma_h]{\tau}_Q$ in \eqref{eq:msfmfe1} is needed for local
stress elimination.
\end{remark}

In the next sections we proceed with establishing existence,
uniqueness, stability, and error analysis for the semidiscrete
MSMFE--MFMFE method \eqref{eq:msfmfe1}--\eqref{eq:msfmfe5}. In
Section~\ref{sec:discr} we present the fully-discrete MSMFE--MFMFE
method and discuss the reduction of the algebraic system at each time
step to a positive definite cell-centered
displacement-pressure system.

\section{Existence and uniqueness for the semidiscrete MSMFE--MFMFE method}\label{sec:well-posed}

We first state the inf-sup stability of the mixed Darcy and elasticity
spaces, which will be utilized in the analysis. It is known
\cite{BBF-book} that the spaces $Z_h\times W_h$ satisfy the
inf-sup condition
\begin{align}\label{inf-sup-darcy}
  \exists \, \beta_1 > 0 \mbox{ such that } \forall \, w_h \in W_h, \quad
  \sup\limits_{0\neq \zeta \in Z_h} \frac{\inp[w_h]{\dvr{\zeta}} }{\|\zeta\|_{\dvr}}
  \ge \beta_1 \|w_h\|.
\end{align}
The inf-sup stability for the mixed elasticity spaces
$\X_h\times V_h \times \Q_h$ with quadrature has been studied
in \cite{msmfe1} on simplices and in \cite{msmfe2} on quadrilaterals.
In the case of quadrilaterals, the following
assumptions on the grid is needed \cite{msmfe2}:

\begin{enumerate} [label={\bf(M\arabic*)},align=left]
\item \label{M1} Each element $E$ has at most one edge on
  $\Gamma_N^{stress}$,

\item \label{M2} The mesh size $h$ is sufficiently small and there
  exists a constant $C$ such that for every pair of neighboring
  elements $E$ and $\tilde E$ such that $E$ or $\tilde E$ is a
  non-parallelogram, and every pair of edges $e \subset \partial E \setminus \partial \tilde E$,
  $\tilde e \subset \partial \tilde E \setminus \partial E$ that share a vertex,
	\begin{align*}
	|\r_{e}-\r_{\tilde e}|_{\mathbb{R}^2} \leq Ch^2,
	\end{align*}
	where $\r_e$ and $\r_{\tilde e}$ are the vectors corresponding to $e$ 
and $\tilde e$, respectively.
\end{enumerate}

We note that \ref{M2} can be thought of as a smoothness assumption on
the grid and it is not needed if the grid consists entirely of
parallelograms. For the rest of the paper we will tacitly assume that
\ref{M1}--\ref{M2} hold on quadrilaterals.

We have the following inf-sup condition on simplices \cite{msmfe1} and
quadrilaterals \cite{msmfe2}:
\begin{align}\label{inf-sup-elast}
  \exists \, \beta_2 > 0 \mbox{ such that } \forall \, v_h \in V_h, \, \xi_h \in \Q_h,
  \quad \sup\limits_{0\neq\t \in \X_h} \frac{\inp[v_h]{\dvr{\t}}
    + \inp[\xi_h]{\t}_Q}{\|\t\|_{\dvr}} \ge \beta_2(\|v_h\| + \|\xi_h\|).
\end{align}

We note that the semidiscrete method \eqref{eq:msfmfe1}--\eqref{eq:msfmfe5} 
is a system of differential-algebraic equations and the standard theory for
ordinary differential equations cannot be directly applied. Instead, the well
posedness analysis of \eqref{eq:msfmfe1}--\eqref{eq:msfmfe5} will be based
on the existence theory for degenerate parabolic systems, in particular 
\cite[Theorem 6.1(b)]{showalter2013monotone}.
\begin{theorem}  \label{thmsho61b}
	Let the linear, symmetric and monotone operator $\mathcal{N}$ be given for the real vector space $E$ to its algebraic
	dual $E^{*}$, and let $E_{b}'$ be the Hilbert space which is the dual of
        $E$ with the seminorm
	\[
        | x |_{b} \ = \ (\mathcal{N} x \, (x))^{1/2}  \, , \quad x \in E.
        \]
	Let $\mathcal{M} \subset E \times E_{b}'$ be a relation with domain 
$D = \{ x \in E \, : \, \mathcal{M}(x) \neq \emptyset \}$. 
Assume $\mathcal{M}$ is monotone and $Rg( \mathcal{N} + \mathcal{M}) \, = \, E_{b}'$. Then, for each $x_{0} \in D$ and
for each $\mathcal{F} \in W^{1 , 1}(0 , T ; E_{b}')$, there is a solution $x$ of
\begin{equation}\label{parabolic-system}        
  \ddt\left( \mathcal{N} x(t) \right) 
  +  \mathcal{M}\left( x(t) \right) \ni \mathcal{F}(t) \, , \quad a.e. \ \ 0 < t < T,
\end{equation}
with
\[ \mathcal{N} x \in W^{1 , \infty}(0 , T; E_{b}') \, , \  x(t) \in D \, ,  \
\mbox{\textit{for all} } 0 \le t \le T \, , 
 \ \mbox{ \textit{ and } } \mathcal{N}x(0) \, = \, \mathcal{N}x_{0} \, .   \]
\end{theorem}

\begin{theorem}\label{well-posed}
  For each $f \in W^{1 , \infty}(0 , T; L^2(\Omega))$, $q \in W^{1 , \infty}(0 , T; L^2(\Omega))$,
  and compatible initial data $(\sigma_{h,0},u_{h,0},\gamma_{h,0},z_{h,0},p_{h,0})$,
  the semidiscrete MSMFE--MFMFE method \eqref{eq:msfmfe1}--\eqref{eq:msfmfe5}
  has a unique solution $(\sigma_h,u_h,\gamma_h,z_h,p_h) \in
  W^{1,\infty}(0,T;L^2(\Omega,\M)) \cap L^\infty(0,T;\X_h)
  \times L^\infty(0,T;V_h) \times
L^\infty(0,T;\Q_h) \times L^\infty(0,T;Z_h) \times
W^{1,\infty}(0,T;W_h)$.
\end{theorem}
\begin{proof}
In order to fit \eqref{eq:msfmfe1}--\eqref{eq:msfmfe5} in the form of
Theorem~\ref{thmsho61b}, we consider a slightly modified formulation,
with \eqref{eq:msfmfe1} differentiated in time and the new variables
$\dot u_h$ and $\dot \gamma_h$ representing $\dt u_h$ and $\dt
\gamma_h$, respectively:
\begin{equation}\label{eq:msfmfe1-diff}
\inp[\dt A(\sigma_h + \a p_h I)]{\tau}_Q + \inp[\dot u_h]{\dvr{\tau}}
          + \inp[\dot \gamma_h]{\tau}_Q = 0, \quad \forall \tau \in \X_h.
  \end{equation}
Introducing the operators
\begin{gather*}
  (\As \sigma_h,\tau) = \inp[A\sigma_h]{\tau}_Q, \ 
  (\Asp \sigma_h, w) = \alpha\inp[A\sigma_h]{wI}_Q, \  
  (\Asu \sigma_h, v) = \inp[\dvr{\sigma_h}]{v}, \ 
  (\Asg \sigma_h, \xi) = \inp[\sigma_h]{\xi}_Q, \\
  (\Az \z_h,\zeta) = \inp[\K z_h]{\zeta}_Q, \quad
  (\Azp \z_h, w) = - \inp[\dvr{z_h}]{w}, \quad (\Ap p_h,w) = \inp[c_0 p_h]{w}
  + \a\inp[A\a p_h I]{wI}_Q, 
  \end{gather*}
we have a system in the form of \eqref{parabolic-system}, where
\begin{align*}
\dot x = \begin{pmatrix}
		\s_h \\ \dot u_h \\ \dot \g_h \\ z_h \\ p_h
\end{pmatrix}, \ 
\mathcal{N} = \begin{pmatrix}
		\As   & 0      & 0 & 0     & \Asp^T \\
		0     & 0      & 0      & 0     & 0 \\
		0     & 0      & 0      & 0     & 0 \\
		0     & 0      & 0      & 0     & 0 \\
		\Asp  & 0      & 0      & 0     & \Ap 
\end{pmatrix}, \ 
\mathcal{M} = \begin{pmatrix}
		0     & \Asu^T & \Asg^T & 0     & 0 \\
		-\Asu & 0      & 0      & 0     & 0 \\
		-\Asg & 0      & 0      & 0     & 0 \\
		0     & 0      & 0      & \Az   & \Azp^T \\
		0     & 0      & 0      & -\Azp & 0
\end{pmatrix}, \ 
\mathcal{F} =   \begin{pmatrix}
0 \\ -f \\ 0 \\ 0 \\ q
\end{pmatrix}.   
\end{align*}
The dual space $E_b'$ is $L^2(\Omega,\M)\times 0 \times 0 \times 0 \times L^2(\Omega)$, 
and the condition $\mathcal{F} \in W^{1 , 1}(0 , T ; E_{b}')$
in Theorem~\ref{thmsho61b} allows for non-zero source terms only in
the equations with time derivatives. In our case this means $f =
0$. We can reduce our problem to a system with $f = 0$ by solving for
each $t \in (0,T]$ an elasticity problem with a source term $f$,
cf. \cite{Showalter-SIAMMA} for a similar approach:
$$
  \begin{pmatrix}
    \As & \Asu^T & \Asg^T \\
    -\Asu & 0      & 0 \\
    -\Asg & 0      & 0
  \end{pmatrix}
  \begin{pmatrix} \s_h^f \\ \dot u_h^f \\ \dot \g_h^f \end{pmatrix} =
  \begin{pmatrix} 0 \\ -f \\ 0 \end{pmatrix},
  $$
  and subtracting this solution from the solution to
  \eqref{eq:msfmfe1}--\eqref{eq:msfmfe5}, resulting in a problem with
  a modified right hand side $\mathcal{F} = (\As(\s_h^f - \dt\s_h^f),
  0, 0, 0, q - \Asp\dt\s_h^f)^T$.

  The range condition $Rg( \mathcal{N} + \mathcal{M}) \, = \, E_{b}'$ can be verified
  by showing that the square finite dimensional homogeneous system: find
  $(\hat\sigma_h,\hat u_h,\hat\gamma_h,\hat z_h,\hat p_h) \in \X_h \times V_h \times \Q_h \times Z_h \times W_h$ such that
\begin{align}
	  &\inp[A(\hat\sigma_h + \a \hat p_h I)]{\tau}_Q + \inp[\hat u_h]{\dvr{\tau}}
          + \inp[\hat\gamma_h]{\tau}_Q = 0, &&
          \forall \tau \in \X_h, \label{eq:msfmfe1-0}\\
	  &\inp[\dvr{\hat\sigma_h}]{v} = 0, && \forall v \in V_h,
          \label{eq:msfmfe2-0}\\
	&\inp[\hat\sigma_h]{\xi}_Q  = 0, && \forall \xi \in \Q_h, \label{eq:msfmfe3-0}\\
	  &\inp[\K \hat z_h]{\zeta}_Q - \inp[\hat p_h]{\dvr{\zeta}}  = 0,
          && \forall \zeta \in Z_h,
          \label{eq:msfmfe4-0}\\
	&\inp[c_0\hat p_h]{w} + \a\inp[A(\hat\sigma_h + \a \hat p_h I)]{wI}_Q
+ \inp[\dvr{\hat z_h}]{w}   = 0, && \forall w \in W_h, \label{eq:msfmfe5-0}
\end{align}
has only the zero solution, see also
\cite[Section~3.4]{Lee-Biot-five-field}.  Taking $(\tau, v, \xi, \zeta, w) =
(\hat\sigma_h,\hat u_h,\hat\gamma_h,\hat z_h,\hat p_h)$ and combining
the equations implies $\|A^{1/2}(\hat\sigma_h + \a \hat p_h I)\|_Q^2 +
\|c_0^{1/2}\hat p_h\|^2 + \|K^{-1/2}\hat z_h\|_Q^2 = 0$, which gives
$\hat\sigma_h + \a \hat p_h I = 0$ and $\hat z_h = 0$, using the
positive definiteness of $A$ and $K$ and the coercivity
\eqref{quad-coercive-eqn}.  Then the Darcy inf-sup condition
\eqref{inf-sup-darcy} implies that $\hat p_h = 0$, and therefore $\hat
\sigma_h = 0$. The elasticity inf-sup condition \eqref{inf-sup-elast}
now implies that $\hat u_h = 0$ and $\hat \gamma_h = 0$.

The above argument can also be used to conclude that $\mathcal{N}$ and
$\mathcal{M}$ are non-negative, and therefore, due to their linearity, monotone.

Finally, we need compatible initial data $\dot x_{0} \in D$, i.e.,
$\mathcal{M}\dot x_0 \in E_b'$. Let us consider first initial data
$x_0 = (\sigma_{h,0},u_{h,0},\gamma_{h,0},z_{h,0},p_{h,0})$ for
the non-differentiated problem \eqref{eq:msfmfe1}--\eqref{eq:msfmfe5}. We take
$x_0$ to be the
elliptic projection of the initial data $\tilde x_0 =
(\sigma_0,u_0,\gamma_0,z_0,p_0)$ for the weak formulation
\eqref{eq:cts1}--\eqref{eq:cts5}, which is constructed from $p_0$ by
the procedure described at the end of Section~\ref{sec:model}.
With the reduction to a problem with $f = 0$,
the construction satisfies $(\mathcal{N} + \mathcal{M}) \tilde x_0 \in E_b'$. Since we have
\begin{equation}\label{ell-proj}
(\mathcal{N} + \mathcal{M}) x_0 = (\mathcal{N} + \mathcal{M}) \tilde x_0,
\end{equation}
this implies that $\mathcal{M} x_0 = (\mathcal{N} + \mathcal{M})
\tilde x_0 - \mathcal{N} x_0 \in E_b'$. For the initial data of the differentiated
problem \eqref{eq:msfmfe1-diff},\eqref{eq:msfmfe2}--\eqref{eq:msfmfe5}, we simply take
$\dot x_0 = (\s_{h,0},0,0,z_{h,0},p_{h,0})$, which also satisfies $\mathcal{M} \dot x_0 \in E_b'$.
We note that $u_{h,0}$ and $\gamma_{h,0}$ are not needed for the differentiated
problem, but will be used to recover the solution of the original problem.

Now, all conditions of Theorem~\ref{thmsho61b} are satisfied
and we conclude the existence of a solution to \eqref{eq:msfmfe1-diff},
\eqref{eq:msfmfe2}--\eqref{eq:msfmfe5} with
$\sigma_h \in W^{1,\infty}(0,T;L^2(\Omega,\M)) \cap L^\infty(0,T;\X_h)$,
$p_h \in W^{1,\infty}(0,T;W_h)$, $\sigma_h(0) = \sigma_{h,0}$, and $p_h(0) = p_{h,0}$.
From the equations 
we can further conclude that $\dot u_h \in L^\infty(0,T;V_h)$,
$\dot \gamma_h \in L^\infty(0,T;\Q_h)$ and
$z_h \in L^\infty(0,T;Z_h)$. By taking $t \to 0$ in \eqref{eq:msfmfe4} and using
that $z_{h,0}$ and $p_{h,0}$ satisfy \eqref{eq:msfmfe4} at $t = 0$,
we also have that $z_h(0) = z_{h,0}$. 

Next, we recover the solution of the original problem. Let us define
\begin{equation}\label{defn-u-gamma}
  u_h(t) = u_{h,0} + \int_0^t \dot u_h ds, \quad \g_h(t) = \g_{h,0} +  \int_0^t \dot \g_h ds,
  \quad \forall \, t \in [0,T].
\end{equation}
By construction, $u_h(0) = u_{h,0}$ and $\gamma_h(0) = \gamma_{h,0}$. Integrating
\eqref{eq:msfmfe1-diff} in time from 0 to any $t \in (0,T]$ and using that
  $\sigma_{h,0}$, $u_{h,0}$, and $\gamma_{h,0}$ satisfy \eqref{eq:msfmfe1} at $t = 0$,
  we conclude that \eqref{eq:msfmfe1} holds for all $t$. This completes the existence
  proof. Uniqueness follows from the stability bound given in Theorem~\ref{thm:stab} in
  the next section.
  \end{proof}

\begin{remark}
  The above argument and the stability bound below do not require $c_0 > 0$, implying
  well posedness even for $c_0 = 0$.
  \end{remark}

\section{Stability analysis of the semidiscrete MSMFE--MFMFE method}
In this section we derive a stability bound for the MSMFE--MFMFE method
\eqref{eq:msfmfe1}--\eqref{eq:msfmfe5}. We remark that stability
analysis for the $\BDM_1$ MFE method \eqref{eq:mfe1}--\eqref{eq:mfe5}
was not performed in \cite{Lee-Biot-five-field}, where only error analysis
was carried out. The stability analysis 
is more involved than the error
analysis, since controlling the boundary condition term
$\gnp[g_p]{\zeta\cdot n}_{\Gd^{pres}}$ requires bounding $\|\dvr
z_h\|$. Even though we consider $g_p = 0$, we derive a bound on
$\|\dvr z_h\|$, thus obtaining full control on $\|z_h\|_{\dvr}$.

\begin{theorem}\label{thm:stab}
There exists a positive constant $C$ independent of $h$ and $c_0$, such that 
the solution of \eqref{eq:msfmfe1}--\eqref{eq:msfmfe5} satisfies
\begin{align}
  &\|\s_h\|_{L^{\infty}(0,T;\Hdiv)} +\|u_h\|_{L^{\infty}(0,T;L^2(\O))} + \|\g_h\|_{L^{\infty}(0,T;L^2(\O))}
  +\|z_h\|_{L^{\infty}(0,T;L^2(\O))} + \|p_h\|_{L^{\infty}(0,T;L^2(\O))} \nonumber \\ 
  &\:+\|\s_h\|_{L^{2}(0,T;\Hdiv)} +\|u_h\|_{L^{2}(0,T;L^2(\O))} + \|\g_h\|_{L^{2}(0,T;L^2(\O))}
  +\|z_h\|_{L^{2}(0,T;\Hdiv)} + \|p_h\|_{L^{2}(0,T;L^2(\O))} \nonumber \\
	& \qquad\quad \leq C\left(\|f\|_{H^{1}(0,T;L^2(\O))}
  +\|q\|_{H^{1}(0,T;L^2(\O))} + \|p_0\|_{H^1(\Omega)} + \|K\nabla p_0\|_{\Hdiv} \right). \label{stability}
	\end{align}
\end{theorem}
\begin{proof}
We differentiate \eqref{eq:msfmfe1} in time, choose $(\t,v,\xi,\zeta,w)
= (\s_h, \dt u_h, \dt \g_h, z_h, p_h)$ in equations
\eqref{eq:msfmfe1}--\eqref{eq:msfmfe5}, and combine them to obtain
\begin{align*}
  \inp[\dt(A\s_h + \a p_h I)]{\s_h + \a p_h I}_Q + \inp[c_0\dt{p_h}]{p_h}
+ \inp[\K z_h]{z_h}_Q = \inp[f]{\dt u_h} + \inp[q]{p_h},
\end{align*}
implying
\begin{align}\label{st-eq:5}
  \frac{1}{2}\dt\left[ \|A^{1/2}(\s_h+\a p_hI)\|^2_Q + \|c_0^{1/2}p_h\|^2 \right]
  +\|K^{-1/2}z_h\|^2_Q  = \inp[f]{\dt u_h} + \inp[q]{p_h}.
\end{align}
Next, integrating \eqref{st-eq:5} in time from $0$ to an arbitrary
$t\in (0,T]$ results in
\begin{align*}
  \frac{1}{2}& \left[\|A^{1/2}(\s_h+\a p_hI)(t)\|^2_Q + \|c_0^{1/2}p_h(t)\|^2\right]
  +\int_0^t \|K^{-1/2}z_h\|^2_Q\, ds\\
	&= \int_0^t\left( \inp[q]{p_h}-\inp[\dt f]{ u_h} \right) ds 
+\frac{1}{2}\left[ \|A^{1/2}(\s_h+\a p_hI)(0)\|^2_Q + \|c_0^{1/2}p_h(0)\|^2\right] 
+ (f,u_h)(t) - (f,u_h)(0).
\end{align*}
Applying the Cauchy-Schwartz and Young's inequalities, we obtain
\begin{align}
  & \|A^{1/2}(\s_h+\a p_hI)(t)\|^2_Q + \|c_0^{1/2}p_h(t)\|^2 +
  2 \int_0^t \|K^{-1/2}z_h\|^2_Q\, ds \nonumber\\
  &\quad \leq \epsilon_1 \left(\|u_h(t)\|^2
  + \int_0^t(\|p_h\|^2+\|u_h\|^2)\,ds \right)
  +\frac{1}{\epsilon_1}\left(\|f(t)\|^2
  +\int_0^t \left(\|q\|^2+\|\dt f\|^2 \right) ds \right)
  \nonumber \\
  &\quad \quad + \|A^{1/2}(\s_h+\a p_hI)(0)\|^2_Q + \|c_0^{1/2}p_h(0)\|^2
  + \|u_h(0)\|^2+\|f(0)\|^2. \label{st-eq:6}
\end{align}
Using the inf--sup condition \eqref{inf-sup-elast} and \eqref{eq:msfmfe1}, we
bound $\|u_h\|$ and  $\|\g_h\|$ as follows,
\begin{align}
	\|u_h\| + \|\g_h\| &\leq C\sup\limits_{0\neq\t \in \X_h} \frac{\inp[u_h]{\dvr{\t}} + \inp[\g_h]{\t}_Q}{\|\t\|_{\dvr}} \nonumber\\
	&= C\sup\limits_{0\neq\t \in \X_h} \frac{-\inp[A^{1/2}(\s_h+\api)]{A^{1/2}\t}_Q  }{\|\t\|_{\dvr}} \leq C\|A^{1/2}(\s_h+\api)\| , \label{st-eq:7}
\end{align}
where in the last step we used the equivalence of norms as stated in
Corollary \ref{quad-inner-prod}. We also note that
\begin{equation}\label{l2-inf-sup-elast}
  \int_0^t \left(\|u_h\|^2 + \|\g_h\|^2\right) ds \le C \int_0^t
  \left(\|\s_h\|^2 + \|p_h\|^2\right) ds.
\end{equation}
Similarly, using the inf-sup
condition \eqref{inf-sup-darcy} and \eqref{eq:msfmfe4}, we have
\begin{align}
  \|p_h\| &\leq C\sup\limits_{0 \neq \zeta \in Z_h}
  \frac{\inp[p_h]{\dvr{\zeta}} }{\|\zeta\|_{\dvr}}
  = C\sup\limits_{0 \neq \zeta \in Z_h} \frac{\inp[K^{-1}z_h]{\zeta}_Q }
  {\|\zeta\|_{\dvr}} \leq C\|K^{-1/2}z_h\|. \label{st-eq:8}
\end{align}
To obtain a bound on $\int_0^t \|\s_h\|^2 ds$, which appears on the
right hand side of \eqref{l2-inf-sup-elast}, we take $\t =\s_h,\, v=
u_h,\, \xi =\gamma_h$ in \eqref{eq:msfmfe1}--\eqref{eq:msfmfe3}, and
use Cauchy-Schwartz and Young's inequalities, to obtain
\begin{align}
  \|\s_h\|^2\leq C \Big( \|p_h\|^2 + \epsilon_2\|u_h\|^2
  + \frac{1}{\epsilon_2}\|f\|^2 \Big). \label{st-eq:10}
\end{align}
Also, testing \eqref{eq:msfmfe2} with $v = J_E \, \dvr \s_h$ on each $E \in \Tc_h$,
we obtain a bound on the stress divergence:
\begin{align}
	\|\dvr \s_h\| \leq \|f\|. \label{st-eq:9}
\end{align}
Combining inequalities \eqref{st-eq:6}--\eqref{st-eq:9} and
choosing $\epsilon_2$ small enough, then $\epsilon_1$ small enough, we obtain
\begin{align}
  & \|A^{1/2}(\s_h+\a p_hI)(t)\|^2 +\|u_h(t)\|^2 + \|\g_h(t)\|^2
  + \|c_0^{1/2}p_h(t)\|^2 + \|\dvr \s_h(t)\|^2  \nonumber \\
  & \qquad + \int_0^t \left(\|\s_h\|^2 + \|u_h\|^2 + \|\g_h\|^2 
  + \|K^{-1/2}z_h\|^2 + \|p_h\|^2 + \|\dvr \s_h\|^2 \right) ds  \nonumber \\
  & \quad \leq C\Big(
  \|f(t)\|^2+\int_0^t \left( \|q\|^2 + \|f\|^2 + \|\dt f\|^2 \right) ds \nonumber \\
& \qquad\qquad
  + \|\s_h(0)\|^2 + \|p_h(0)\|^2 + \|u_h(0)\|^2
          +\|f(0)\|^2 \Big) .\label{st-eq:8-1}
\end{align}

\noindent
{\bf Estimate for $\dvr z_h$.}
We note that \eqref{st-eq:8-1} is a self-contained stability estimate. We now proceed
with obtaining a bound on $\|\dvr z_h\|$. In the process, we also
obtain a bound on $\|K^{-1/2}z_h(t)\|$ for all $t$, and as a result, a bound on 
$\|p_h(t)\|$ for all $t$ that is independent of $c_0$.
We choose on each $E \in \Tc_h$, $w_h = J_E \, \dvr z_h$ in \eqref{eq:msfmfe5} and obtain
\begin{align}
  \|\dvr z_h\| \leq C\left(\|c_0^{1/2}\dt p_h\| +  \|\dt A^{1/2} (\s_h+\api )\|
  +\|q\|\right) \label{st-eq:12}.
\end{align}
To control the first two terms on the right hand side of
\eqref{st-eq:12}, we differentiate equations
\eqref{eq:msfmfe1}--\eqref{eq:msfmfe4} in time and combine them with 
\eqref{eq:msfmfe5} as it was done in \eqref{st-eq:5}--\eqref{st-eq:6}, with
the choice $(\t,v,\xi,\z,w) = (\dt\s_h, \dt u_h, \dt \g_h, z_h, \dt p_h)$, resulting in
\begin{align}
  & 2\int_0^t\left(\|\dt A^{1/2}(\s_h+\api)\|_Q^2 + \|c_0^{1/2}\dt p_h\|^2\right)\, ds
  + \|K^{-1/2}z_h(t)\|^2_Q \nonumber\\
  & \quad \le \epsilon\left(\|p_h(t)\|^2 + \int_0^t\|\dt u_h\|^2 ds \right)
  + \frac{1}{\epsilon}\left(\|q(t)\|^2 + \int_0^t\|\dt f\|^2 ds \right) \nonumber\\
  & \qquad + \int_0^t\left(\| p_h\|^2 + \|\dt q\|^2\right) ds
  + \|K^{-1/2}z_h(0)\|^2_Q  + \|p_h(0)\|^2 + \|q(0)\|^2 \label{st-eq:13}.
\end{align}
Using the inf--sup condition \eqref{inf-sup-elast} and \eqref{eq:msfmfe1},
differentiated in time, we have
\begin{align}
\|\dt u_h\| + \|\dt \g_h\| &\leq C\|\dt A^{1/2}(\s_h+\api)\|. \label{st-eq:14}
\end{align}
Combining \eqref{st-eq:13}, \eqref{st-eq:14}, and \eqref{st-eq:8}, we get
\begin{align}
  &\int_0^t\left(\|\dt A^{1/2}(\s_h+\api)\|^2+\|\dt u_h\|^2
  +\|\dt \g_h\|^2 + \|c_0^{1/2}\dt p_h\|^2\right) ds
  + \|K^{-1/2}z_h(t)\|^2 +\|p_h(t)\|^2 \nonumber\\
  &\quad \leq C\left(\int_0^t\left(\|p_h\|^2 +
  \|\dt q\|^2 + \|\dt f\|^2 \right)\,ds + \|q(t)\|^2  
	+\|z_h(0)\|^2 + \|p_h(0)\|^2 +\|q(0)\|^2\right).  \label{st-eq:15}
\end{align}
Integrating \eqref{st-eq:12} in time and using \eqref{st-eq:15} and
\eqref{st-eq:8-1}, results in
\begin{align}
  & \|K^{-1/2}z_h(t)\|^2 +\|p_h(t)\|^2 +  \int_0^t\|\dvr z_h\|^2 ds   \nonumber \\
  &\qquad \leq C\Big(\|q(t)\|^2
  + \|f(t)\|^2+\int_0^t \left( \|q\|^2 + \|f\|^2 + \|\dt q\|^2 + \|\dt f\|^2 \right) ds \nonumber \\
& \qquad\qquad
  + \|\s_h(0)\|^2 + \|p_h(0)\|^2 + \|u_h(0)\|^2 + \|z_h(0)\|^2 + \|q(0)\|^2
         + \|f(0)\|^2 \Big). \label{st-eq:16}
\end{align}
We note that the control on $\|A^{1/2}(\s_h+\a p_hI)(t)\|$ and $\|p_h(t)\|$ also implies
a bound on $\|\s_h(t)\|$:
\begin{align}
\|\s_h\| \leq  C (\|A^{1/2}(\s_h + \api)\| + \|p_h\|). \label{st-eq:18}
\end{align}
    Finally, we recall the construction of the initial
      data $(\sigma_0,u_0,\gamma_0,z_0,p_0)$ for the weak formulation
      \eqref{eq:cts1}--\eqref{eq:cts5}, see Section~\ref{sec:model},
      and that the discrete initial data
      $(\sigma_{h,0},u_{h,0},\gamma_{h,0},z_{h,0},p_{h,0})$ is taken
      as its elliptic projection, see \eqref{ell-proj}. Then following
      the steady-state version of the arguments presented in
      \eqref{st-eq:5}--\eqref{st-eq:18}, we obtain
\begin{align}
  & \|\s_h(0)\| + \|u_h(0)\| + \|\gamma_h(0)\| + \|p_h(0)\| + \|z_h(0)\| 
  \le C(\|\s_0\| + \|u_0\| + \|\gamma_0\| + \|p_0\| + \|z_0\|) \nonumber \\
  & \quad\qquad \le C (\|p_0\|_{H^1(\Omega)} + \|K\nabla p_0\|_{\Hdiv}). \label{init-data-bound}
\end{align}
The proof is completed by combining \eqref{st-eq:8-1},
\eqref{st-eq:9}, \eqref{st-eq:16}, \eqref{st-eq:18}, and
\eqref{init-data-bound}.
\end{proof}

\begin{remark}
  The constant in \eqref{stability} does not depend on $c_0$, so we
  have stability even for $c_0 = 0$. Furthermore, since we did not use
  Gronwall's inequality in the proof, the constant also does not
  involve exponential growth in time, resulting in a long-time stability.
  \end{remark}

\section{Error analysis}

In this section we establish optimal order error estimates for all variables in their natural
norms. 

\subsection{Preliminaries}
We begin with several auxiliary results that will be used to bound the approximation and
quadrature errors. Due to the reduced approximation properties of the MFE spaces on
general quadrilaterals \cite{arnold2005quadrilateral}, we restrict the
quadrilateral elements to be $O(h^2)$-perturbations of parallelograms:
\begin{align}\label{h2-parall}
\|\r_{34}-\r_{21}\| \leq Ch^2.
\end{align}
In this case it is easy to verify (see \cite{wheeler2006multipoint} for details) that
\begin{align}
	|DF_E|_{1,\infty,\Eh} \leq Ch^2 \quad \text{and} \quad \left|\frac{1}{J_E}DF_E\right|_{j,\infty,\Eh}\leq Ch^{j-1},\, j=1,2. \label{scaling-of-mapping-2}
\end{align}
Let $Q^0:L^2(\O)\rightarrow W_h$ be a projection operator satisfying for any $\phi \in L^2(\O)$,
\begin{equation*}\label{defn-Q0}
(\hat Q^0 \hat\phi - \hat\phi,\hat w)_{\hat E} = 0, \quad \forall \, \hat w \in \hat W (\hat E),
\quad Q^0 \phi = \hat Q^0 \hat\phi \circ F_E^{-1} \ \ \forall E \in \mathcal{T}_h.
\end{equation*}
We will also use $Q^0:L^2(\O,\R^d)\rightarrow V_h$, which is the above
operator applied component-wise. It follows from \eqref{piola-prop} that
\begin{align}
\begin{aligned}\label{const-proj}  
  \forall \, \phi \in L^2(\O,\R^d), \quad
  (Q^0 \phi - \phi, \dvr \tau) & = 0, \quad \forall \, \tau \in \X_h, \\
  \forall \, \phi \in L^2(\O), \quad 
  (Q^0 \phi - \phi, \dvr \zeta) & = 0, \quad \forall \, \zeta \in Z_h.
  \end{aligned}
\end{align}
Let $Q^1:L^2(\O,\N)\rightarrow \Q_h$ be the 
$L^2$-projection operator satisfying for any $\phi \in L^2(\O,\N)$,
\begin{align}
(Q^1 \phi - \phi, \xi) = 0, \quad \forall \, \xi\in \Q_h. \label{lin-proj}
\end{align}
Let $\Pi: \X \cap H^1(\Omega,\M) \to \X_h$ be the canonical mixed projection operator
acting on tensor valued functions. We will
also use the same notation for the projection operator acting on vector
valued functions, $\Pi: Z \cap H^1(\Omega,\R^d) \to Z_h$. It is shown in
\cite{brezzi1985two,BBF-book} and \cite{wang1994mixed} that $\Pi$ satisfies
\begin{align}
\begin{aligned}
\label{div-prop}
\forall \psi \in H^1(\Omega,\M), \quad (\dvr (\Pi\psi -\psi), v) &= 0, \qquad \forall v \in V_h, \\
\forall \psi \in H^1(\Omega,\R^d), \quad (\dvr (\Pi \psi - \psi), w) &= 0, \qquad \forall w \in W_h.
\end{aligned} 
\end{align}
We will also make use of the mixed projection operator onto the lowest
order Raviart-Thomas space $\RT_0$
\cite{RT,Nedelec80,BBF-book}.  This additional construction
is needed only for the error analysis on quadrilaterals, although for
uniformity in the forthcoming proofs we will treat the simplicial case
in the same fashion. We denote the $\RT_0$-based spaces by $\X^0_h$
and $Z^0_h$ for tensors and vectors, respectively, where the former is
obtained from $d$ copies of the latter.  The degrees of freedom of
$\X_h^0$ or $Z^0_h$ are constant values of the normal stress or
velocity on all edges (faces). The $\RT_0$ mixed projection operator,
denoted by $\Pi^0$, has properties similar to the $\BDM_1$ projection
operator $\Pi$. It also satisfies
	\begin{align}
	\begin{aligned}\label{prop-5}
	  &\dvr\Pi^0\t = \dvr\t \quad\mbox{and}\quad
          \|\Pi^0\t\| \le C\|\t\| \quad \forall \t\in\X_h, \\
	  &\dvr\Pi^0 \zeta = \dvr \zeta \quad\mbox{and}\quad
          \|\Pi^0 \zeta\| \le C\|\zeta\| \quad \forall \zeta \in Z_h. 
	\end{aligned}
	\end{align}
	
The following lemma summarizes well-known continuity and approximation properties of the projection operators, where $\H \in \{\M,\R^d\}$.
\begin{lemma}
There exists a constant $C>0$ such that
\begin{align}
	&\| \phi - Q^0 \phi \| \le C\|\phi\|_r h^r, && \forall \phi\in H^r(\O),&& 0\le r\le 1, \label{approx-1}\\
	&\| \phi - Q^1\phi \| \le C\|\phi\|_r h^r, && \forall \phi \in H^r(\O,\N),&& 0\le r\le 1, \label{approx-2} \\
	&\| \psi - \Pi\psi \| \le C \|\psi\|_r h^r, && \forall \psi\in H^r(\O,\H),\, && 1\le r\le 2, \label{approx-3}\\
	&\| \psi - \Pi^0\psi \| \le C \|\psi\|_1 h, && \forall \psi\in H^1(\O,\H), \label{approx-4} \\
	&\| \dvr(\psi - \Pi\psi) \|+\| \dvr(\psi - \Pi^0\psi) \| \le C \|\dvr\psi\|_r h^r, && \forall \psi\in H^{r+1}(\O,\H),\, && 0\le r\le 1. \label{approx-5}
\end{align}
In addition, for all elements $E\in \Tc_h$, there exists a constant $C>0$, such that
\begin{align}
  \|Q^0 \phi\|_{E} &\leq C \|\phi\|_{E}, && \forall \phi\in L^2(E), \label{l2-continuity-q0} \\
 \|Q^1 \phi\|_{1,E} &\leq C \|\phi\|_{1,E}, && \forall \phi\in H^1(E,\N), \label{h1-continuity-q1} \\
  \|\Pi \psi\|_{1,E} &\leq C \|\psi\|_{1,E}, && \forall \psi\in H^1(E,\H). \label{h1-continuity-bdm} 
\end{align}
\end{lemma} 
\begin{proof}
The proof of bounds for the $L^2$-projections
\eqref{approx-1}--\eqref{approx-2} can be found in
\cite{ciarlet2002finite}; and bounds \eqref{approx-3}--\eqref{approx-5}
can be found in \cite{BBF-book,roberts1991mixed} for affine
elements and \cite{wang1994mixed,arnold2005quadrilateral} for
$h^2$-parallelograms. Finally, \eqref{l2-continuity-q0} is the stability of the
$L^2$-projection and the proof of
\eqref{h1-continuity-q1}--\eqref{h1-continuity-bdm} was presented in
\cite{wheeler2006multipoint}.
\end{proof}

The following result is needed in the error analysis.
\begin{lemma}\label{quad-error-const}
For any $\hat\tau \in \Xh(\Eh)$ and $\hat\zeta \in \Zh(\Eh)$,
\begin{align}
  \inp[\hat\tau - \Pih^0\hat\tau]{\th_0}_{\Qh,\Eh} = 0 \quad
  \mbox{for all constant tensors $\th_0$,} \label{q-err-const-tensor}\\
  \inp[\hat\zeta - \Pih^0\hat\zeta]{\hat\zeta_0}_{\Qh,\Eh} = 0 \quad
  \mbox{for all constant vectors $\hat\zeta_0$}. \label{q-err-const-vector}
\end{align}
\end{lemma}
\begin{proof}
  The property \eqref{q-err-const-vector} was shown in
  \cite[Lemma~2.2]{wheeler2006multipoint} on the reference square. The proof on
  the reference simplex follows in a similar way. The property
  \eqref{q-err-const-tensor} follows from \eqref{q-err-const-vector}.
  \end{proof}

For $\phi,\, \psi \in \X_h, \Q_h, Z_h, W_h$, denote the quadrature error by
\begin{align} \label{quad-err-def}
  \forall E \in \mathcal{T}_h, \quad \tet_E(L\phi,\psi) := (L\phi,\psi)_E - (L\phi,\psi)_{Q,E}, \quad
	\tet(L\phi,\psi) := (L\phi,\psi) - (L\phi,\psi)_{Q}.
\end{align}
The next result summarizes the quadrature error bounds.
\begin{lemma}\label{lemma-quad-bound}
  For all $E \in \mathcal{T}_h$, if $K^{-1}|_E\in W^{1,\infty}(E)$ and
  $A|_E\in W^{1,\infty}(E)$, then there is a constant $C>0$ independent of $h$ such that 
\begin{align}
  & \left|\theta_E\inp[K^{-1} \zeta]{\rho}\right| \leq C h\|K^{-1}\|_{1,\infty,E}\|\zeta\|_{1,E}\|\rho\|_E,
  && \forall \zeta \in Z_h,\, \rho\in Z^0_h, \label{theta-bound-vel} \\
  & \left|\theta_E\inp[A \tau]{\chi}\right| \leq C h\|A\|_{1,\infty,E}\|\tau\|_{1,E}\|\chi\|_E,
  && \forall \tau\in \X_h,\, \chi\in \X^0_h, \label{theta-bound-str}\\
  & \left|\theta_E\inp[A \tau]{w I}\right| \leq C h\|A\|_{1,\infty,E}\|\tau\|_{1,E}\|w\|_E,
  && \forall \tau\in \X_h,\, w\in W_h, \label{theta-bound-str-pr}\\
& \left|\theta_E\inp[A w I]{r I}\right| \leq C h\|A\|_{1,\infty,E}\|w\|_{E}\|r\|_E, && \forall w, r\in W_h, \label{theta-bound-pr}\\
  & \left|\theta_E\inp[ \tau]{\xi}\right| \leq C h\|\tau\|_{1,E}\|\xi\|_{E},
  && \forall \tau \in \X_h, \xi\in \Q_h, \label{theta-bound-rot}\\
& \left|\theta_E\inp[ \tau]{\xi}\right| \leq C h\|\tau\|_{E}\|\xi\|_{1,E},
  && \forall \tau \in \X^0_h, \xi\in \Q_h, \label{theta-bound-rot-2}\\
  & \left|\inp[K^{-1}\rho]{\zeta-\Pi^0 \zeta}_{Q,E}\right|
  \leq C h \|K^{-1}\|_{1,\infty,E} \|\rho\|_{1,E}\|\zeta\|_E,
  \quad && \forall \rho,\zeta \in Z_h, \label{er-bound-vel} \\
  & \left|\inp[A(\chi + wI)]{\tau-\Pi^0 \tau}_{Q,E}\right| 
  \leq C h \|A\|_{1,\infty,E}(\|\chi\|_{1,E} + \|w\|_{E})\|\tau\|_E, &&
  \forall \chi,\tau \in \X_h, w \in W_h, \label{er-bound-st} \\
  & \left|\inp[\xi]{\tau - \Pi^0 \tau}_{Q,E}\right| \leq C h\|\xi\|_{1,E}\|\tau\|_E,
  && \forall \xi \in \Q_h, \tau\in \X_h. \label{er-bound-rot}
\end{align}
\end{lemma}
\begin{proof}
The estimates \eqref{theta-bound-vel} and \eqref{er-bound-vel} can be
found in \cite{wheeler2006multipoint}. We note that
\eqref{er-bound-vel} was stated only on quadrilaterals in
\cite{wheeler2006multipoint}, but it also holds on simplices, since it
follows from mapping to the reference element and
\eqref{q-err-const-vector}. Bounds \eqref{theta-bound-str} and
\eqref{theta-bound-rot}--\eqref{theta-bound-rot-2} were proven in
\cite{msmfe1} on simplices and in \cite{msmfe2} on quadrilaterals. The
proofs of bounds \eqref{theta-bound-str-pr}--\eqref{theta-bound-pr}
for the two element types are similar to the respective proofs of
\eqref{theta-bound-str}. Bounds \eqref{er-bound-st}
and \eqref{er-bound-rot} were shown in \cite{msmfe2} on quadrilaterals.
Their proof on simplices is similar, using \eqref{q-err-const-tensor}.
\end{proof}

\begin{remark}
We note that, since the $\BDM_1$ space on quadrilaterals involves
quadratic terms, the quadrature bounds \eqref{theta-bound-vel},
\eqref{theta-bound-str}, and \eqref{theta-bound-rot-2} require
restricting one of the test functions to the $\RT_0$ space, which also leads
to the additional error terms in
\eqref{er-bound-vel}--\eqref{er-bound-rot}.  This restriction is not
necessary on simplices, where $\BDM_1$ is the space of linear
polynomials. In order to present a unified convergence proof for
simplices and quadrilaterals, we make the restriction to $\RT_0$ on
simplices as well.  A simplified proof without this
restriction on simplices is also possible, following the approaches
in \cite{wheeler2006multipoint} and \cite{msmfe1}.
\end{remark}

The above bounds are stated on an element $E \in \Tc_h$. In the convergence proof they will be used
by summing over all elements. We will assume that $\|K^{-1}\|_{1,\infty,E}$ and
$\|A\|_{1,\infty,E}$ are uniformly bounded independently of $h$ and will denote this space
by $W^{1,\infty}_{\Tc_h}$.

\subsection{Main convergence result}

\begin{theorem}\label{thm:error}
  If $A \in W^{1,\infty}_{\Tc_h}$, $K^{-1} \in W^{1,\infty}_{\Tc_h}$, and 
  the solution of \eqref{eq:cts1}--\eqref{eq:cts5} is sufficiently smooth, then
  there exists a positive constant $C$ independent of $h$ and $c_0$, such that 
the solution of \eqref{eq:msfmfe1}--\eqref{eq:msfmfe5} satisfies
\begin{align}
&\|\s-\s_h\|_{L^{\infty}(0,T;\Hdiv)} +\|u-u_h\|_{L^{\infty}(0,T;L^2(\O))} + \|\g-\g_h\|_{L^{\infty}(0,T;L^2(\O))}
  +\|z-z_h\|_{L^{\infty}(0,T;L^2(\O))} \nonumber \\
  &\qquad  + \|p-p_h\|_{L^{\infty}(0,T;L^2(\O))}
  +\|\s-\s_h\|_{L^{2}(0,T;\Hdiv)} +\|u-u_h\|_{L^{2}(0,T;L^2(\O))} \nonumber \\
&\qquad
  + \|\g-\g_h\|_{L^{2}(0,T;L^2(\O))} +\|z-z_h\|_{L^{2}(0,T;\Hdiv)} + \|p-p_h\|_{L^{2}(0,T;L^2(\O))} \nonumber \\
  &\quad
  \leq Ch \Big( \|\s\|_{H^1(0,T;H^1(\O))} 
+ \|\dvr \s\|_{L^\infty(0,T;H^1(\O))} + \|\dvr \s\|_{L^2(0,T;H^1(\O))} \nonumber \\  
&\qquad
+  \|u\|_{L^2(0,T;H^1(\O))}  +\|u\|_{L^{\infty}(0,T;H^1(\O))}
+\|\g\|_{H^1(0,T;H^1(\O))} 
\nonumber \\
&\qquad
+ \|z\|_{H^1(0,T;H^1(\O))}
+ \|\dvr z\|_{L^2(0,T;H^1(\O))}
+\|p\|_{H^1(0,T;H^1(\O))}  \Big).  \label{error}
\end{align}
\end{theorem}

\begin{proof}
The derivation of the error bounds follows the structure of the
stability analysis. It involves special manipulation of the error
system, combined with estimation of the approximation errors and the
quadrature errors. We form the error system by subtracting the
discrete problem \eqref{eq:msfmfe1}--\eqref{eq:msfmfe5} from the
continuous one \eqref{eq:cts1}--\eqref{eq:cts5}:
\begin{align}
  &\inp[A(\s + \a p I)]{\t}
  -\inp[A(\s_h + \a p_h I)]{\t}_Q
  + \inp[u-u_h]{\dvr{\t}} + \inp[\g]{\t}-\inp[\g_h]{\t}_Q = 0,
  && \forall \t \in \X_h, \label{er-eq:1_1}\\
&\inp[\dvr(\s-\s_h)]{v} = 0, && \forall v \in V_h, \label{er-eq:1_2}\\
&\inp[ \s]{\xi}-\inp[ \s_h]{\xi}_Q = 0, && \forall \xi \in \Q_h, \label{er-eq:1_3}\\
  &\inp[\K z]{\zeta} -\inp[\K z_h]{\zeta}_Q- \inp[p-p_h]{\dvr{\zeta}} =0, &&
  \forall \zeta \in Z_h, \label{er-eq:1_4}\\
  &\inp[c_0\dt( p- p_h)]{w} + \a\inp[\dt A( \s + \a p I)]{w I}
  -\a\inp[\dt A(\s_h + \a p_h I)]{w I}_Q  \nonumber \\
& \qquad\qquad\qquad\qquad + \inp[\dvr( z-z_h)]{w}  = 0, && \forall w \in W_h. \label{er-eq:1_5}
\end{align}
We split the errors into approximation and discrete errors as follows:
\begin{align*}
	&\s-\s_h = (\s-\Pi \s) + (\Pi \s -\s_h):= \psi_{\s} + \phi_{\s}, \\
	&u-u_h = (u- Q^0 u) + (Q^0 u -u_h):= \psi_u + \phi_u, \\
	&\g-\g_h = (\g- Q^1 \g) + (Q^1 \g -\g_h):= \psi_{\g} + \phi_{\g}, \\
	&z-z_h = (z- \Pi z) + (\Pi z -z_h):= \psi_z + \phi_z, \\
	&p-p_h = (p- Q^0 p) + (Q^0 p -p_h):= \psi_p + \phi_p. 
\end{align*}
We first manipulate the error
system \eqref{er-eq:1_1}--\eqref{er-eq:1_5} to obtain error terms
that can be bounded using either the orthogonality and approximation
properties of the projection operators,
\eqref{const-proj}--\eqref{div-prop} and
\eqref{approx-1}--\eqref{approx-5}, or the estimates for the
quadrature error terms, \eqref{theta-bound-vel}--\eqref{er-bound-rot}.
We rewrite the first equation \eqref{er-eq:1_1} in the following way:
\begin{align*}
&\inp[A(\phi_{\s} + \a\phi_p I)]{\t}_Q + \inp[\phi_u]{\dvr{\t}} + \inp[\phi_{\g}]{\t}_Q \\
  &\qquad =
  - \inp[A(\s + \a pI)]{\t} 
 + \inp[A(\Pi\s + \a Q^0pI)]{\t}_Q
  +\inp[\psi_u]{\dvr{\t}} +\inp[Q^1\g]{\t}_Q-\inp[\g]{ \t} .
\end{align*}
It follows from \eqref{const-proj} that $\inp[\psi_u]{\dvr{\t}} =0$.
With the goal to use a test function $\Pi^0\t$, which is needed to bound the
quadrature error, we manipulate the rest of the terms as follows:
\begin{align}
  &\inp[A(\phi_{\s} + \a\phi_p I)]{\t}_Q + \inp[\phi_u]{\dvr{\t}} 
  + \inp[\phi_{\g}]{\t}_Q \nonumber\\
	&\quad=-\inp[A(\s+\a p I)]{\t-\Pi^0\t}- \inp[A(\psi_{\s} +\a\psi_pI)]{\Pi^0\t} 
	-\inp[A(\Pi\s+\a Q^0pI)]{\Pi^0\t}\nonumber\\
	&\quad\quad+ \inp[A(\Pi\s+\a Q^0pI)]{\Pi^0\t}_Q
        +\inp[A(\Pi\s+\a Q^0pI)]{\t-\Pi^0\t}_Q\nonumber\\
        &\quad\quad -\inp[\g]{ \t-\Pi^0 \t} - \inp[\psi_{\g}]{ \Pi^0 \t}
	-\inp[Q^1\g]{ \Pi^0\t} +\inp[Q^1\g]{ \Pi^0 \t}_Q
        +\inp[Q^1\g]{\t-\Pi^0 \t}_Q. \label{er-eq:2_1_1}
\end{align}
Taking $\t-\Pi^0\t$ as a test function in \eqref{eq:cts1} and using \eqref{prop-5}, we obtain
\begin{align}
\inp[A(\s+\a p I)]{\t-\Pi^0\t}  + \inp[\g]{\t-\Pi^0 \t} =0. \label{er-eq:2_1_2}
\end{align}
Combining \eqref{er-eq:2_1_1}--\eqref{er-eq:2_1_2} and using the quadrature error notation, we get
\begin{align}
  &\inp[A(\phi_{\s} + \a\phi_p I)]{\t}_Q 
  + \inp[\phi_u]{\dvr{\t}} + \inp[\phi_{\g}]{\t}_Q \nonumber\\
  &\quad= - \inp[A(\psi_{\s} +\a\psi_pI)]{\Pi^0\t}- \inp[\psi_{\g}]{ \Pi^0 \t}
  -\theta\inp[A(\Pi\s + \a Q^0pI)]{\Pi^0\t} - \theta\inp[Q^1\g]{ \Pi^0\t} \nonumber\\
  &\quad\quad
    +\inp[A(\Pi\s+\a Q^0pI)]{\t-\Pi^0\t}_Q
	+\inp[Q^1\g]{\t-\Pi^0 \t}_Q. \label{er-eq:2_1_3}
\end{align}
We proceed with the manipulation of the rest of the equations in the error system
\eqref{er-eq:1_1}--\eqref{er-eq:1_5}. Using \eqref{div-prop} and taking
$v = J_E \, \dvr \phi_\s$ on each $E \in \Tc_h$, the second error equation
\eqref{er-eq:1_2} implies
\begin{align}
	\dvr \phi_{\s} = 0. \label{er-eq:2_2_1}
\end{align}
We rewrite the third error equation \eqref{er-eq:1_3} as
\begin{align}
\inp[\psi_{\s}]{\xi} + \theta\inp[\Pi \s]{\xi} + \inp[\phi_{\s}]{\xi}_Q = 0. \label{er-eq:2_2_2}
\end{align}
We rewrite the Darcy's law error equation \eqref{er-eq:1_4} in a way similar to
\eqref{er-eq:2_1_1}--\eqref{er-eq:2_1_3}:
\begin{align*}
  \inp[K^{-1}\phi_z]{\zeta}_Q - \inp[\phi_p]{\dvr \zeta}
  &=
    -\inp[K^{-1}z]{\zeta-\Pi^0\zeta}
  -\inp[K^{-1}(z-\Pi z)]{\Pi^0 \zeta} -\inp[K^{-1}\Pi z]{\Pi^0\zeta} \\
  &\quad+\inp[K^{-1}\Pi z]{\Pi^0 \zeta}_Q +\inp[K^{-1}\Pi z]{\zeta-\Pi^0 \zeta}_Q
+ \inp[\psi_p]{\dvr \zeta}.
\end{align*}
Using \eqref{const-proj}, we have that $\inp[\psi_p]{\dvr \zeta}=0$. Also,
testing \eqref{eq:cts4} with $\zeta-\Pi^0 \zeta$ yields $\inp[K^{-1}z]{\zeta-\Pi^0\zeta} =0$,
hence, we have
\begin{align}
  \inp[K^{-1}\phi_z]{\zeta}_Q-\inp[\phi_p]{\dvr \zeta} =
  &-\inp[K^{-1}\psi_z]{\Pi^0 \zeta} -\theta\inp[K^{-1}\Pi z]{\Pi^0\zeta} 
 +\inp[K^{-1}\Pi z]{\zeta-\Pi^0 \zeta}_Q . \label{er-eq:2_4}
\end{align}
Finally, using \eqref{div-prop}, we rewrite the last equation in
the error system, \eqref{er-eq:1_5}, as follows,
\begin{align}
  \inp[c_0\dt\phi_p]{w} &+ \a\inp[\dt A(\phi_{\s} + \a\phi_p I)]{w I}_Q
  + \inp[\dvr\phi_z]{w} \nonumber \\
  & = - \inp[c_0\dt\psi_p]{w}
  -\a\inp[\dt A(\psi_{\s} + \a\psi_p I) ]{w I}
  - \a \theta\inp[\dt A(\Pi \s + \a Q^0pI)]{w I}.  \label{er-eq:2_5}
\end{align}
We next combine the equations and make an appropriate
choice of the test functions. In particular, we differentiate
\eqref{er-eq:2_1_3} in time, set $\t =\phi_{\s},\, \xi = \dt\phi_{\g},\, \zeta
=\phi_z,\, w=\phi_p$, and combine \eqref{er-eq:2_1_3}--\eqref{er-eq:2_5}:
\begin{align}
  \frac{1}{2}&\dt\left(\|A^{1/2}(\phi_{\s}+\a \phi_p I)\|^2_Q
  +\|c_0^{1/2}\phi_p\|^2 \right) +\| K^{-1/2}\phi_z\|^2_Q  \nonumber \\
  & = - \inp[c_0\dt\psi_p]{\phi_p}
 - \inp[\dt A(\psi_{\s} +\a\psi_pI)]{\Pi^0\phi_{\s} + \a\phi_p I}
  - \inp[\dt \psi_{\g}]{ \Pi^0 \phi_{\s}}
 -\inp[K^{-1}\psi_z]{\Pi^0 \phi_z}
 +\inp[\psi_{\s}]{\dt \phi_{\g}}
  \nonumber \\
  &\quad 
  -\theta\inp[\dt A(\Pi\s + \a Q^0pI)]{\Pi^0\phi_{\s}+\a\phi_p I}
  - \theta\inp[\dt Q^1\g]{ \Pi^0\phi_{\s}}  
  -\theta\inp[K^{-1}\Pi z]{\Pi^0\phi_z}
  +  \theta\inp[\Pi \s]{ \dt \phi_{\g}}
  \nonumber \\
  &\quad
  +\inp[\dt A(\Pi\s+\a Q^0pI)]{\phi_{\s}-\Pi^0\phi_{\s}}_Q 
  +\inp[\dt Q^1\g]{\phi_{\s}-\Pi^0 \phi_{\s}}_Q
  +\inp[K^{-1}\Pi z]{\phi_z-\Pi^0 \phi_z}_Q,
  \label{er-eq:3}
\end{align}
where we have listed first the terms involving approximation error, followed by quadrature error
terms, and the three extra terms arising from the use of operator $\Pi^0$. We note that there are
two terms involving $\dt \phi_{\g}$, which will be handled by integration by parts after time
integration.
We proceed by deriving bounds for the rest of the terms appearing on the right-hand side.
For the approximation error terms, using  \eqref{prop-5} and \eqref{approx-1}--\eqref{approx-3},
we have
\begin{align}
  &\left| \inp[c_0\dt\psi_p]{\phi_p} + \inp[\dt A(\psi_{\s} +\a\psi_pI)]{\Pi^0\phi_{\s} + \a\phi_p I}
  + \inp[\dt \psi_{\g}]{ \Pi^0 \phi_{\s}}
 +\inp[K^{-1}\psi_z]{\Pi^0 \phi_z} \right| \nonumber \\
 &\qquad \qquad \leq Ch^2(\|\dt \s\|^2_1 +\|\dt p\|^2_1  +\|\dt \g\|^2_1 +\|z\|^2_1)
 +\epsilon_1(\|\phi_{\s}\|^2 + \|\phi_p\|^2+\|\phi_z\|^2).\label{er-eq:4_1}
\end{align}
For the quadrature error terms, applying \eqref{theta-bound-vel}--\eqref{theta-bound-pr}
and \eqref{h1-continuity-bdm}--\eqref{l2-continuity-q0} results in
\begin{align}
  &\left|\theta\inp[\dt A(\Pi\s + \a Q^0pI)]{\Pi^0\phi_{\s}+\a\phi_p I}
  + \theta\inp[\dt Q^1\g]{ \Pi^0\phi_{\s}}  
  + \theta\inp[K^{-1}\Pi z]{\Pi^0\phi_z}  \right| \nonumber \\
  &\qquad \qquad \leq Ch^2(\|\dt \s\|_1^2 + \|\dt p\|_1^2 + \|\dt \gamma\|_1^2  +\|z\|_1^2)
  +\epsilon_1(\|\phi_{\s}\|^2 + \|\phi_p\|^2 +\|\phi_z\|^2). \label{er-eq:4_2}
\end{align}
For the last three terms in \eqref{er-eq:3}, due to \eqref{er-bound-vel}--\eqref{er-bound-rot}
and \eqref{h1-continuity-bdm}--\eqref{h1-continuity-q1}, we obtain
\begin{align}
  & \left| \inp[\dt A(\Pi\s+\a Q^0pI)]{\phi_{\s}-\Pi^0\phi_{\s}}_Q
  +\inp[\dt Q^1\g]{\phi_{\s}-\Pi^0 \phi_{\s}}_Q
  +\inp[K^{-1}\Pi z]{\phi_z-\Pi^0 \phi_z}_Q  \right|
  \nonumber \\
  & \qquad \qquad \leq Ch^2(\|\dt \s\|^2_1 +\|\dt p\|^2_1 + \|\dt \g\|_1^2+\|z\|_1^2)
  +\epsilon_1(\|\phi_{\s}\|^2+\|\phi_z\|^2). \label{er-eq:4_3}
\end{align}
Next, we combine \eqref{er-eq:3}--\eqref{er-eq:4_3} and integrate in time from $0$
to an arbitrary $t\in (0,T]$:
\begin{align}
  & \|A^{1/2}(\phi_{\s}+\a \phi_p I)(t)\|^2_Q + \|c_0^{1/2}\phi_p(t)\|^2
  +\int_0^t\|K^{-1/2}\phi_z\|^2_Q\, ds \nonumber \\
  & \quad
  \leq \int_0^t \left(\inp[\psi_{\s}]{\dt \phi_{\g}} + \theta\inp[\Pi\s]{ \dt \phi_{\g}}\right) ds
  + \epsilon_1\int_0^t(\|\phi_{\s}\|^2 + \|\phi_p\|^2+\|\phi_z\|^2)\, ds \nonumber \\
  &\qquad+ Ch^2\int_0^t(\|\dt \s\|^2_1 +\|\dt p\|^2_1 +\|\dt \g\|^2_1+\|z\|^2_1)\, ds
  + \|A^{1/2}(\phi_{\s}+\a \phi_p I)(0)\|^2_Q
  + \|c_0^{1/2}\phi_p(0)\|^2. \label{er-eq:4}
\end{align}
For the first two terms on the right-hand side we use integration by parts:
\begin{align}
  &  \int_0^t \left(\inp[\psi_{\s}]{\dt \phi_{\g}}
  + \theta\inp[\Pi\s]{ \dt \phi_{\g}}\right) ds \nonumber \\
& \quad
  =    - \int_0^t \left(\inp[ \dt \psi_{\s}]{ \phi_{\g}}
  +  \theta\inp[ \dt\Pi \s]{  \phi_{\g}} \right) ds 
  +\inp[\psi_{\s}]{ \phi_{\g}}(t)
  + \theta\inp[\Pi \s]{  \phi_{\g}}(t)
- \inp[\psi_{\s}]{\phi_{\g}}(0) 
  - \theta\inp[\Pi \s]{ \phi_{\g}}(0) \nonumber \\
  & \quad
  \leq \epsilon_1 \left( \|\phi_{\g}(t)\|^2 + \int_0^t \|\phi_{\g}\|^2 ds \right)
  + C \|\phi_{\g}(0)\|^2
  + Ch^2 \left( \| \s(t)\|^2_1  + \| \s(0)\|^2_1 + \int_0^t \|\dt \s \|^2_1  \, ds\right).
  \label{ibp-term}
\end{align}
where we used \eqref{approx-3}, \eqref{theta-bound-rot}, and \eqref{h1-continuity-bdm}
in the last step. We proceed with bounding the terms involving
$\|\phi_{\s}\|$, $\|\phi_p\|$, $\|\phi_z\|$, and $\|\phi_{\g}\|$ that appear on the right-hand sides
of \eqref{er-eq:4} and \eqref{ibp-term}. Using the elasticity inf-sup condition \eqref{inf-sup-elast}
together with \eqref{er-eq:1_1}, we get
\begin{align}
	\|\phi_u\| + \|\phi_{\g}\| \leq &C\sup\limits_{0\neq\t \in \X_h} \frac{\inp[\phi_u]{\dvr{\t}} + \inp[\phi_{\g}]{\t}_Q}{\|\t\|_{\dvr}} \nonumber \\
	&= C\sup\limits_{0\neq \t \in \X_h}
        \frac{\inp[A(\s_h+\api)]{\t}_Q-\inp[A(\s+\a pI)]{\t}
          + \inp[Q^1\g]{ \t}_Q-\inp[\g]{ \t}}{\|\t\|_{\dvr}}.  \label{er-eq:7}
\end{align}
Using manipulations as in \eqref{er-eq:2_1_1}--\eqref{er-eq:2_1_3},
along with the bounds \eqref{approx-1}--\eqref{approx-3},
\eqref{theta-bound-str}, \eqref{theta-bound-rot} and
\eqref{er-bound-st}--\eqref{er-bound-rot}, we have
\begin{align}
	\inp[A(\s_h+\api)]{\t}&_Q-\inp[A(\s+\a pI)]{\t} +\inp[Q^1\g]{ \t}-\inp[\g]{ \t}_Q \nonumber\\
	&=-\inp[A(\phi_{\s}+\a\phi_pI)]{\t}_Q  - \inp[A(\psi_{\s} +\a\psi_pI)]{\Pi^0\t}- \inp[A(\Pi \s + \a Q^0 p I)]{\t - \Pi^0\t}_Q\nonumber \\
	&+\theta\inp[A(\Pi\s+\a Q^0pI)]{\Pi^0\t}_Q +\inp[Q^1\g]{\t-\Pi^0 \t}_Q - \theta\inp[Q^1 \g]{\Pi^0\t} -\inp[\psi_{\g}]{\Pi^0\t} \nonumber\\
	&\leq C \left( h(\|\s\|_1+\|p\|_1+\|\g\|_1)
        + \|A^{1/2}(\phi_{\s}+\a \phi_p I)\|\right) \|\t\|.\label{er-eq:8}
\end{align}
Combining \eqref{er-eq:7} and \eqref{er-eq:8}, we obtain
\begin{align}\label{inf-sup-elast-1}
  \|\phi_u\| + \|\phi_{\g}\| \leq Ch(\|\s\|_1+\|p\|_1+\|\g\|_1)
  + C\|A^{1/2}(\phi_{\s}+\a \phi_p I)\|, 
\end{align}
as well as
\begin{align}\label{inf-sup-elast-2}
  \int_0^t \left(\|\phi_u\|^2 + \|\phi_{\g}\|^2 \right) ds \leq C h^2
  \int_0^t \left(\|\s\|_1^2 +\|p\|_1^2 +\|\g\|_1^2 \right) ds
  + C \int_0^t \left(\|\phi_{\s}\|^2  + \|\phi_p\|^2 \right) ds.
\end{align}
For $\|\phi_p\|$, using the fact that $Z_h^0\times W_h$ is a stable Darcy pair, \eqref{er-eq:1_4}
and \eqref{approx-3} and \eqref{theta-bound-vel}, we obtain
\begin{align}
  \|\phi_p\| &\leq C\sup\limits_{0\neq \zeta\in Z^0_h}\frac{\inp[\dvr \zeta]{\phi_p}}{\|\zeta\|_{\dvr}}
  =  C\sup\limits_{0\neq \zeta\in Z^0_h}\frac{\inp[\K z]{\zeta} -\inp[\K z_h]{\zeta}_Q}{\|\zeta\|_{\dvr}} \nonumber \\
  & = C\sup\limits_{0\neq \zeta\in Z^0_h}\frac{\inp[K^{-1}\phi_z]{\zeta}_Q
    -\inp[K^{-1}\psi_z]{\zeta}+\theta\inp[K^{-1}\Pi z]{\zeta}}{\|\zeta\|_{\dvr}}
  \leq C(h\|z\|_1 + \|K^{-1/2}\phi_z\|),\label{er-eq:10_1}
\end{align}
implying
\begin{equation}\label{pres-inf-sup-bound}
  \int_0^t \|\phi_p\|^2 ds \le C \int_0^t \left(h^2 \|z\|_1^2 + \|K^{-1/2}\phi_z\|^2 \right) ds.
  \end{equation}
Finally, to obtain a bound on $\int_0^t \|\phi_{\s}\|^2 ds$, which appears on the right hand side
in \eqref{inf-sup-elast-2}, we choose $\t=\phi_{\s}$ in \eqref{er-eq:2_1_3} and
$\xi = \phi_{\g}$ in \eqref{er-eq:2_2_2} and combine them, using 
also \eqref{er-eq:2_2_1}, to obtain
\begin{align*}
  \|A^{1/2}\phi_{\s}\|_Q^2 & =  -\a\inp[A\phi_p I ]{\phi_{\s}}_Q
  - \inp[A(\psi_{\s} +\a\psi_pI)]{\Pi^0\phi_{\s}}- \inp[\psi_{\g}]{ \Pi^0 \phi_{\s}}
  \nonumber \\
  & \quad\qquad
  -\theta\inp[A(\Pi\s + \a Q^0pI)]{\Pi^0\phi_{\s}}
  -\theta\inp[Q^1\g]{ \Pi^0\phi_{\s}}+\inp[A(\Pi\s+\a Q^0pI)]{\phi_{\s}-\Pi^0\phi_{\s}}_Q
  \nonumber \\
  & \quad\qquad
  +\inp[Q^1\g]{\phi_{\s}-\Pi^0 \phi_{\s}}_Q
  + (\psi_\s,\phi_\g) + \theta(\Pi\s,\phi_\g) \nonumber \\
  & \leq Ch^2(\|\s\|_1^2+\|p\|_1^2+\|\g\|_1^2)
+C\|\phi_p\|^2 +\epsilon_2(\|\phi_{\g}\|^2 + \|\phi_{\s}\|^2),
\end{align*}
where in the last step we used \eqref{prop-5}, \eqref{approx-1}--\eqref{approx-3},
\eqref{theta-bound-str}, \eqref{theta-bound-rot}, \eqref{theta-bound-rot-2},
\eqref{er-bound-st}, and \eqref{er-bound-rot}.
Thus, we have
\begin{align}
  \int_0^t\|\phi_{\s}\|^2\, ds\leq  C h^2 \int_0^t (\|\s\|^2_1+\|p\|^2_1 +\|\g\|^2_1) \,ds
  +C\int_0^t \|\phi_p\|^2 ds + \epsilon_2 \int_0^t \|\phi_{\g}\|^2 ds.
\label{er-eq:11}
\end{align}
Combining \eqref{er-eq:2_2_1}, \eqref{er-eq:4}--\eqref{er-eq:11} and
choosing $\epsilon_2$ small enough, then $\epsilon_1$ small enough,
gives the estimate
\begin{align}
  &\|A^{1/2}(\phi_{\s} +\a \phi_p I)(t)\|^2 +\|\phi_u(t)\|^2
  + \|\phi_{\g}(t)\|^2+\|c_0^{1/2}\phi_p(t)\|^2 + \|\dvr \phi_{\s}\|^2 \nonumber \\
  &\qquad\qquad +\int_0^t \left(
 \|\phi_{\s}\|^2 +  \|\phi_u\|^2 + \|\phi_{\g}\|^2
 + \|K^{-1/2}\phi_z\|^2
  + \|\phi_p\|^2 + \|\dvr \phi_{\s}\|^2 \right) ds \nonumber \\
  &\qquad\leq  C \Big( h^2\int_0^t \left(\|\dt \s\|^2_1 +\|\dt p\|^2_1 +\|\dt \g\|^2_1
  + \| \s\|^2_1 +\| p\|^2_1 + \|\g\|^2_1 + \|z\|^2_1 \right) ds  \nonumber \\
  & \qquad\qquad\qquad
  + h^2\left(\| \s(t)\|^2_1 +\| p(t)\|^2_1 +\|\g(t)\|^2_1 + \|\s(0)\|_1^2\right) \nonumber \\
  & \qquad\qquad\qquad  + \| \phi_\s(0)\|^2 +\| \phi_p(0)\|^2 +\|\phi_\g(0)\|^2 \Big).
  \label{er-eq:14}
\end{align}

\noindent
{\bf Estimate for $\dvr \phi_z$.} We note that \eqref{er-eq:14} is a
self-contained error estimate. Similarly to the stability argument, we
proceed with bounding $\|\dvr \phi_z\|$, obtaining also bounds on
$\|K^{-1/2}\phi_z(t)\|$ and $\|\phi_p(t)\|$ for all $t$.
We choose $w = J_E \, \dvr \phi_z$ on each $E \in \Tc_h$
in  \eqref{er-eq:2_5}, which yields
\begin{align*}
  \| J_E^{1/2} \dvr \phi_z\|^2 = & -(c_0\dt \phi_p, J_E \, \dvr \phi_z)
  -(c_0\dt \psi_p, J_E \, \dvr \phi_z)
  - \a\inp[\dt A(\phi_{\s} + \a \phi_p I)]{J_E(\dvr \phi_z) I}_Q \\
  &
  -\a\inp[\dt A(\psi_{\s} + \a \psi_p I)]{J_E(\dvr \phi_z) I}
  - \a \theta\inp[\dt A(\Pi \s + \a Q^0 p I)]{J_E(\dvr \phi_z) I}.
\end{align*}
Using \eqref{approx-1}, \eqref{approx-3} and \eqref{theta-bound-str}--\eqref{theta-bound-rot},
we obtain
\begin{align}
  \|\dvr \phi_z\| \leq C\left(\|c_0^{1/2}\dt\phi_p\|+ \|\dt A^{1/2} (\phi_{\s}+\a \phi_pI)\|
  + h(\|\dt p\|_1 + \|\dt \s\|_1) \right). \label{er-eq:15}
\end{align}
It remains to bound the first two terms on the right-hand side of \eqref{er-eq:15}.
Similarly to the stability argument, cf.~\eqref{st-eq:13}, we differentiate
\eqref{er-eq:2_1_3}--\eqref{er-eq:2_4} in time,
set $\t =\dt \phi_{\s},\, \xi = \dt\phi_{\g},\, \zeta =\phi_z,\, w= \dt \phi_p$, and combine
\eqref{er-eq:2_1_3}--\eqref{er-eq:2_5}, resulting in a time-differentiated version of
\eqref{er-eq:3}:
\begin{align}
  &\frac{1}{2}\dt\|K^{-1/2}\phi_z\|^2_Q +\|\dt A^{1/2}(\phi_{\s}
  +\a \phi_p I)\|^2_Q +\|c_0^{1/2}\dt\phi_p\|^2 \nonumber\\
  & = - \inp[c_0\dt\psi_p]{\dt\phi_p}
 - \inp[\dt A(\psi_{\s} +\a\psi_pI)]{\dt(\Pi^0\phi_{\s} + \a\phi_p I)}
  - \inp[\dt \psi_{\g}]{ \dt\Pi^0 \phi_{\s}}
  -\inp[\dt K^{-1}\psi_z]{\Pi^0 \phi_z}
  \nonumber \\ 
   &\quad 
+\inp[\dt \psi_{\s}]{\dt \phi_{\g}}
   -\theta\inp[\dt A(\Pi\s + \a Q^0pI)]{\dt(\Pi^0\phi_{\s}+\a\phi_p I)}
  - \theta\inp[\dt Q^1\g]{ \dt\Pi^0\phi_{\s}}  
 \nonumber \\
  &\quad
  -\theta\inp[\dt K^{-1}\Pi z]{\Pi^0\phi_z}
  +  \theta\inp[\dt\Pi \s]{ \dt \phi_{\g}}
  +\inp[\dt A(\Pi\s+\a Q^0pI)]{\dt(\phi_{\s}-\Pi^0\phi_{\s})}_Q
 \nonumber \\
  &\quad  
  +\inp[\dt Q^1\g]{\dt(\phi_{\s}-\Pi^0 \phi_{\s})}_Q
  +\inp[\dt K^{-1}\Pi z]{\phi_z-\Pi^0 \phi_z}_Q.
  \label{er-eq:16}
\end{align}
Before bounding the terms on the right above, we note that we would
like the bounds to be in terms of $\|\dt A^{1/2}(\phi_{\s} +\a \phi_p I)\|$, since we do not
have separate control of $\|\dt\phi_\s\|$ and $\|\dt\phi_p\|$. 
  To this end, we first note that the projector $\Pi^0$ is defined element by element
  and let $\Pi^0_E: H^1(E,\M) \mapsto \X_h^0|_E$ be the local $\RT_0$
  projector on an element $E \in \Tc_h$. Using that
for each $E$, $\a\phi_p I|_E \in \X_h^0|_E$, we have that
$\Pi_E^0(\a\phi_p I) = (\a\phi_p I)|_E$. Then,
for the second and sixth term above we have 
$$
(\Pi^0\phi_{\s} + \a\phi_p I)|_E = \Pi_E^0(\phi_{\s} + \a\phi_p I).
$$
Similarly, for the tenth and eleventh term we have
$$
(\phi_{\s}-\Pi^0 \phi_{\s})|_E = (\phi_{\s} + \a\phi_p I)|_E  - \Pi_E^0(\phi_{\s} + \a\phi_p I).
$$
Also, since $\phi_p I$ is a symmetric matrix, for the third and seventh terms we have
$$
\inp[\dt \psi_{\g}]{ \dt\Pi^0 \phi_{\s}}_E = \inp[\dt \psi_{\g}]{ \dt\Pi_E^0(\phi_{\s} + \a\phi_p I)}_E,
\quad
\theta_E\inp[\dt Q^1\g]{ \dt\Pi^0\phi_{\s}} =
\theta_E\inp[\dt Q^1\g]{ \dt\Pi_E^0(\phi_{\s} + \a\phi_p I)}.
$$
Now, noting that the terms on the right in \eqref{er-eq:16} can be expressed as sums
over mesh elements, we use the above identities and bound these terms as
in \eqref{er-eq:4_1}--\eqref{er-eq:4_3}:
\begin{align}
  &\left| \inp[c_0\dt\psi_p]{\dt\phi_p} + \inp[\dt A(\psi_{\s} +\a\psi_pI)]{\dt(\Pi^0\phi_{\s}
    + \a\phi_p I)}
  + \inp[\dt \psi_{\g}]{ \dt\Pi^0 \phi_{\s}} \right. \nonumber \\
  & \qquad\qquad \left.
  +\inp[\dt K^{-1}\psi_z]{\Pi^0 \phi_z}
  +\inp[\dt \psi_{\s}]{\dt \phi_{\g}}\right|
  \nonumber \\
  &\qquad \leq Ch^2(\|\dt \s\|^2_1 +\|\dt p\|^2_1  +\|\dt \g\|^2_1 +\|\dt z\|^2_1) \nonumber \\
  & \qquad\qquad
  +\epsilon(\|c_0^{1/2}\dt\phi_p\|^2 + \|\dt A^{1/2}(\phi_{\s} + \a\phi_p I)\|^2
  + \|\dt\phi_\g\|^2 +\|\phi_z\|^2),\label{er-eq:16_1}
\end{align}
\begin{align}
  &\left|\theta\inp[\dt A(\Pi\s + \a Q^0pI)]{\dt(\Pi^0\phi_{\s}+\a\phi_p I)}
  + \theta\inp[\dt Q^1\g]{ \dt\Pi^0\phi_{\s}}  
  + \theta\inp[\dt K^{-1}\Pi z]{\Pi^0\phi_z}
  + \theta\inp[\dt\Pi \s]{ \dt \phi_{\g}} \right| \nonumber \\
  &\quad \leq Ch^2(\|\dt \s\|_1^2 +\|\dt p\|_1^2 +\|\dt \g\|_1^2  +\|\dt z\|_1^2)
  +\epsilon(\|\dt A^{1/2}(\phi_{\s} + \a\phi_p I)\|^2
  + \|\dt\phi_\g\|^2 +\|\phi_z\|^2), \label{er-eq:16_2}
\end{align}
\begin{align}
  & \left| \inp[\dt A(\Pi\s+\a Q^0pI)]{\dt(\phi_{\s}-\Pi^0\phi_{\s})}_Q
  +\inp[\dt Q^1\g]{\dt(\phi_{\s}-\Pi^0 \phi_{\s})}_Q
  +\inp[\dt K^{-1}\Pi z]{\phi_z-\Pi^0 \phi_z}_Q  \right|
  \nonumber \\
  & \quad \leq Ch^2(\|\dt \s\|^2_1 +\|\dt p\|^2_1 + \|\dt \g\|_1^2+\|\dt z\|_1^2)
  +\epsilon(\|\dt A^{1/2}(\phi_{\s} + \a\phi_p I)\|^2 + \|\phi_z\|^2). \label{er-eq:16_3}
\end{align}
Combining \eqref{er-eq:16}--\eqref{er-eq:16_3}, taking $\epsilon$
small enough, and integrating in time, we get
\begin{align}
  & \|K^{-1/2}\phi_z(t)\|^2_Q +\int_0^t\left(\|\dt A^{1/2}(\phi_{\s}
  +\a \phi_p I)\|^2_Q +\|c_0^{1/2}\dt\phi_p\|^2\right) ds \nonumber \\
  & \qquad \leq \|K^{-1/2}\phi_z(0)\|_Q^2 +
  \epsilon\int_0^t \left( \| \dt \phi_{\g}\|^2  + \|\phi_z\|^2 \right) ds\nonumber\\
  &\qquad \quad + Ch^2\int_0^t \left(\|\dt \s\|^2_1 +\|\dt p\|^2_1
  +\|\dt \g\|^2_1 +\|\dt z\|^2_1 \right) ds. \label{er-eq:17}
\end{align}
Similarly to \eqref{inf-sup-elast-1}, the elasticity inf-sup condition \eqref{inf-sup-elast},
differentiated in time, implies
\begin{align}\label{er-eq:20_1}
  \int_0^t \left(\|\dt \phi_u\|^2 + \|\dt \phi_{\g}\|^2 \right) ds \leq C h^2
  \int_0^t \left(\|\dt\s\|_1^2 +\|\dt p\|_1^2 +\|\dt \g\|_1^2 \right) ds
  + C \int_0^t \|\dt A^{1/2}(\phi_{\s} + \a\phi_p I)\|^2  ds.
\end{align}
Combining \eqref{er-eq:17}--\eqref{er-eq:20_1} with \eqref{er-eq:10_1}, we conclude that
\begin{align}
  & \|K^{-1/2}\phi_z(t)\|^2  + \| \phi_p(t)\|^2
  +\int_0^t\left(\|\dt A^{1/2}(\phi_{\s}+\a \phi_p I)\|^2 + \|c_0^{1/2}\dt\phi_p\|^2
    \right)\, ds \nonumber \\
	& \qquad \qquad \leq  \epsilon\int_0^t \|\phi_z\|^2 ds + Ch^2\|z(t)\|^2 \nonumber\\
    &\qquad \qquad \qquad + Ch^2\int_0^t \left(\|\dt \s\|^2_1 +\|\dt p\|^2_1
    +\|\dt \g\|^2_1 +\|\dt z\|^2_1 \right) ds. \label{er-eq:21}
\end{align}
Therefore, \eqref{er-eq:15} and \eqref{er-eq:21} give
\begin{align}
  &\|K^{-1/2}\phi_z(t)\|^2_Q  + \| \phi_p(t)\|^2+\int_0^t\|\dvr \phi_z\|^2 ds
  \leq \epsilon\int_0^t\|\phi_z\|^2\, ds \nonumber \\ 
	&\quad\quad\quad+ Ch^2\left(\int_0^t(\|\dt z\|_1^2+\|\dt\s\|_1^2+\|\dt p\|_1^2+\|\dt\g\|_1^2)\, ds  
	+ \|z(t)\|_1^2\right). \label{er-eq:22}
\end{align}
We also note that
\begin{align}
  \|\phi_{\s}\| \le C\left(\|A^{1/2}(\phi_{\s}+\a \phi_pI )\| + \|\phi_p \|\right). \label{er-eq:23}
\end{align}
Finally, combining \eqref{er-eq:14}, \eqref{er-eq:22} and \eqref{er-eq:23}, we obtain
\begin{align}
  &\|A^{1/2}(\phi_{\s} +\a \phi_p I)(t)\|^2 + \| \phi_{\s}(t)\|^2_{\dvr} +\|\phi_u(t)\|^2
  + \|\phi_{\g}(t)\|^2  + \|K^{-1/2}\phi_z(t)\|^2  + \| \phi_p(t)\|^2\nonumber\\
  & \qquad +\int_0^t \left(
\|\phi_{\s}\|^2_{\dvr}
  +  \|\phi_u\|^2 + \|\phi_{\g}\|^2 
+   \|K^{-1/2}\phi_z\|^2 + \|\dvr \phi_z\|^2  + \|\phi_p\|^2
  \right)  \nonumber \\
  & \leq
C \Big( h^2\int_0^t \left(\|\dt \s\|^2_1 +\|\dt p\|^2_1 +\|\dt \g\|^2_1+\|\dt z\|^2_1
  + \| \s\|^2_1 +\| p\|^2_1 + \|\g\|^2_1 + \|z\|_1^2 \right) ds  \nonumber \\
  & \qquad\qquad
  + h^2\left(\| \s(t)\|^2_1 +\| p(t)\|^2_1 +\|\g(t)\|^2_1 + \|z(t)\|_1^2 +
  \|\s(0)\|_1^2\right) \nonumber \\
  & \qquad\qquad
  + \| \phi_\s(0)\|^2 +\| \phi_p(0)\|^2 +\|\phi_\g(0)\|^2 + \|\phi_z(0)\|^2 \Big).
  \label{final-phi-bound}
\end{align}
For the initial error, we recall that the discrete initial data is taken to be the elliptic
projection of the continuous initial data, see \eqref{ell-proj}. Then, similarly to
\eqref{init-data-bound}, we have
\begin{align}\label{initial-error}
  \| \phi_\s(0)\| +\| \phi_p(0)\| +\|\phi_\g(0)\| + \|\phi_z(0)\| \le
  C (\| \psi_\s(0)\| + \| \psi_p(0)\| +\|\psi_\g(0)\| + \|\psi_z(0)\| + \|\psi_u(0)\|).
\end{align}
Bounds \eqref{final-phi-bound}--\eqref{initial-error}, combined with the use of the
triangle inequality and the approximation bounds
\eqref{approx-1}--\eqref{approx-5}, imply the assertion of the theorem.
\end{proof}

\section{Fully-discrete MSMFE--MFMFE method}\label{sec:discr}

In this section we present the fully-discrete method based on the
backward Euler time discretization and show how the algebraic system
at each time step can be reduced to a positive definite
cell-centered displacement-pressure system.

Let $0 = t_0 < t_1 < \cdots < t_N = T$ be a partition of the time
interval $[0,T]$ with time steps $\Delta t_n = t_n - t_{n-1}$, $n =
1,\ldots,N$, $\Delta t = \max_{1\le n \le N} \Delta t_n$.
Let $\varphi^n = \varphi(t_n)$ and $\dt^n \varphi =
(\varphi^n - \varphi^{n-1})/\Delta t_n$. For a Banach space $H$ on $\Omega$ with a
norm $\|\cdot\|_H$, we introduce the discrete-in-time norms
$$
\|\varphi\|_{l^2(0,T;H)} := \left(\sum_{n=1}^N \Delta t_n \|\varphi\|^2_H\right)^{\frac12}, \quad
\|\varphi\|_{l^\infty(0,T;H)} := \max_{0 \le n \le N} \|\varphi\|_H.
$$
The fully-discrete MSMFE--MFMFE method is: given compatible initial data
$(\sigma^0_h,u^0_h,\gamma^0_h,z^0_h,p^0_h)$, find,
for $n = 1,\ldots,N$, 
$(\sigma^n_h,u^n_h,\gamma^n_h,z^n_h,p^n_h) \in \X_h \times V_h \times
\Q_h \times Z_h \times W_h$ such that
\begin{align}
	  &\inp[A(\sigma^n_h + \a p^n_h I)]{\tau}_Q + \inp[u^n_h]{\dvr{\tau}}
          + \inp[\gamma^n_h]{\tau}_Q = 0 , &&
          \forall \tau \in \X_h, \label{eq:discr1}\\
	  &- \inp[\dvr{\sigma^n_h}]{v} = \inp[f^n]{v}, && \forall v \in V_h,
          \label{eq:discr2}\\
	  &\inp[\sigma^n_h]{\xi}_Q  = 0, && \forall \xi \in \Q_h,
          \label{eq:discr3}\\
	  &\inp[\K z^n_h]{\zeta}_Q - \inp[p^n_h]{\dvr{\zeta}}  = 0
          && \forall \zeta \in Z_h,
          \label{eq:discr4}\\
	&\inp[c_0\dt^n{p_h}]{w} + \a\inp[\dt^n A(\sigma_h + \a p_h I)]{wI}_Q
          + \inp[\dvr{z^n_h}]{w}   = \inp[q^n]{w}, && \forall w \in W_h.
          \label{eq:discr5}
\end{align}

\begin{lemma}\label{well-posed-fully-discrete}
  The fully discrete method \eqref{eq:discr1}--\eqref{eq:discr5}
  has a unique solution.
\end{lemma}
\begin{proof}
  The assertion of the lemma follows from the solvability of the
  resolvent system \eqref{eq:msfmfe1-0}--\eqref{eq:msfmfe5-0} shown in the proof of
  Theorem~\ref{well-posed}.
\end{proof}

The following convergence theorem can be proved using the framework in the proof of
Theorem~\ref{thm:error}, combined with standard tools for treating the discrete time
derivatives. 

\begin{theorem}\label{thm:error-discr}
  If $A \in W^{1,\infty}_{\Tc_h}$, $K^{-1} \in W^{1,\infty}_{\Tc_h}$, and 
  the solution of \eqref{eq:cts1}--\eqref{eq:cts5} is sufficiently smooth, then
  there exists a positive constant $C$ independent of $h$ and $c_0$, such that 
the solution of \eqref{eq:discr1}--\eqref{eq:discr5} satisfies
\begin{align}
&\|\s-\s_h\|_{l^{\infty}(0,T;\Hdiv)} +\|u-u_h\|_{l^{\infty}(0,T;L^2(\O))} + \|\g-\g_h\|_{l^{\infty}(0,T;L^2(\O))}
  +\|z-z_h\|_{l^{\infty}(0,T;L^2(\O))} \nonumber \\
  &\qquad  + \|p-p_h\|_{l^{\infty}(0,T;L^2(\O))}
  +\|\s-\s_h\|_{l^{2}(0,T;\Hdiv)} +\|u-u_h\|_{l^{2}(0,T;L^2(\O))} \nonumber \\
&\qquad
  + \|\g-\g_h\|_{l^{2}(0,T;L^2(\O))} +\|z-z_h\|_{l^{2}(0,T;\Hdiv)} + \|p-p_h\|_{l^{2}(0,T;L^2(\O))} \nonumber \\
  &\quad
  \leq Ch \Big( \|\s\|_{H^1(0,T;H^1(\O))} 
+ \|\dvr \s\|_{L^\infty(0,T;H^1(\O))} + \|\dvr \s\|_{L^2(0,T;H^1(\O))} \nonumber \\  
&\qquad
+  \|u\|_{L^2(0,T;H^1(\O))}  +\|u\|_{L^{\infty}(0,T;H^1(\O))}
+\|\g\|_{H^1(0,T;H^1(\O))} 
\nonumber \\
&\qquad
+ \|z\|_{H^1(0,T;H^1(\O))}
+ \|\dvr z\|_{L^2(0,T;H^1(\O))}
+\|p\|_{H^1(0,T;H^1(\O))} \Big) \nonumber \\
&\qquad
+ C \Delta t \left(\|\s\|_{H^2(0,T;L^2(\O))} + \|u\|_{H^2(0,T;L^2(\O))} + \|\g\|_{H^2(0,T;L^2(\O))}
+ \|p\|_{H^2(0,T;L^2(\O))} \right).  \label{error-discr}
\end{align}
\end{theorem}
%

\subsection{Reduction to a cell-centered displacement-pressure system}

The vertex quadrature rule applied to the stress and velocity bilinear
forms, $\inp[A\sigma^n_h]{\tau}_Q$ in \eqref{eq:discr1} and $\inp[\K
  z^n_h]{\zeta}_Q$ in \eqref{eq:discr4}, respectively. results in the
corresponding matrices $\As$ and $\Az$ being block-diagonal with blocks
associated with the mesh vertices. More precisely, consider any
interior vertex $\r$ shared by $k$ edges or faces $e_1,\dots,e_k$ as
shown in Figure~\ref{fig:0}.  Let $\zeta_1,\ldots,\zeta_k$ be the
velocity degrees of freedom associated with the vertex and let
$z_1,\ldots,z_k$ be the corresponding normal velocity values, see
Figure~\ref{fig:0_1}. For the sake of visualization, the normal
velocities are drawn at a distance from the vertex. The vertex
quadrature rule $(K^{-1}\cdot,\cdot)_Q$ localizes the interaction of
basis functions around each vertex by decoupling them from the the
rest of the basis functions, so taking $\zeta_1,\ldots,\zeta_k$ in
\eqref{eq:discr4} results in a local $k \times k$ linear system.
Therefore $\Az$ is block-diagonal with $k \times k$ blocks associated
with mesh vertices. Similarly, $\As$ is block-diagonal with $d\,k
\times d\,k$ blocks, see Figure~\ref{fig:0_2}. Due to the positive
definiteness of $A$ and $K$ and Lemma~\ref{quad-coercive}, the blocks
of $\As$ and $\Az$ are symmetric and positive definite. Therefore the
velocity and stress can be easily eliminated by solving small local
linear systems. Moreover, the rotation can be further eliminated as
follows. Let $\Asg$ be the matrix corresponding to
$\inp[\sigma^n_h]{\xi}_Q$ in \eqref{eq:discr3}.  The localization of
the basis function interaction around vertices due to the vertex
quadrature rule implies that $\Asg$ is block-diagonal with $d(d-1)/2
\times dk$ blocks. After the stress elimination, the rotation matrix
is $\Ag\As^{-1}\Ag^T$. Since $\As$ is block-diagonal with $dk\times
dk$ blocks, then $\Ag\As^{-1}\Ag^T$ is block-diagonal with $d(d-1)/2
\times d(d-1)/2$ blocks. In fact, for $d=2$ the matrix is
diagonal. Each block couples the rotation degrees of freedom
associated with the corresponding vertex. The blocks are symmetric and
positive definite due to the inf-sup condition \eqref{inf-sup-elast}
and the positive definiteness of $\As^{-1}$. Therefore the rotation
can be easily eliminated, resulting in a cell-centered
displacement-pressure system. The above procedure can be expressed in
matrix form as follows, where $\s$ is the algebraic vector corresponding
to $\s_h^n$, etc.:
\begin{figure}
	\centering
	\begin{subfigure}[b]{0.45\textwidth}
          \centerline{\resizebox{4.5cm}{!}{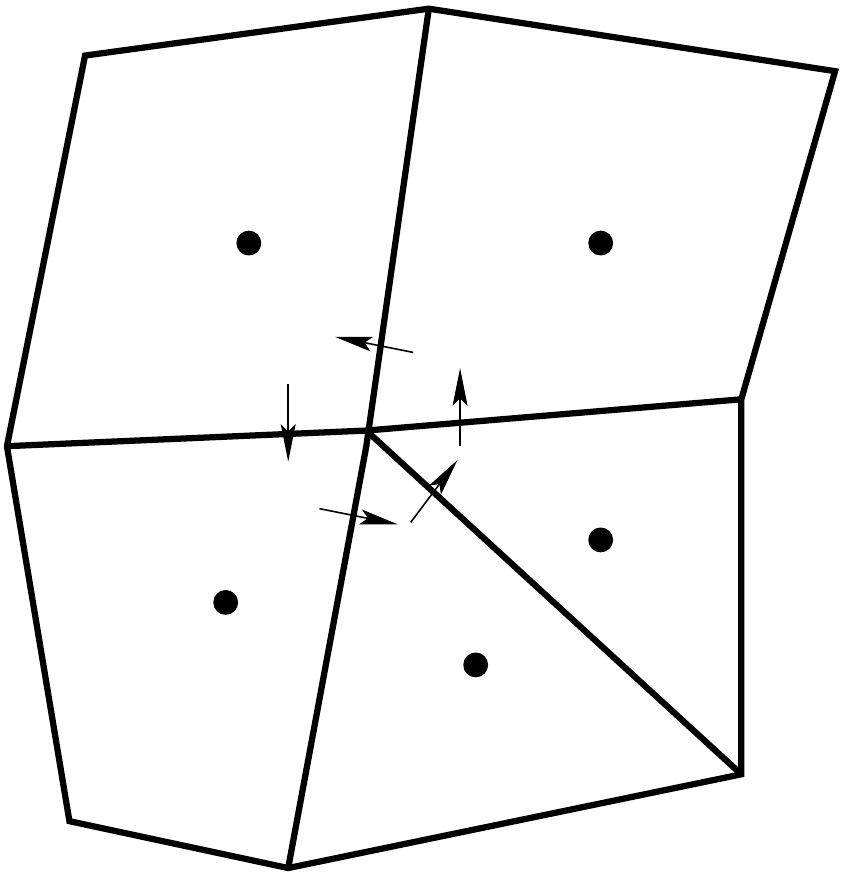}}

          \bigskip
          \smallskip
		\caption{Darcy degrees of freedom}
		\label{fig:0_1}
	\end{subfigure}
	\begin{subfigure}[b]{0.45\textwidth}
		\centerline{\includegraphics[width=.8\textwidth]{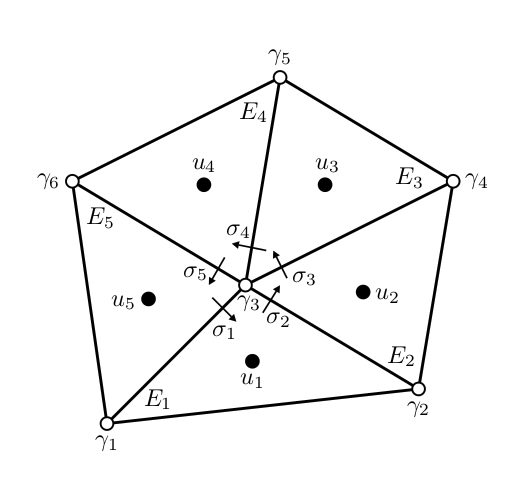}}
		\caption{Elasticity degrees of freedom}
		\label{fig:0_2}
	\end{subfigure}
	\caption{Interactions of the degrees of freedom in the MSMFE--MFMFE method.}\label{fig:0}
\end{figure}
\begin{align}
	&\hspace{-14em}\begin{pmatrix}
		\As   & \Asu^T & \Asg^T & 0     & \Asp^T \\
		-\Asu & 0      & 0      & 0     & 0 \\
		-\Asg & 0      & 0      & 0     & 0 \\
		0     & 0      & 0      & \Az   & \Azp^T \\
		\Asp  & 0      & 0      & -\Azp & \Ap 
	\end{pmatrix}
	\begin{pmatrix}
		\s \\ u \\ \g \\ z \\ p
	\end{pmatrix} \nonumber \\
	\xrightarrow{\s = -\As^{-1}\Asu^Tu - \As^{-1}\Asg^T\g - \As^{-1}\Asp^Tp}&
	\begin{pmatrix}
		\Asu\As^{-1}\Asu^T  & \Asu\As^{-1}\Asg^T  & 0     & \Asu\As^{-1}\Asp^T    \\
		\Asg\As^{-1}\Asu^T  & \Asg\As^{-1}\Asg^T  & 0     & \Asg\As^{-1}\Asp^T    \\
		0                   & 0                   & \Az   & \Azp^T                \\
		-\Asp\As^{-1}\Asu^T & -\Asp\As^{-1}\Asg^T & -\Azp & \Ap-\Asp\As^{-1}\Asp^T 
	\end{pmatrix}
	\begin{pmatrix}
		u \\ \g \\ z \\ p
	\end{pmatrix} \nonumber \\
	\xrightarrow{z = -\Az^{-1}\Azp^Tp}&
	\begin{pmatrix}
		 \Ausu    & \Ausg    & \Ausp    \\
		 \Ausg^T  & \Agsg    & \Agsp    \\
		 -\Ausp^T & -\Agsp^T & \Apszp
	\end{pmatrix}
	\begin{pmatrix}
		 u \\ \g \\ p
	\end{pmatrix} \nonumber \\
	\xrightarrow{\g = -\Agsg^{-1}\Agsp p - \Agsg^{-1}\Ausg^T u}&
	\begin{pmatrix}
		\Ausu - \Ausg\Agsg^{-1}\Ausg^T & \Ausp-\Ausg\Agsg^{-1}\Agsp    \\
		-\Ausp^T + \Agsp^T\Agsg^{-1}\Ausg^T       & \Apszp + \Agsp^T\Agsg^{-1}\Agsp
	\end{pmatrix}
	\begin{pmatrix}
		u \\ p
	\end{pmatrix}, \label{matrix-reduction}
\end{align}
where
\begin{align*}
	&\Ausu := \Asu\As^{-1}\Asu^T, &&\Ausg := \Asu\As^{-1}\Asg^T, \\
	&\Agsg := \Asg\As^{-1}\Asg^T, &&\Ausp := \Asu\As^{-1}\Asp^T, \\
	&\Agsp := \Asg\As^{-1}\Asp^T, &&\Apszp := \Ap-\Asp\As^{-1}\Asp^T + \Azp\Azz^{-1}\Azp^T.
\end{align*}

\begin{remark}
  The expression $z = -\Az^{-1}\Azp^Tp$ above means that the normal
  velocity at each vertex is explicitly expressed in terms of the
  pressures at the centers of the elements that share that vertex, see
  also Figure~\ref{fig:0_1}.  Similarly, $\s = -\As^{-1}\Asu^Tu -
  \As^{-1}\Asg^T\g - \As^{-1}\Asp^Tp$ means that the normal stress at
  each vertex is expressed in terms of the displacements, rotations,
  and pressures at the centers of the elements that share the
  vertex. These expressions motivate the terms multipoint flux and
  multipoint stress. They are used to recover the velocity and the stress
  after solving for the pressure and the displacement.
  \end{remark}

\begin{proposition}\label{spd}
  The cell-centered displacement-pressure matrix obtained in \eqref{matrix-reduction} is 
  block-skew-symmetric and positive definite.
\end{proposition}
\begin{proof}
  Let us denote the four blocks of the matrix in \eqref{matrix-reduction} by $A_{ij}$,
  $i,j=1,2$. The block-skew-symmetric property follows from
  $$
  - A_{12}^T = - (\Ausp-\Ausg\Agsg^{-1}\Agsp)^T = -\Ausp^T + \Agsp^T\Agsg^{-1}\Ausg^T = A_{21},
  $$
  using that $\Agsg$ is symmetric. Therefore, for any
  $\begin{pmatrix} v^T & w^T \end{pmatrix} \ne 0$, we have
  $$
  \begin{pmatrix} v^T & w^T \end{pmatrix}
  \begin{pmatrix} A_{11} & A_{12} \\ A_{21} & A_{22} \end{pmatrix}
  \begin{pmatrix} v \\ w \end{pmatrix} = v^T A_{11} v + w^T A_{22} w,
  $$
  so we need to show that the diagonal blocks are
  positive definite. For $A_{11}$ we have
\begin{align*}
  A_{11} = \Ausu - \Ausg\Agsg^{-1}\Ausg^T = \Asu\As^{-1}\Asu^T
  - \Asu\As^{-1}\Asg^T (\Asg\As^{-1} \Asg^T)^{-1} \Asg\As^{-1}\Asu^T,
\end{align*}
which is a Schur complement of the displacement-rotation matrix
$$
\begin{pmatrix}\Asu\As^{-1}\Asu^T  & \Asu\As^{-1}\Asg^T \\
  \Asg\As^{-1}\Asu^T  & \Asg\As^{-1}\Asg^T \end{pmatrix}.
$$
The latter is symmetric and positive definite,
since for any $\begin{pmatrix} v^T & \xi^T \end{pmatrix} \neq 0$, 
\begin{align*}
    \begin{pmatrix} v^T & \xi^T \end{pmatrix}
    \begin{pmatrix} \Au\As^{-1}\Au^T & \Au\As^{-1}\Ag^T 
    \\ \Ag\As^{-1}\Au^T & \Ag\As^{-1}\Ag^T \end{pmatrix}
    \begin{pmatrix} v \\ \xi \end{pmatrix} 
    = (\Au^Tv + \Ag^T\xi)^T\As^{-1}(\Au^Tv + \Ag^T\xi) > 0,
\end{align*}
due to the positive definiteness of $\As$ and the elasticity inf-sup condition
\eqref{inf-sup-elast}.  Then $A_{11}$ is also symmetric and positive
definite, using \cite[Theorem 7.7.6]{Horn-Johnson}. For $A_{22}$ we have
$$
A_{22} = \Ap-\Asp\As^{-1}\Asp^T + \Azp\Azz^{-1}\Azp^T + \Agsp^T\Agsg^{-1}\Agsp.
$$
The matrix $\Ap-\Asp\As^{-1}\Asp^T$ is positive semidefinite, using
\cite[Theorem 7.7.6]{Horn-Johnson}, since it is a Schur complement of the matrix
$$
A^{\s p} := \begin{pmatrix} \As & \Asp ^T \\ \Asp & \Ap \end{pmatrix},
$$
which is positive semidefinite, since
$(\tau^T \, w^T) \, A^{\s p} \, (\tau \ w)^T = \|A^{1/2}(\tau_h + \a w_h I)\|_Q^2$.
The middle matrix $\Azp\Azz^{-1}\Azp^T$ is positive definite, using that $\Azz$ is
positive definite and the Darcy inf-sup condition \eqref{inf-sup-darcy}.
Finally, the matrix $\Agsp^T\Agsg^{-1}\Agsp$ is positive semidefinite, since
$\Agsg$ is positive definite. Combined, the three properties imply that
$A_{22}$ is symmetric and positive definite. 
\end{proof}

\begin{remark}
The positive-definiteness of the matrix in \eqref{matrix-reduction}
established in Proposition~\ref{spd} allows for an efficient Krylov
space iterative solver like GMRES to be used for the solution of the
reduced displacement-pressure system. Moreover, since the diagonal blocks
are symmetric and positive definite, the block-diagonal part of the matrix
provides an efficient preconditioner. 
\end{remark}

\section{Numerical results}
The proposed fully discrete MSMFE--MFMFE method has been implemented on
simplicial grids using the FEniCS Project \cite{LoggMardalEtAl2012a}
and on quadrilaterals using the deal.II finite element library
\cite{dealii}. In this section we provide several numerical tests
verifying the theoretical convergence rates and illustrating the
behavior of the method. We also present an example showing the
locking-free property of the method in the case of a small storativity
coefficient.

\subsection{Example 1}
We first verify the convergence of the method on simplicial grids in three
dimensions. We use the unit cube as a computational domain
and choose the analytical solution for pressure and displacement as
follows:
\begin{align*}
	p = \cos(t)(x+y+z+1.5), \quad u = \sin(t)\begin{pmatrix} -0.1(e^x - 1)\sin(\pi x)\sin(\pi y) \\ 
	-(e^x - 1)(y - \cos(\frac{\pi}{12})(y-0.5) + \sin(\frac{\pi}{12})(z-0.5)-0.5) \\
	-(e^x - 1)(z - \sin(\frac{\pi}{12})(y-0.5) - \cos(\frac{\pi}{12})(z-0.5)-0.5)
	\end{pmatrix}.
\end{align*}
The permeability tensor is of the form
\begin{align*}
	K = \begin{pmatrix} x^2+y^2+1 & 0 	      & 0 \\ 
						0 		  &	z^2+1  	  & \sin(xy)\\
						0		  & \sin(xy)  & x^2y^2+1 \end{pmatrix},
\end{align*}
and the rest of the parameters are presented in Table \ref{T0}.
\begin{table}[ht!]
	\begin{center}
		\begin{tabular}{l | r r}
			\hline
			Parameter                          & Symbol           & Values                        \\ 
			\hline
			Lame coefficient                   & $\mu$             & $100.0$ \\
			Lame coefficient                   & $\lambda$          & $100.0$ \\
			Mass storativity                   & $c_0$         & $1.0$          \\
			Biot-Willis constant               & $\alpha$                 & 1.0                             \\
			Total time                         & T                     & $10^{-3}$                           \\ 
			Time step                          & $\Delta t$           & $10^{-4}$ \\ \hline
		\end{tabular}
	\end{center}
	\caption{Parameters for Examples 1.}
	\label{T0}
\end{table}

Using the analytical solution provided above and equations
\eqref{biot-1}--\eqref{biot-2}, we obtain the rest of variables and
the right-hand side functions. Dirichlet boundary conditions for the
pressure and the displacement are specified on the entire boundary of
the domain.

\begin{table}
	\begin{center}
		\begin{tabular}{c|cc|cc|cc}
			\hline
			& \multicolumn{2}{c|}{$\|\sigma - \sigma_h\|_{L^2(0,T;L^2(\O))} $} & \multicolumn{2}{c|}{$ \|\dvr(\sigma - \sigma_h)\|_{L^2(0,T;L^2(\O))} $} & \multicolumn{2}{c}{$ \|u - u_h\|_{L^2(0,T;L^2(\O))}$}  \\ 
			$h$	&	error	&	rate	&	error	&	rate	&	error	&	rate	
			\\ \hline
1/4	&	1.55E-02	&	--	&	2.29E-01	&	--	&	8.43E-01	&	--	\\
1/8	&	4.97E-03	&	1.6	&	1.14E-01	&	1.0	&	2.30E-01	&	1.0	\\
1/16	&	2.16E-03	&	1.2	&	5.65E-02	&	1.0	&	8.85E-02	&	1.0	\\
1/32	&	1.03E-03	&	1.1	&	2.82E-02	&	1.0	&	4.11E-02	&	1.0	\\ \hline
			& \multicolumn{2}{c|}{$\|\g - \g_h\|_{L^2(0,T;L^2(\O))} $} & \multicolumn{2}{c|}{$ \|z - z_h\|_{L^2(0,T;L^2(\O))}$} & \multicolumn{2}{c}{$\|\dvr(z-z_h)\|_{L^2(0,T;L^2(\O))}$} \\
			$h$	&	error	&	rate	&	error	&	rate	&	error	&	rate	 \\ \hline
1/4	&	7.65E-01	&	--	&	4.34E-04	&	--	&	5.85E-02	&	--	\\
1/8	&	2.32E-01	&	1.7	&	2.26E-04	&	0.9	&	2.31E-02	&	1.3	\\
1/16	&	7.04E-02	&	1.7	&	1.14E-04	&	1.0	&	1.05E-02	&	1.1	\\
1/32	&	2.13E-02	&	1.7	&	5.68E-05	&	1.0	&	5.00E-03	&	1.1	\\ \hline
			& \multicolumn{2}{c|}{$\|p - p_h\|_{L^2(0,T;L^2(\O))} $} & \multicolumn{2}{c|}{$ \|\s - \s_h\|_{L^{\infty}(0,T;L^2(\O))}$} & \multicolumn{2}{c}{$\|u-u_h\|_{L^{\infty}(0,T;L^2(\O))}$} \\
			$h$	&	error	&	rate	&	error	&	rate	&	error	&	rate	 \\ \hline
1/4	&	2.58E-01	&	--	&	2.29E-01	&	--	&	2.55E+00	&	--	\\
1/8	&	1.26E-01	&	1.0	&	1.14E-01	&	1.0	&	7.12E-01	&	1.8	\\
1/16	&	6.18E-02	&	1.0	&	5.67E-02	&	1.0	&	2.91E-01	&	1.3	\\
1/32	&	3.09E-02	&	1.0	&	2.82E-02	&	1.0	&	1.38E-01	&	1.1	\\ \hline
			& \multicolumn{2}{c|}{$\|\g - \g_h\|_{L^{\infty}(0,T;L^2(\O))} $} & \multicolumn{2}{c|}{$ \|z - z_h\|_{L^{\infty}(0,T;L^2(\O))}$} & \multicolumn{2}{c}{$\|p-p_h\|_{L^{\infty}(0,T;L^2(\O))}$} \\
			$h$	&	error	&	rate	&	error	&	rate	&	error	&	rate	 \\ \hline
1/4	&	2.35E+00	&	--	&	4.78E-04	&	--	&	2.58E-01	&	--	\\
1/8	&	7.06E-01	&	1.7	&	2.57E-04	&	0.9	&	1.26E-01	&	1.0	\\
1/16	&	2.12E-01	&	1.7	&	1.33E-04	&	0.9	&	6.21E-02	&	1.0	\\
1/32	&	6.37E-02	&	1.7	&	6.69E-05	&	1.0	&	3.09E-02	&	1.0	\\ \hline
		\end{tabular}
	\end{center}
	\caption{Example 1, numerical errors and convergence rates.}
	\label{T1}
\end{table}

\begin{figure}
	\centering
	\begin{subfigure}[b]{0.23\textwidth}
		\includegraphics[width=\textwidth]{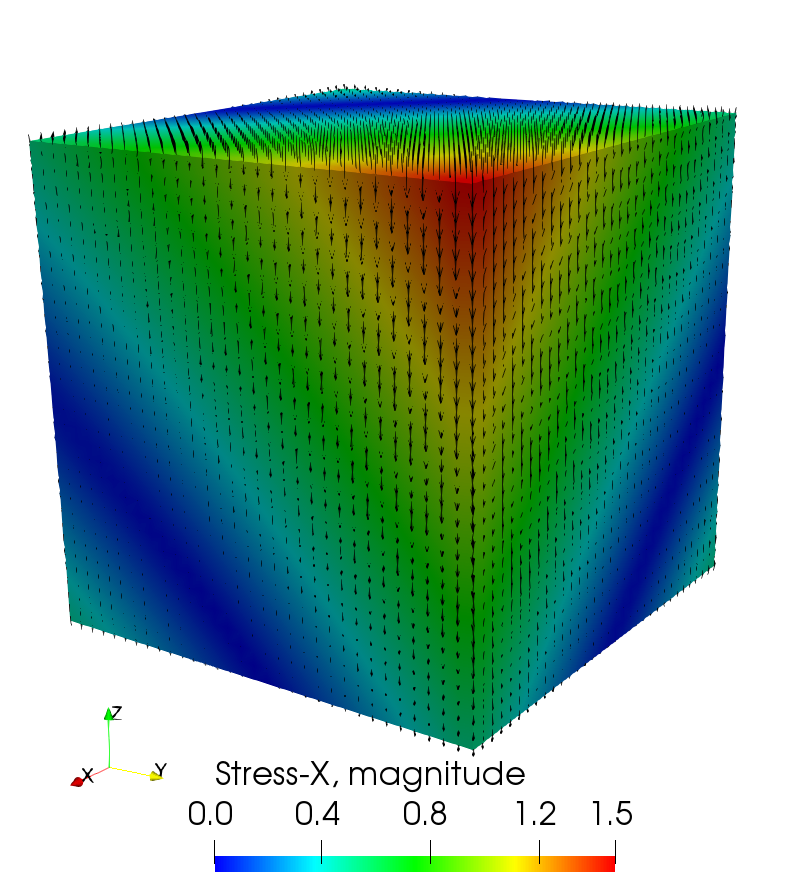}
		\caption{Stress, $x$-component}
		\label{fig:1_1}
	\end{subfigure}
	\begin{subfigure}[b]{0.23\textwidth}
		\includegraphics[width=\textwidth]{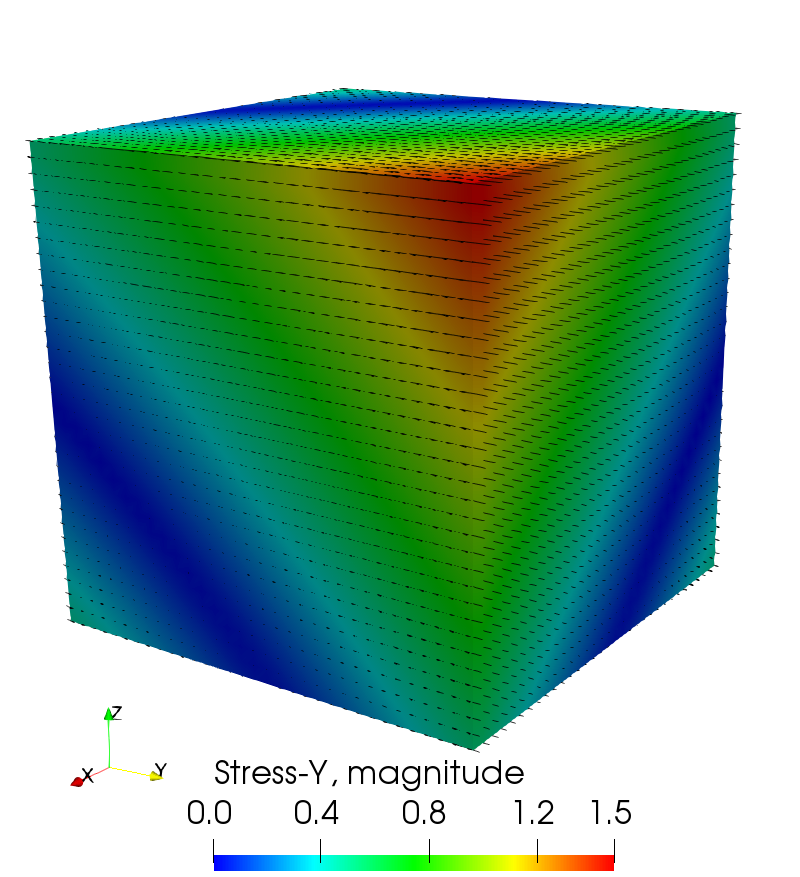}
		\caption{Stress, $y$-component}
		\label{fig:1_2}
	\end{subfigure}
	\begin{subfigure}[b]{0.23\textwidth}
		\includegraphics[width=\textwidth]{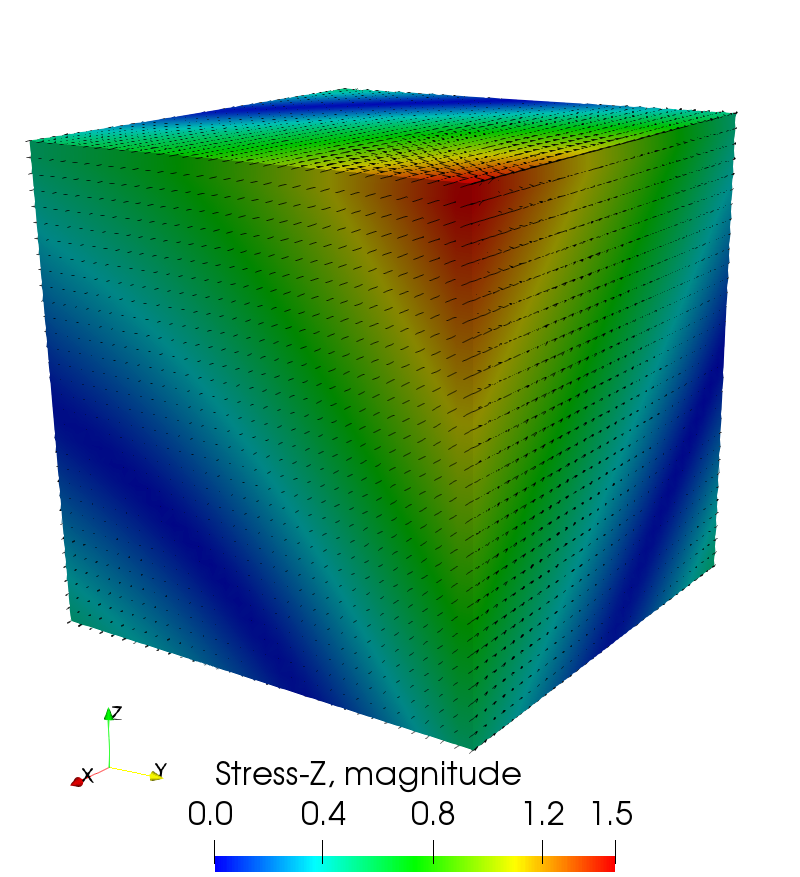}
		\caption{Stress, $z$-component}
		\label{fig:1_3}
	\end{subfigure}
	\begin{subfigure}[b]{0.23\textwidth}
		\includegraphics[width=\textwidth]{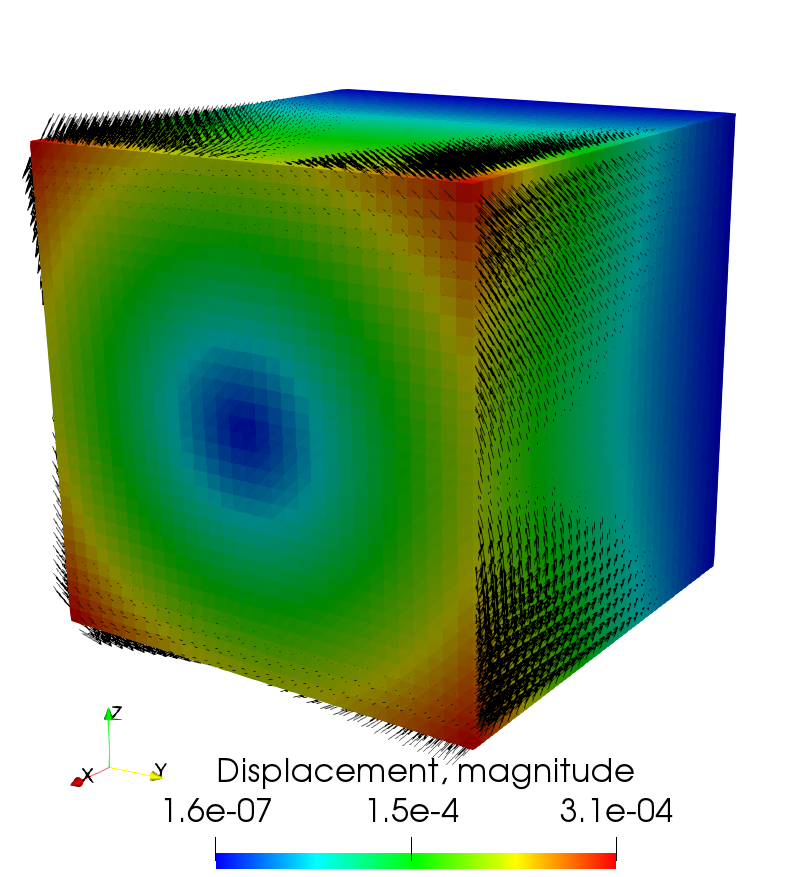}
		\caption{Displacement}
		\label{fig:1_4}
	\end{subfigure}
	\begin{subfigure}[b]{0.23\textwidth}
		\includegraphics[width=\textwidth]{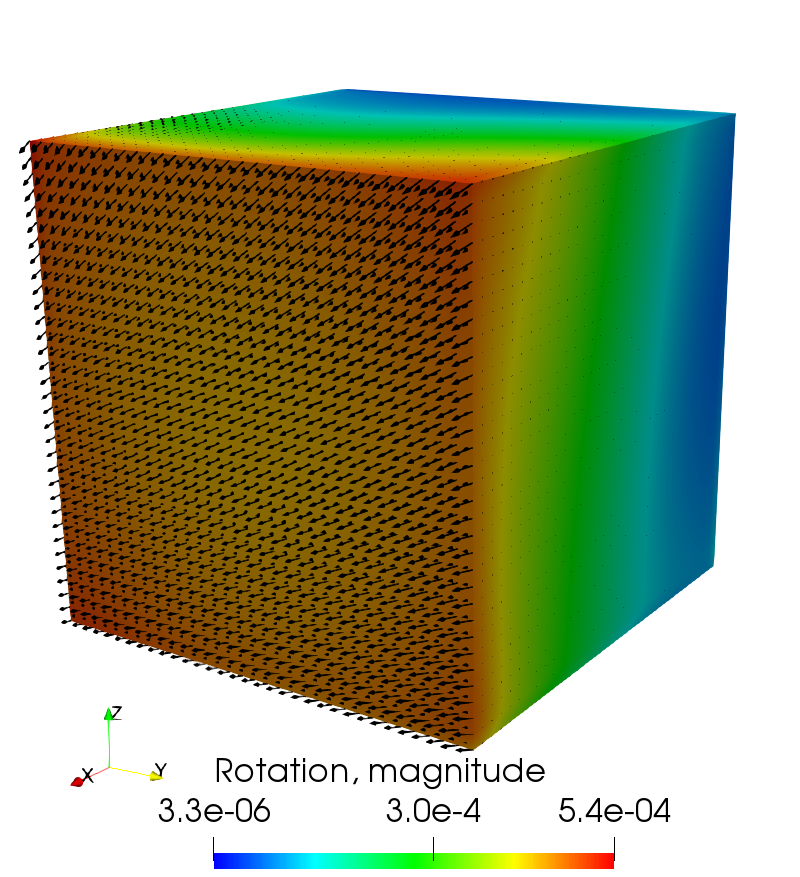}
		\caption{Rotation}
		\label{fig:1_5}
	\end{subfigure}
	\begin{subfigure}[b]{0.23\textwidth}
		\includegraphics[width=\textwidth]{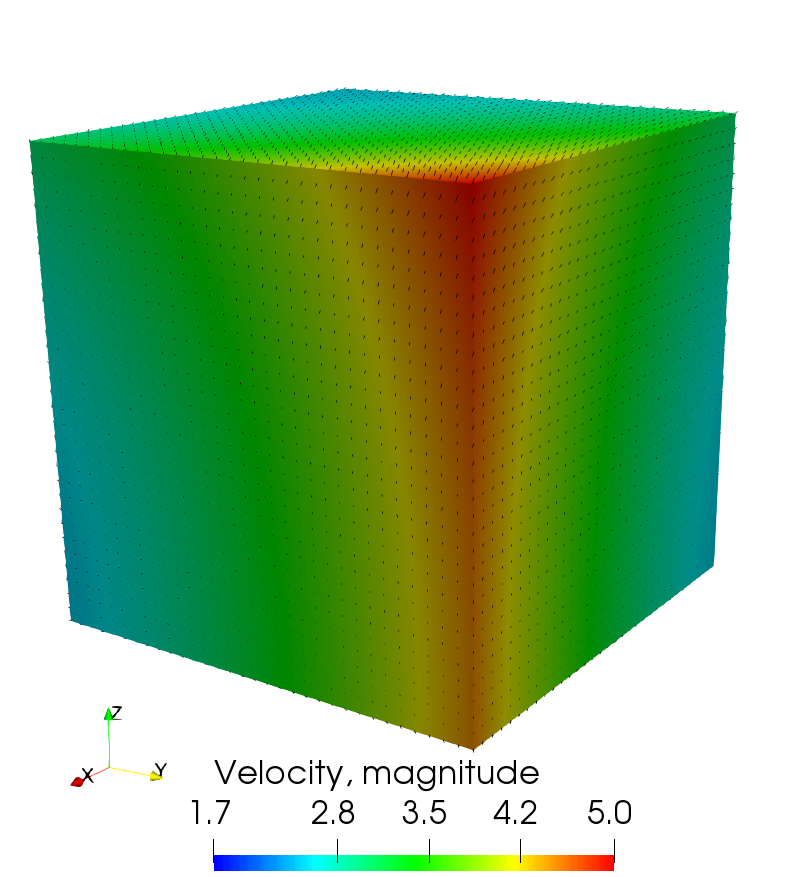}
		\caption{Darcy velocity}
		\label{fig:1_6}
	\end{subfigure}
	\begin{subfigure}[b]{0.23\textwidth}
		\includegraphics[width=\textwidth]{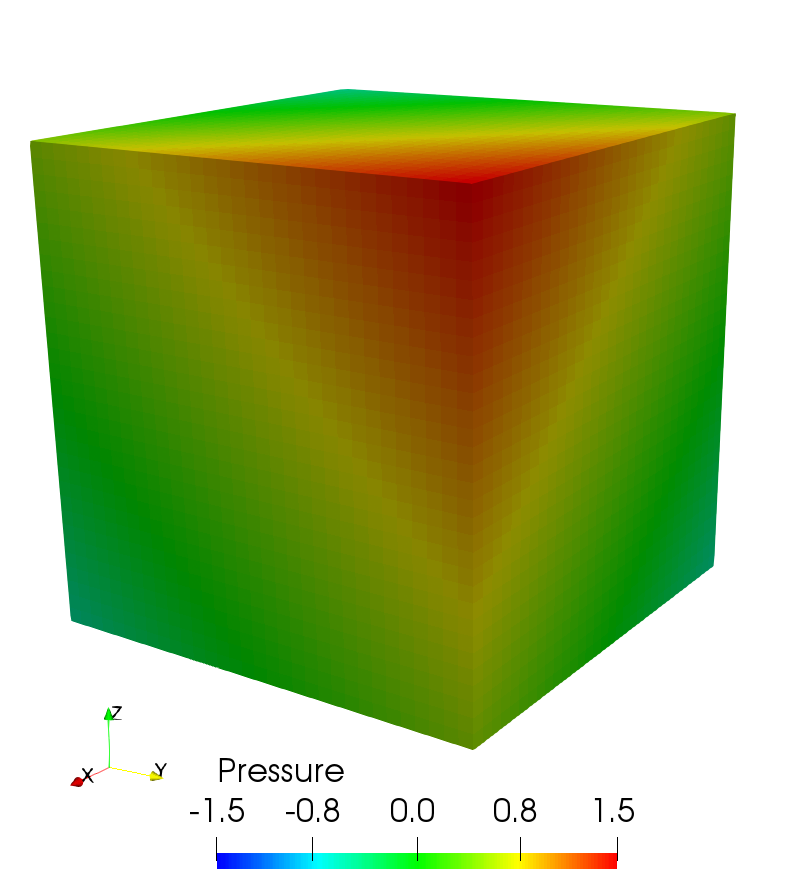}
		\caption{Darcy pressure}
		\label{fig:1_7}
	\end{subfigure}
	\caption{Example 1, computed solution with $h=\frac{1}{32}$ at the final time.}\label{fig:1}
\end{figure}

In Table \ref{T1} we present the relative errors and spatial
convergence rates on a sequence of mesh refinements. We take a
sufficiently small time step $\Delta t = 10^{-4}$ to ensure that the
time discretization error does not dominate. We observe at least first
order of convergence in all norms, as predicted by the theory. The
error $\|\gamma - \gamma_h\|$ exhibits convergence of order higher
than one, which can be attributed to the linear polynomial
approximation.
The numerical solution on the finest level 
at the final time is shown in Figure~\ref{fig:1}.

\subsection{Example 2}
In the second test case we study the convergence of the method on $h^2$-parallelogram grids. 
We consider the analytical solution
\begin{align*}
p=\exp(t)(\sin(\pi x)\cos(\pi y) + 10), \quad u= \exp(t) \begin{pmatrix}
											x^3y^4 + x^2 + \sin((1-x)(1-y))\cos(1-y)\\
											(1-x)^4(1-y)^3 + (1-y)^2 + \cos(xy)\sin(x)	\end{pmatrix},
\end{align*}
and the permeability tensor
$$ 
K = \begin{pmatrix}
(x+1)^2 + y^2 & \sin(xy) \\ \sin(xy) & (x+1)^2	\end{pmatrix}. 
$$
In this example as elasticity parameters we use the Poisson ratio
$\nu$ and the Young's modulus $E$. We set $\nu=0.2$ and take $E$ to
vary over the domain, $E = \sin(5\pi x)\sin(5\pi y) + 5$. The Lam\'{e}
parameters are then computed using the well known relations
$$ 
\lambda = \frac{E \nu}{(1+\nu)(1-2\nu)}, \quad \mu = \frac{E}{2(1+\nu)}. 
$$
In this test case we also illustrate the behavior of the method for small 
mass storativity and set $c_0 = 10^{-5}$. The Biot-Willis constant $\alpha$ and 
the time discretization parameters are the same as in Table \ref{T0}.

The computational domain for this case is obtained as follows. We start with the
unit square and partition it into a $4\times4$ square mesh with $h=\frac14$. We then
move the mesh points using the map
\begin{align*}
x = \xh +0.03\cos(3\pi \xh)\cos(3\pi\yh), \quad y=\yh -0.04\cos(3\pi\xh)\cos(3\pi\yh),
\end{align*}
which gives a deformed computational domain with a $4\times4$ quadrilateral grid, see
Figure \ref{fig:2}. A sequence of mesh refinements is obtained by a uniform refinement
of the elements of the coarse grid. The resulting sequence of meshes satisfies the
$h^2$-parallelogram property \eqref{h2-parall}.

As in the previous test case, we observe at least first order
convergence for all variables in their respective norms, see
Table~\ref{T2}. The computed solution with $h=\frac{1}{32}$ at the
final time is shown in Figure~\ref{fig:2}. This example not only
confirms the theoretical convergence rates on $h^2$-parallelogram grids, but also
illustrates that the method can handle well variable elasticity parameters and small
mass storativity.

\begin{table}
	\begin{center}
		\begin{tabular}{c|cc|cc|cc}
			\hline
			& \multicolumn{2}{c|}{$\|\sigma - \sigma_h\|_{L^2(0,T;L^2(\O))} $} & \multicolumn{2}{c|}{$ \|\dvr(\sigma - \sigma_h)\|_{L^2(0,T;L^2(\O))} $} & \multicolumn{2}{c}{$ \|u - u_h\|_{L^2(0,T;L^2(\O))}$}  \\ 
			$h$	&	error	&	rate	&	error	&	rate	&	error	&	rate	
			\\ \hline
1/8	&	9.65E-02	&	--	&	1.30E-01	&	--	&	8.02E-02	&	--	\\
1/16	&	4.97E-02	&	1.0	&	6.46E-02	&	1.0	&	3.97E-02	&	1.0	\\
1/32	&	2.52E-02	&	1.0	&	3.23E-02	&	1.0	&	1.98E-02	&	1.0	\\
1/64	&	1.27E-02	&	1.0	&	1.61E-02	&	1.0	&	9.87E-03	&	1.0	\\
1/128	&	6.35E-03	&	1.0	&	8.07E-03	&	1.0	&	4.93E-03	&	1.0	\\ \hline
			& \multicolumn{2}{c|}{$\|\g - \g_h\|_{L^2(0,T;L^2(\O))} $} & \multicolumn{2}{c|}{$ \|z - z_h\|_{L^2(0,T;L^2(\O))}$} & \multicolumn{2}{c}{$\|\dvr(z-z_h)\|_{L^2(0,T;L^2(\O))}$} \\
			$h$	&	error	&	rate	&	error	&	rate	&	error	&	rate	 \\ \hline
1/8	&	2.03E-01	&	--	&	1.44E-01	&	--	&	2.88E-01	&	--	\\
1/16	&	7.51E-02	&	1.4	&	7.05E-02	&	1.0	&	1.75E-01	&	0.7	\\
1/32	&	2.77E-02	&	1.4	&	3.47E-02	&	1.0	&	8.18E-02	&	1.1	\\
1/64	&	1.02E-02	&	1.5	&	1.72E-02	&	1.0	&	3.35E-02	&	1.3	\\
1/128	&	3.70E-03	&	1.5	&	8.60E-03	&	1.0	&	1.39E-02	&	1.3	\\ \hline
			& \multicolumn{2}{c|}{$\|p - p_h\|_{L^2(0,T;L^2(\O))} $} & \multicolumn{2}{c|}{$ \|\s - \s_h\|_{L^{\infty}(0,T;L^2(\O))}$} & \multicolumn{2}{c}{$\|u-u_h\|_{L^{\infty}(0,T;L^2(\O))}$} \\
			$h$	&	error	&	rate	&	error	&	rate	&	error	&	rate	 \\ \hline
1/8	&	8.97E-03	&	--	&	9.65E-02	&	--	&	8.02E-02	&	--	\\
1/16	&	4.49E-03	&	1.0	&	4.97E-02	&	1.0	&	3.97E-02	&	1.0	\\
1/32	&	2.24E-03	&	1.0	&	2.52E-02	&	1.0	&	1.98E-02	&	1.0	\\
1/64	&	1.12E-03	&	1.0	&	1.27E-02	&	1.0	&	9.87E-03	&	1.0	\\
1/128	&	5.61E-04	&	1.0	&	6.35E-03	&	1.0	&	4.93E-03	&	1.0	\\ \hline
			& \multicolumn{2}{c|}{$\|\g - \g_h\|_{L^{\infty}(0,T;L^2(\O))} $} & \multicolumn{2}{c|}{$ \|z - z_h\|_{L^{\infty}(0,T;L^2(\O))}$} & \multicolumn{2}{c}{$\|p-p_h\|_{L^{\infty}(0,T;L^2(\O))}$} \\
			$h$	&	error	&	rate	&	error	&	rate	&	error	&	rate	 \\ \hline
1/8	&	2.03E-01	&	--	&	1.60E-01	&	--	&	9.03E-03	&	--	\\
1/16	&	7.51E-02	&	1.4	&	8.07E-02	&	1.0	&	4.50E-03	&	1.0	\\
1/32	&	2.77E-02	&	1.4	&	3.69E-02	&	1.1	&	2.24E-03	&	1.0	\\
1/64	&	1.02E-02	&	1.5	&	1.75E-02	&	1.1	&	1.12E-03	&	1.0	\\
1/128	&	3.70E-03	&	1.5	&	8.64E-03	&	1.0	&	5.61E-04	&	1.0	\\ \hline
		\end{tabular}
	\end{center}
	\caption{Example 2, numerical errors and convergence rates.}
	\label{T2}
\end{table}

\begin{figure}
	\centering
	\begin{subfigure}[b]{0.23\textwidth}
		\includegraphics[width=\textwidth]{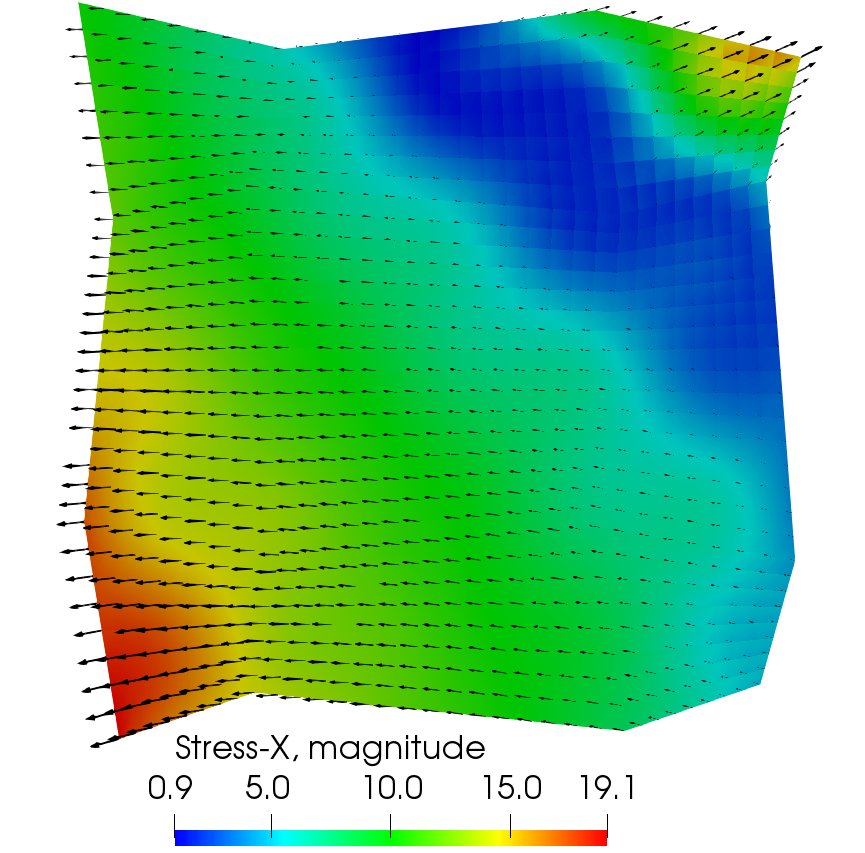}
		\caption{Stress, $x$-component}
		\label{fig:2_1}
	\end{subfigure}
	\begin{subfigure}[b]{0.23\textwidth}
		\includegraphics[width=\textwidth]{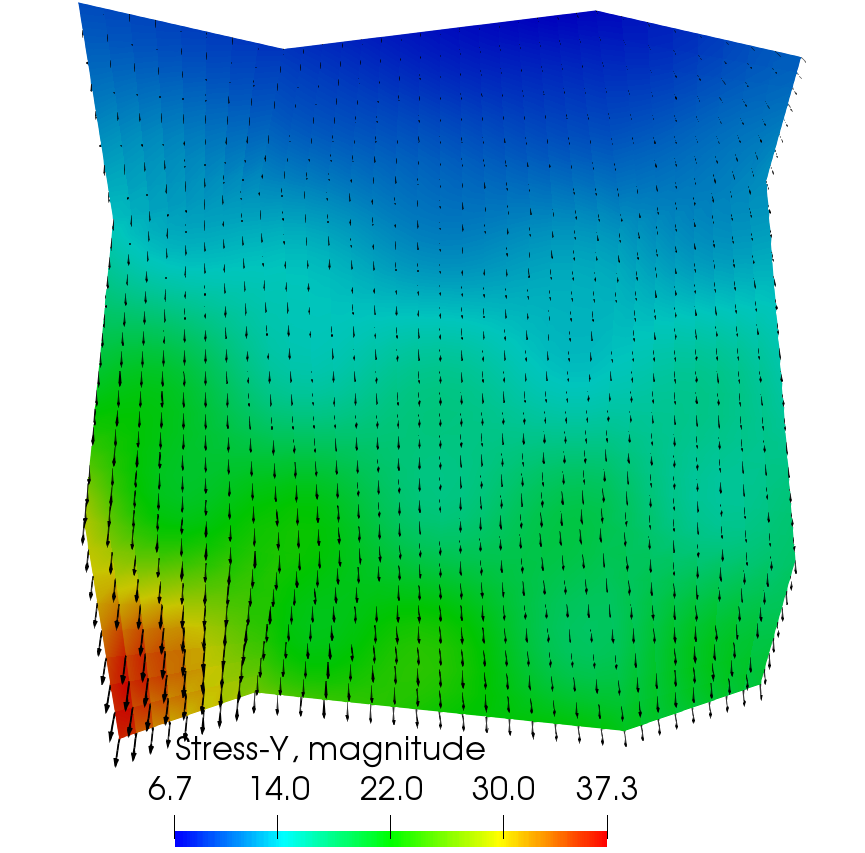}
		\caption{Stress, $y$-component}
		\label{fig:2_2}
	\end{subfigure}
	\begin{subfigure}[b]{0.23\textwidth}
		\includegraphics[width=\textwidth]{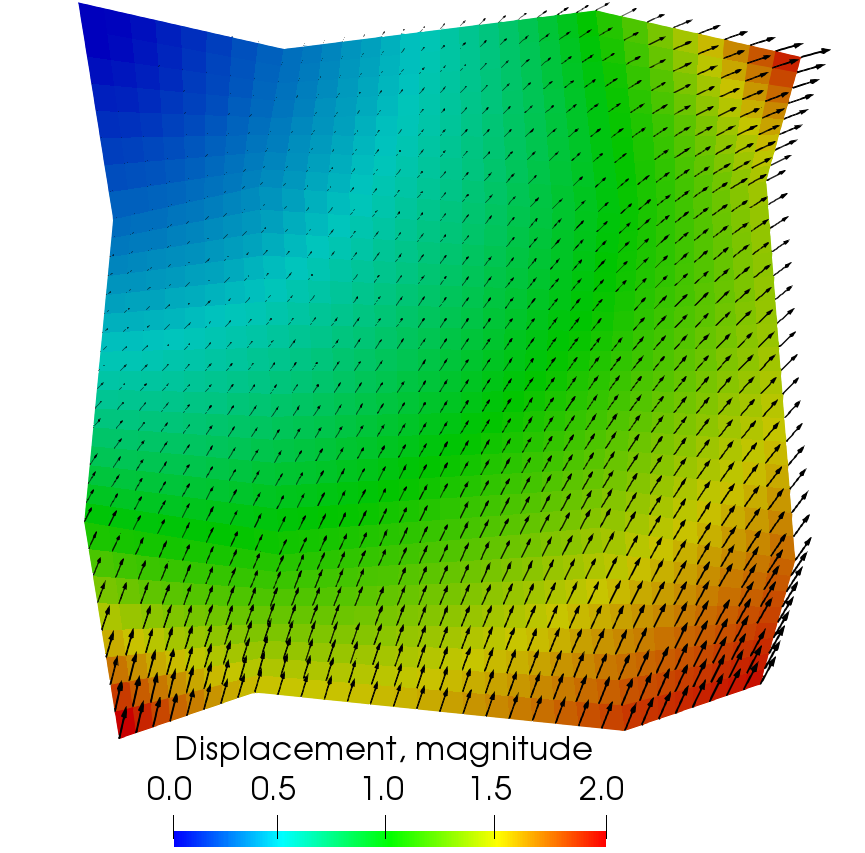}
		\caption{Displacement}
		\label{fig:2_3}
	\end{subfigure}
	\begin{subfigure}[b]{0.23\textwidth}
		\includegraphics[width=\textwidth]{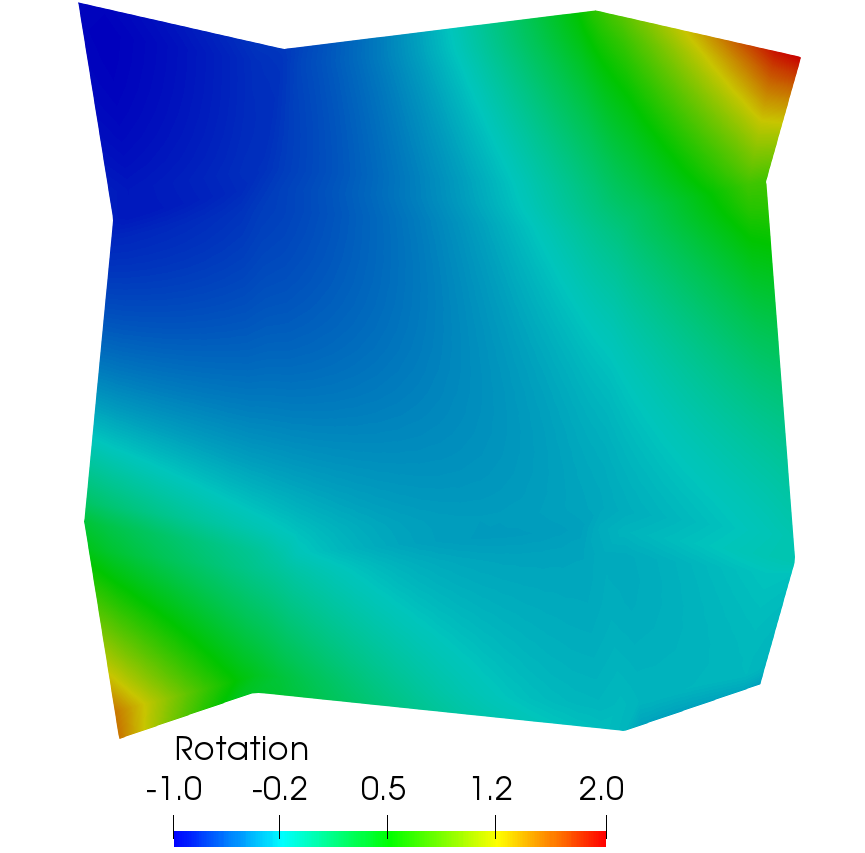}
		\caption{Rotation}
		\label{fig:2_4}
	\end{subfigure}
	\begin{subfigure}[b]{0.23\textwidth}
		\includegraphics[width=\textwidth]{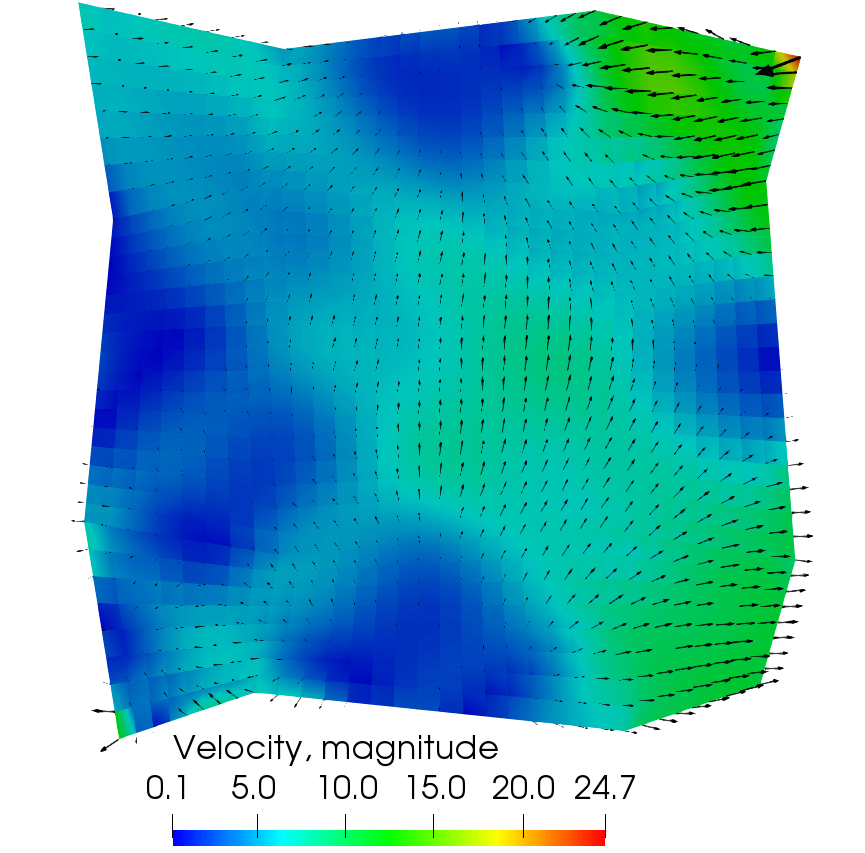}
		\caption{Darcy velocity}
		\label{fig:2_5}
	\end{subfigure}
	\begin{subfigure}[b]{0.23\textwidth}
		\includegraphics[width=\textwidth]{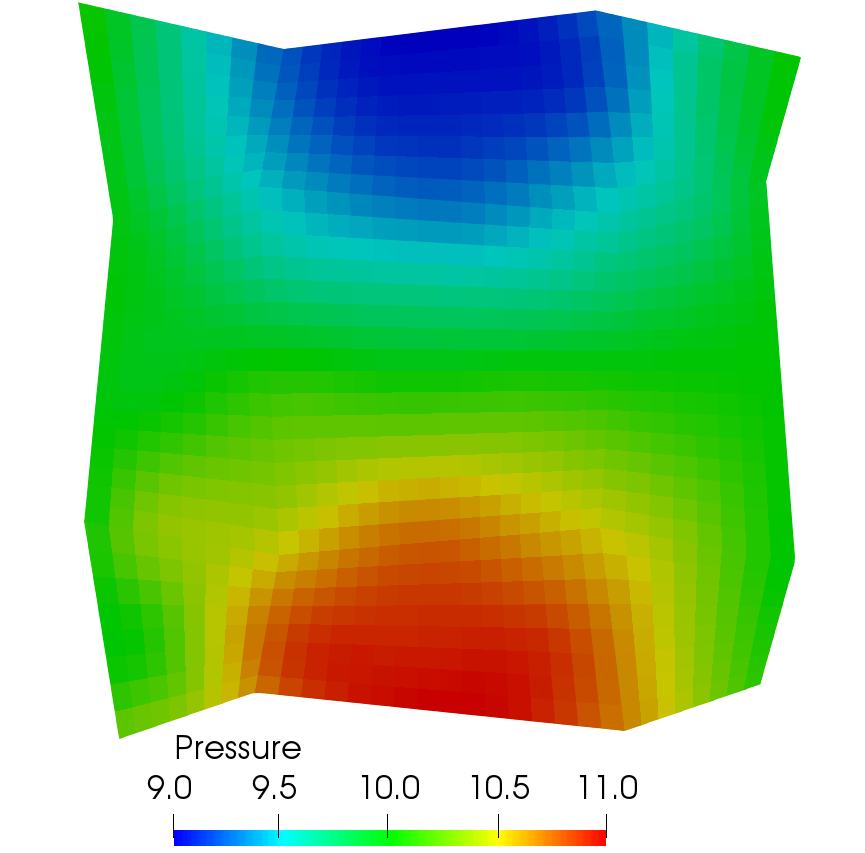}
		\caption{Darcy pressure}
		\label{fig:2_6}
	\end{subfigure}
	\caption{Example 2, computed solution with $h=\frac{1}{32}$ at the final time.}\label{fig:2}
\end{figure}

\subsection{Example 3}
We next focus on studying the locking-free properties of the
MSMFE-MFMFE method when applied to the solution of a two-dimensional
footing problem \cite{ORB,gaspar2008stabilized}. A load of given
intensity $\sigma_0$ is applied along a strip along the top of a
rectangular block of porous, saturated, and deformable soil. The
lateral sides and the bottom of the block are fixed. The entire
boundary is free to drain. The computational domain is $\Omega = [-50, 50] \times
[0, 75]$. We label the middle section of the top boundary, $x\in[-50/3, 50/3]$, $y=75$,
by $\Gamma_1$, the rest of the top side by
$\Gamma_2$, and all other boundaries by $\Gamma_3$. The boundary conditions are as follows:
\begin{align*}
	&\s\,n = (0,-\sigma_0)^T, &&\mbox{on } \Gamma_1,\\
	&\s\,n = (0,0)^T, &&\mbox{on } \Gamma_2, \\
	&u = (0,0)^T, &&\mbox{on } \Gamma_3, \\
	&p = 0, &&\mbox{on } \dO. \\
\end{align*}
The model parameters are: Young's modulus $E=3\cdot 10^4$ (N/m$^2$),
permeability $K=10^{-4}$ (m$^2$/Pa), load intensity $\sigma_0=10^4$
(N/m$^2$) and mass storativity $c_0 = 0.001$. We test the behavior of
the method in the incompressibility limit by setting Poisson ratio
$\nu=0.4995$. The initial pressure and displacement are set to zero.
We discretize the domain into 62025 unstructured simplices and solve
the problem for total time of $T=50$s using time step of size $\Delta t =
1$s.

\begin{figure}
	\centering
	\begin{subfigure}[b]{0.32\textwidth}
		\includegraphics[width=\textwidth]{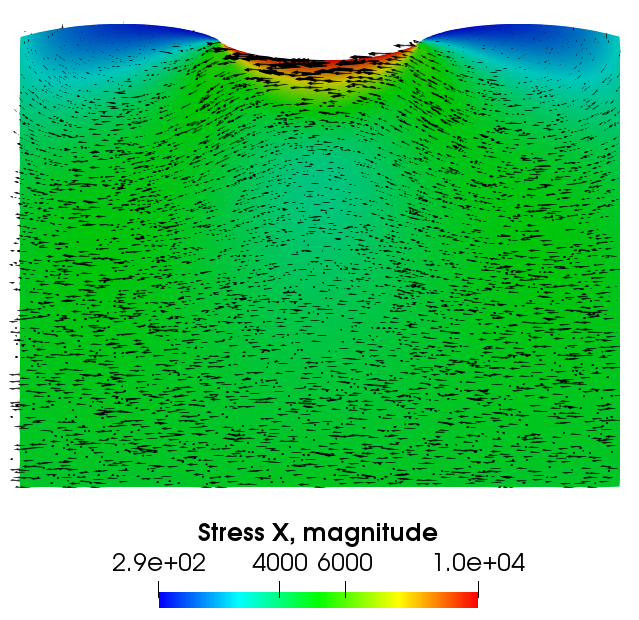}
		\caption{Stress, $x$-component}
		\label{fig:3_1}
	\end{subfigure}
	\begin{subfigure}[b]{0.32\textwidth}
		\includegraphics[width=\textwidth]{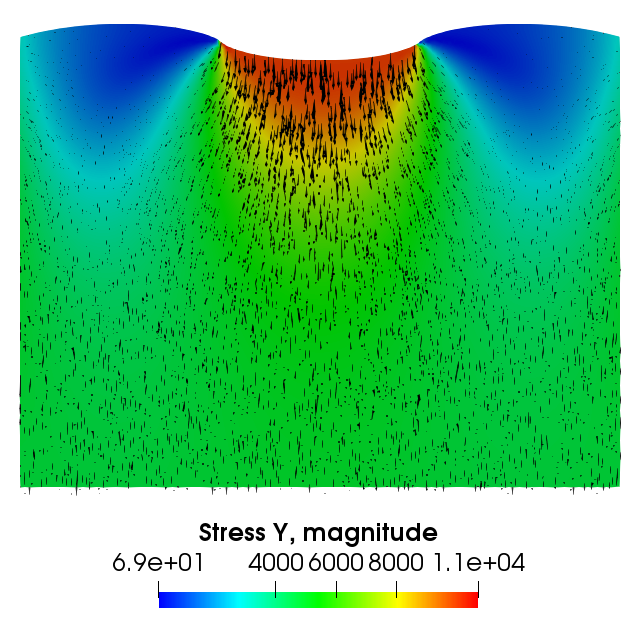}
		\caption{Stress, $y$-component}
		\label{fig:3_2}
	\end{subfigure}
	\begin{subfigure}[b]{0.32\textwidth}
		\includegraphics[width=\textwidth]{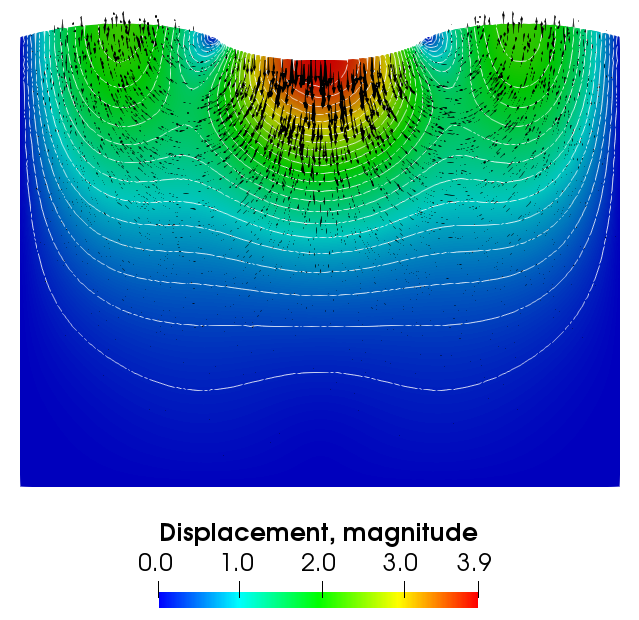}
		\caption{Displacement}
		\label{fig:3_3}
	\end{subfigure}
	\begin{subfigure}[b]{0.32\textwidth}
		\includegraphics[width=\textwidth]{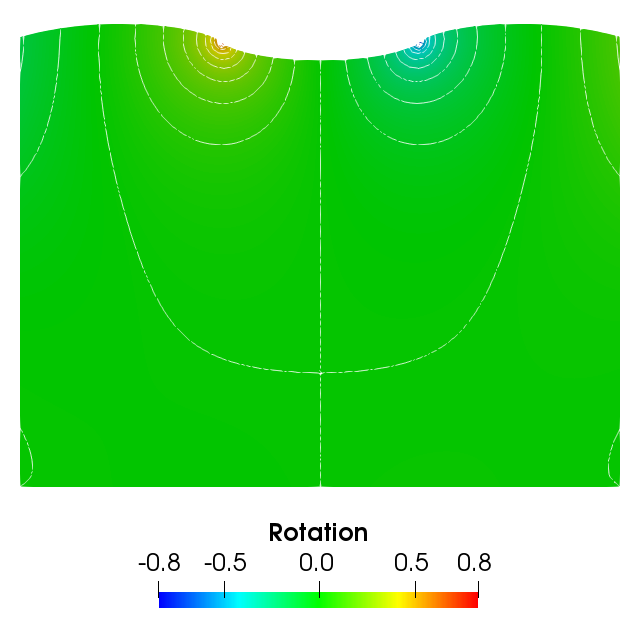}
		\caption{Rotation}
		\label{fig:3_4}
	\end{subfigure}
	\begin{subfigure}[b]{0.32\textwidth}
		\includegraphics[width=\textwidth]{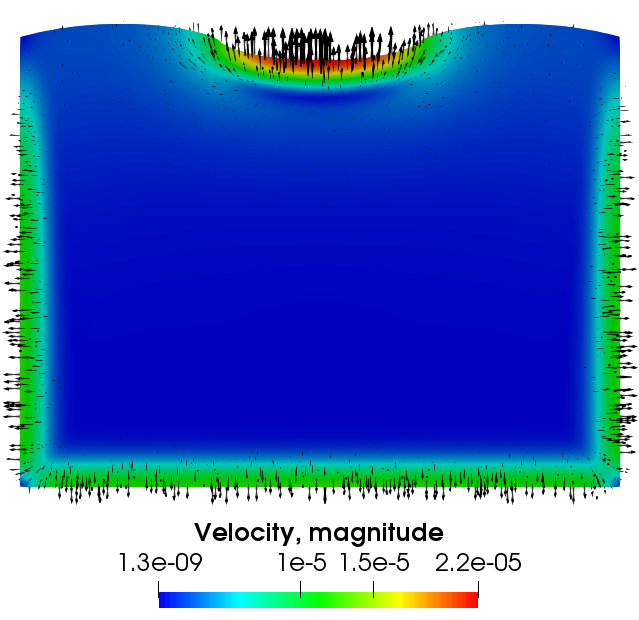}
		\caption{Darcy velocity}
		\label{fig:3_5}
	\end{subfigure}
	\begin{subfigure}[b]{0.32\textwidth}
		\includegraphics[width=\textwidth]{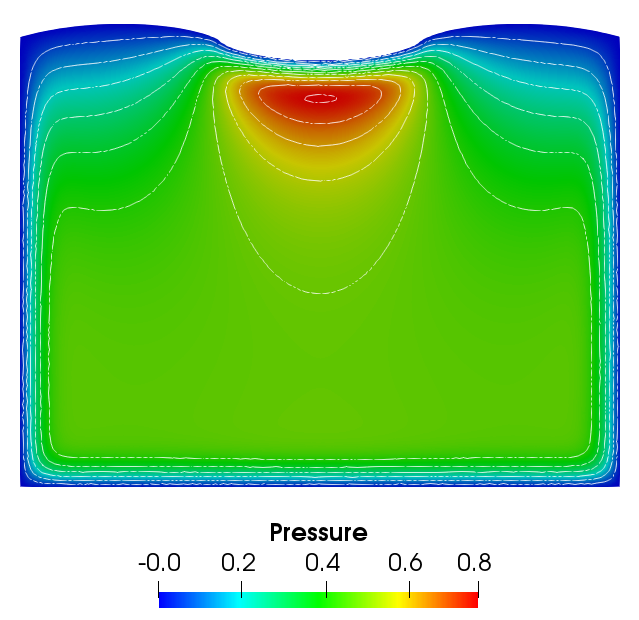}
		\caption{Darcy pressure}
		\label{fig:3_6}
	\end{subfigure}
	\caption{Example 3, computed solution at the final time on the deformed domain.}\label{fig:3}
\end{figure}

It is observed in \cite{ORB,gaspar2008stabilized} that for this value
of the Poisson ratio, inf-sup unstable discretizations may result in
spurious pressure modes and/or locking in the computed displacement.
In Figure \ref{fig:3} we show the solution obtained by MSMFE-MFMFE
method at the final time. For visualization purpose, the solution is
plotted on the deformed domain. Neither spurious oscillations in the
pressure, nor locking effects in the displacement are present,
illustrating that the proposed method inherits the locking-free
properties of the classical mixed method it is derived from. We
further note the smooth stress approximation and the accurate
resolution of the pressure and velocity boundary layers, as well as
the rotation singularities.

\subsection{Example 4}
In the last example we further illustrate the locking-free properties
of the MSMFE--MFMFE method in a different parameter regime.  It is
shown in \cite{phillips2009overcoming} that, with continuous finite
elements for the elasticity part of the system, locking occurs when
the storativity and permeability coefficients are very small. In this
regime, the locking is exhibited as spurious pressure oscillations at
early times.  A typical model problems that illustrates such behavior
is the cantilever bracket problem \cite{Liu-thesis}.
The computational domain is the unit square. We impose a no-flow
boundary condition along all sides. The deformation is fixed along the
left edge, and a downward traction is applied along the top. The
bottom and right sides are traction-free. More precisely, with the
sides of the domain labeled as $\Gamma_1,\dots,\Gamma_4$, starting
from the bottom side and going counterclockwise, we impose
\begin{align*}
	&z\cdot n = 0, &&\mbox{on } \dO = \Gamma_1\cup \Gamma_2\cup\Gamma_3\cup\Gamma_4, \\
	&\s\,n = (0,-1)^T, &&\mbox{on } \Gamma_3,\\
	&\s\,n = (0,0)^T, &&\mbox{on } \Gamma_1 \cup \Gamma_2, \\
	&u = (0,0)^T, &&\mbox{on } \Gamma_4.
\end{align*}
We use the same physical parameters as in \cite{phillips2009overcoming}, 
as they typically induce locking:
$$ 
E = 10^5, \quad \nu=0.4, \quad \alpha = 0.93, \quad c_0 = 0, \quad K = 10^{-7}.
$$
The time step is $\Delta t = 0.001$ and the total simulation time is $T = 1$.
\begin{figure}
	\centering
	\begin{subfigure}[b]{0.25\textwidth}
		\includegraphics[width=\textwidth]{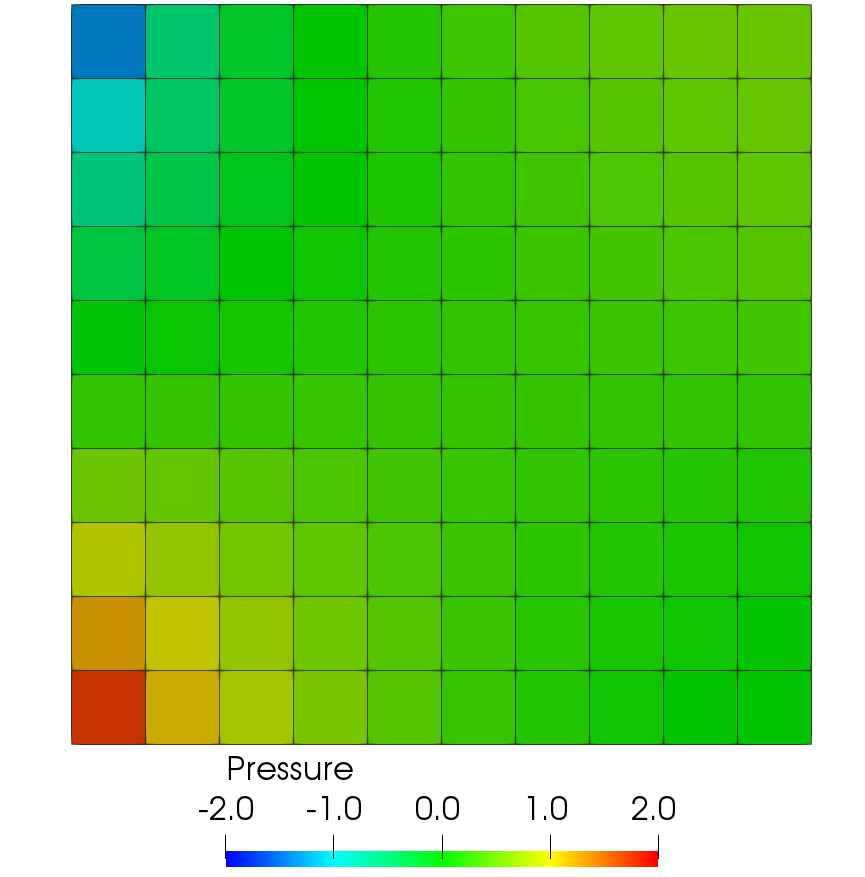}
		\caption{Pressure field, $t=0.001.$}
		\label{fig:4_1}
	\end{subfigure}
	\begin{subfigure}[b]{0.74\textwidth}
		\includegraphics[width=\textwidth]{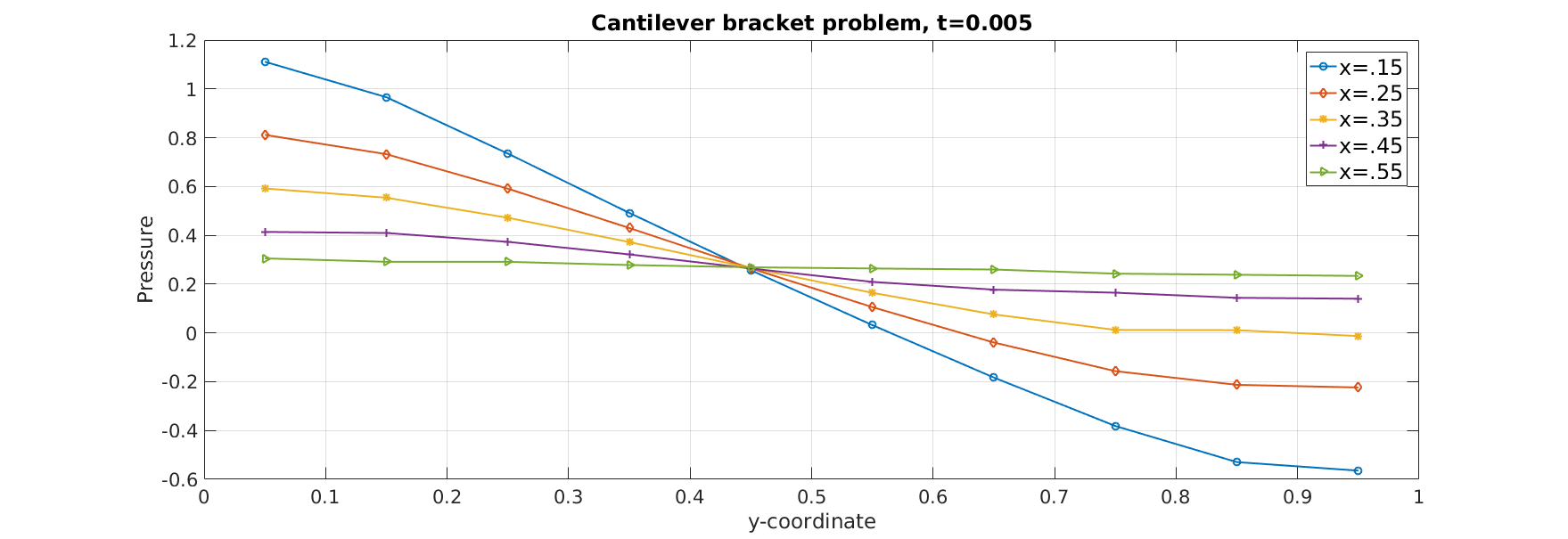}
		\caption{Pressure along different $x-$lines, $t=0.005.$}
		\label{fig:4_2}
	\end{subfigure}
	\caption{Example 3, computed pressure solutions.}\label{fig:4}
\end{figure}

Figure \ref{fig:4_1} shows that the MSMFE--MFMFE method yields a smooth
pressure field, in contrast to the non-physical checkerboard pattern
that one obtains with continuous elasticity elements at the early time
steps, see \cite{phillips2009overcoming}. In addition, Figure
\ref{fig:4_2} shows the pressure solution along different $x-$lines at
time $t=0.005$. It illustrates the lack of oscillations and
shows that our solution agrees with the one
obtained by DG-mixed or stabilized CG-mixed discretizations
\cite{phillips2009overcoming, Liu-thesis}. We remark that our method
requires solving a much smaller algebraic system than these two methods, which
furthermore is positive definite and more efficient to solve.

\bibliographystyle{abbrv}
\bibliography{msfmfe}

\end{document}